\documentclass{amsart}
\usepackage{amsmath,amssymb,amsthm}
\usepackage{color}
\usepackage[mathscr]{euscript}
\usepackage[arrow,curve,frame,color]{xy}
\usepackage{hyperref}

\numberwithin{equation}{section}
\theoremstyle{plain}

\definecolor{orange}{rgb}{1.0,0.3,0}

\newtheorem{theorem}[equation]{Theorem}
\newtheorem{thm}[equation]{Theorem}

\newtheorem{lemma}[equation]{Lemma}
\newtheorem{corollary}[equation]{Corollary}

\theoremstyle{remark}
\newtheorem{remark}[equation]{Remark}

\theoremstyle{definition}
\newtheorem{definition}[equation]{Definition}
\newtheorem{defn}[equation]{Definition}

\newtheorem{question}[equation]{Question}
\newtheorem*{question*}{Question}

\newcommand{\A}{{\mathcal A}}

\newcommand{\C}{{\mathcal C}}

\newcommand{\F}{{\mathcal F}}

\renewcommand{\L}{{\mathcal L}}

\newcommand{\N}{\mathbb N}

\newcommand{\R}{\mathbb R}

\renewcommand{\S}{{\mathcal S}}

\newcommand{\U}{{\mathcal U}}
\newcommand{\V}{{\mathcal V}}
\newcommand{\W}{{\mathcal W}}

\newcommand{\cone}{\operatorname{Cone}}

\newcommand{\dist}{\operatorname{dist}}
\newcommand{\dom}{\operatorname{Dom}}

\newcommand{\frag}{\operatorname{Frag}}

\newcommand{\id}{\operatorname{id}}

\newcommand{\im}{\operatorname{Im}}

\newcommand{\length}{\operatorname{length}}
\newcommand{\lip}{\operatorname{lip}}
\newcommand{\Lip}{\operatorname{Lip}}

\newcommand{\md}{\operatorname{md}}

\renewcommand{\span}{\operatorname{span}}

\newcommand{\var}{\operatorname{var}}

\newcommand{\al}{\alpha}

\newcommand{\D}{\partial}

\newcommand{\ga}{\gamma}
\newcommand{\Ga}{\Gamma}

\newcommand{\la}{\lambda}

\newcommand{\lra}{\longrightarrow}

\newcommand{\ol}{\overline}

\newcommand{\on}{\:\mbox{\rule{0.1ex}{1.2ex}\rule{1.1ex}{0.1ex}}\:}
\newcommand{\ra}{\rightarrow}
\newcommand{\restr}{\mbox{\Large \(|\)\normalsize}}

\newcommand{\si}{\sigma}

\newcommand{\tang}{\operatorname{Bw-up}}

\def\Xint#1{\mathchoice
   {\XXint\displaystyle\textstyle{#1}}%
   {\XXint\textstyle\scriptstyle{#1}}%
   {\XXint\scriptstyle\scriptscriptstyle{#1}}%
   {\XXint\scriptscriptstyle\scriptscriptstyle{#1}}%
   \!\int}
\def\XXint#1#2#3{{\setbox0=\hbox{$#1{#2#3}{\int}$}
     \vcenter{\hbox{$#2#3$}}\kern-.5\wd0}}

\def\av{\Xint-}

\def\ccgroup{{\mathbb G}}
\def\horalg{{\mathfrak{h}}}
\DeclareMathOperator\sgn{sgn}
\DeclareMathOperator\adoubling{As}
\DeclareMathOperator\spt{spt} 
\def\pcreatedst#1#2{\expandafter\def\csname #1dst\endcsname##1,##2.{#2(##1,##2)}
\expandafter\def\csname #1dstp\endcsname##1.{{#2}_{##1}}
\expandafter\def\csname #1dname\endcsname{#2}} 
\def\pcreatenrm#1#2{\expandafter\def\csname
  #1nrm\endcsname##1.{\left\|##1\right\|_{#2}} \expandafter\def\csname
  #1nrmname\endcsname{\left\|\,\cdot\,\right\|_{#2}}} 
\pcreatedst{h}{d_{\text{\normalfont H}}} 
\DeclareMathOperator\rprm{Rep}
\DeclareMathOperator\curves{Curves}
\DeclareMathOperator\fillfrag{Fill}
\pcreatedst{hz}{d_{Z,\text{\normalfont H}}} 
\pcreatedst{r}{\varrho} 
\pcreatedst{z}{d_Z} 
\pcreatedst{w}{d_W} 
\pcreatedst{inv}{d_{Y_\infty}}
\pcreatedst{invkone}{d_{Y_{k+1}}}
\pcreatedst{invk}{d_{Y_{k}}}
\pcreatedst{x}{d_X} 
\pcreatedst{y}{d_Y} 
\pcreatedst{F}{\varrho_F} 
\pcreatedst{car}{d_{\ccgroup}} 
\pcreatenrm{dx}{\dset,\rdname}
\pcreatenrm{hor}{\horalg}
\pcreatenrm{r}{\rdname}
\pcreatenrm{gen}{}
\pcreatenrm{al}{\alpha{}}
\pcreatenrm{be}{\beta{}}
\pcreatenrm{ta}{\tau{}}
\pcreatenrm{om}{\omega{}}
\pcreatenrm{Om}{\Omega{}}
\pcreatenrm{rLIP}{{\rdname,\text{\normalfont LIP}}}
\pcreatenrm{rLip}{{\rdname,\text{\normalfont Lip}}}
\pcreatenrm{tom}{\tilde\Omega{}}
\pcreatenrm{T}{T}
\pcreatenrm{t}{TX}
\pcreatenrm{ct}{T^*X}
\pcreatenrm{tz}{TZ}
\pcreatenrm{f}{F}
\pcreatenrm{ltwo}{l^2}
\pcreatenrm{empty}{}
\pcreatenrm{al}{\alpha}
\def\albrep#1.{{\mathcal A}_{#1}} 
\DeclareMathOperator\frags{Frag} 
\def\mpush#1.{{#1}_{\sharp}} 
\newcommand{\lines}{\operatorname{Lines}}

\def\epsi{\varepsilon}
\def\tnorm#1.{{{\left\|#1\right\|_{\Lip^*}}}}
\def\ptnorm#1,#2.{{{\left\|#1\right\|_{#2,\Lip^*}}}}
\def\cotnorm#1.{{{\left\|#1\right\|_{\Lip}}}}
\def\lebmeas{{\mathcal L}^1} 
\def\glip#1.{{\bf L}(#1)} 
\let\biglip=\Lip 
\let\smllip=\lip 
\let\metdiff=\md
\let\mrest=\on
\def\fellc{F_c}
\def\biglipr{\text{\normalfont $\rdname$-Lip}}
\def\rmd{\text{\normalfont $\rdname$-md}}
\def\lipalg#1.{{\rm Lip}_{\text{\normalfont b}}(#1)} 
\def\lipfun#1.{{\rm Lip}(#1)} 
\def\lipalgb#1,#2.{{\rm Lip}_{\text{\normalfont b},#2}(#1)}
\def\ball#1,#2.{B(#1,#2)} 
\def\clball#1,#2.{\bar B(#1,#2)} 
\def\dst#1,#2.{d(#1,#2)} 
\def\dstp#1.{d_{#1}} 
\def\hmeas#1.{\mathscr{H}^{\setbox0=\hbox{$#1\unskip$}\ifdim\wd0=0pt 1
    \else #1\fi}} 
\def\rhmeas#1.{\mathscr{H}^{\setbox0=\hbox{$#1\unskip$}\ifdim\wd0=0pt 1
    \else #1\fi}_{\rdname}} 
\def\natural{{\mathbb N}} 
\let\zahlen=\integers
\def\rational{{\mathbb Q}} 
\def\real{{\mathbb R}} 
\DeclareMathOperator\reg{Reg} 
\def\radon{\text{\normalfont Rad}}
\DeclareMathOperator\graphstar{St}
\def\cbb#1.{\text{\normalfont CBB}({\setbox0=\hbox{$#1\unskip$}\ifdim\wd0=0pt 0
    \else #1\fi})}              
\def\cba#1.{\text{\normalfont CBA}({\setbox0=\hbox{$#1\unskip$}\ifdim\wd0=0pt 0
    \else #1\fi})}              
\def\cat#1.{\text{\normalfont CAT}({\setbox0=\hbox{$#1\unskip$}\ifdim\wd0=0pt 0
    \else #1\fi})}              
\def\dset{D_X}
\pcreatenrm{dxdx}{\dset,\xdname}

\begin{document}

\begin{abstract}
We prove metric differentiation  for differentiability
spaces in the sense of Cheeger \cite{cheeger,keith,bate-diff}.  
As corollaries we give
a new proof of one of the main results of \cite{cheeger},  
a proof that the Lip-lip constant of any 
Lip-lip space in the sense of Keith \cite{keith} is equal to $1$, and
new nonembeddability results.
\end{abstract}

\title[Metric Differentiation]{Infinitesimal structure of differentiability spaces, and metric differentiation}
\author{Jeff Cheeger}
\address[Jeff Cheeger]{Courant Institute of Mathematical Sciences}
\email{cheeger@cims.nyu.edu}
\author{Bruce Kleiner}
\address[Bruce Kleiner]{Courant Institute of Mathematical Sciences}
\email{bkleiner@cims.nyu.edu}
\author{Andrea Schioppa}
\address[Andrea Schioppa]{ETH Zurich}
\email{andrea.schioppa@math.ethz.ch}
\thanks{J.C. was supported by a collaboration grant from the Simons 
foundation, and NSF grant DMS-1406407.
B.K. was supported by a Simons Fellowship, a Simons collaboration
grant,  and NSF grants
DMS-1105656, and DMS-1405899.  A.S. was supported by NSF DMS-1105656 and 
European Research Council grant n.291497.}
\date{\today}
\maketitle

\tableofcontents

\section{Introduction}

In this paper we study the metric geometry of differentiability spaces
in the sense of Cheeger \cite{cheeger,keith,bate-diff}.  We 
develop the  infinitesimal geometry of Lipschitz curves and Lipschitz 
functions, generalizing and refining earlier work on 
spaces satisfying Poincar\'e inequalities and differentiability 
spaces; using this we formulate
and establish 
metric differentiation for differentiability spaces.
We then give several applications of these results.  They include
a new proof that the minimal generalized upper gradient of a Lipschitz function is
its pointwise upper Lipschitz constant,  which is one
of the main results of \cite{cheeger}, 
an alternate proof that the Lip-lip constant of any
differentiability space is equal to $1$
\cite{deralb}, and
new nonembeddability results.

In order to motivate the
theory and place it in context, we begin with some background.   We will make some 
additional historical comments at the conclusion of the introduction, after stating 
our results.

\subsection*{Metric differentiation for $\R^n$}
The first instance of metric differentiation 
 was for Lipschitz maps 
$F:\R^n\ra Z$,
where $Z$ is an arbitrary metric space; this is due to Ambrosio 
in the  $n=1$ case and Kirchheim 
 for general $n$ \cite{ambrosio_metric_bv,kirchheim_metric_diff}.
 Although
Rademacher's differentiability theorem for Lipschitz maps
$\R^n\ra\R^m$
does not apply in this situation,
and in fact the usual notion of differentiability does not
even make sense since $Z$ has no linear structure, Ambrosio
and Kirchheim introduced 
a new kind of differentiation --- metric differentiation ---  and
proved that it always holds.  Metric differentiation associates 
to the map $F$ a measurable
Finsler metric, i.e.   a measurable
assignment $x_0\mapsto \|\cdot\|_F(x_0)$ of a seminorm 
(here we identify the tangent space $T_{x_0}\R^n$ with $\R^n$ itself), 
which captures the geometry of the pullback distance
function
\begin{equation}
  \label{eq:Fpseudo}
  \Fdst x_1,x_2.=\zdst F(x_1),F(x_2).
\end{equation}
in the sense that for almost
every $x_0\in \R^n$, the pseudodistance $\varrho_F$ satisfies
\begin{equation}
  \label{eq:kirch_seminorm}
  \Fdst x,x_0.=\fnrm x-x_0.(x_0)+o\left(\|x-x_0\|_{\real^N}\right).
\end{equation}
A slightly different (and stronger) way to express metric differentiation is in terms of 
the family of pseudodistances 
$\{\varrho_F^\la(x_0):\R^n\times\R^n\ra [0,\infty)\}_{\la\in (0,\infty)}$
obtained by  
rescaling
$\varrho_F$ centered at $x_0$:
\begin{equation}
\label{eqn_blow_up_distance}
\varrho_F^\la(x_0)(x_1,x_2)=\la\cdot\varrho_F(x_0+\la^{-1}x_1,x_0+\la^{-1}(x_2))\,.
\end{equation}
For almost every $x_0$, as $\la\ra\infty$ the pseudodistance $\varrho_F^\la(x_0)$
converges uniformly on compact subsets of $\R^n\times\R^n$
 to the pseudodistance associated with the seminorm
$\|\cdot\|_F(x_0)$.  An additional aspect of
metric differentiation
is that for a  Lipschitz curve $\ga:I\ra \R^n$, the
length of the path $F\circ\ga:I\ra Z$ is given by integrating the speed of $\ga$
with respect to the Finsler metric $\|\cdot\|_F$, 
\begin{equation}
\label{eqn_finsler_length}
\length(F\circ\ga)=\int_I \|\ga'(t)\|_F(\ga(t))\,dt\,,
\end{equation}
provided that for a.e. $t\in I$, 
the norm $\|\cdot\|_F$ is defined at $\ga(t)$, and (\ref{eq:kirch_seminorm}) holds
with $x_0=\ga(t)$.  Such curves $\ga$ exist in abundance by Fubini's theorem. 

Like Rademacher's theorem for Lipschitz maps
$\R^n\ra \R^m$, metric differentiation for maps  $\R^n\ra Z$ as above
 can be proved by reducing to the $n=1$ case.   Likewise,
one ingredient in our approach to
metric differentiation for differentiability spaces is a specific form of the 
$1$-dimensional case of metric differentiation due to Ambrosio-Kirchheim,
\cite{ambrosio_kirchheim}.

The $\R^n$ version of metric differentiation
has been applied to the theory of
rectifiable sets and currents in metric spaces
\cite{kirchheim_metric_diff,ambrosio_kirchheim,ambrosio_kirchheim_currents},
to the theory of Sobolev spaces with  metric space targets
\cite{korevaar_schoen_metric_diff}, and in geometric
group theory 
\cite{kleiner_metric_diff,wenger_sharp_isoperimetric_constant,wenger_filling_invariants}. 
As an historical note, 
we mention that metric differentiation 
was discovered 
independently in conversations between Korevaar-Schoen and the second author in 92-93, who 
were 
unaware of Kirchheim's work at the time \cite{korevaar_schoen_metric_diff}.

\subsection*{Metric differentiation for Carnot groups}
A generalization of metric differentiation to Carnot groups was established
by Pauls \cite{pauls_metric_diff}.    If $F:\ccgroup\ra Z$
is a Lipschitz map from a Carnot group $\ccgroup$ equipped with a Carnot-Caratheodory metric
to a metric space $Z$, then for any $x_0\in \ccgroup$ one can apply the canonical rescaling
of $\ccgroup$ to the pseudodistance $\varrho_F$ to 
produce a family of rescaled pseudodistances 
$$
\{\varrho_F^\la(x_0):\ccgroup\times\ccgroup\ra [0,\infty)\}_{\la\in(0,\infty)}
$$
analogous to (\ref{eqn_blow_up_distance}).
Pauls showed that there is a measurable assignment
$x_0\mapsto \|\cdot\|(x_0)$ of seminorms to the horizontal subbundle of $\ccgroup$, such
that for almost every $x_0\in \ccgroup$ with respect to Haar measure, as $\la\ra\infty$,
the rescalings
$\varrho_F^\la(x_0)$ converge on compact subsets of $\ccgroup\times\ccgroup$ to the 
Carnot-Caratheodory pseudodistance associated
with $\|\cdot\|_F(x_0)$; however,   this convergence is only asserted to hold
on the subset of pairs $(x_1,x_2)\in \ccgroup\times \ccgroup$ lying on horizontal geodesics.
This restriction to  special pairs is necessary even in the case of the Heisenberg group, 
as  was shown in \cite{magnani_kirchheim_counterexa}.
Pauls used his metric differentiation theorem to prove that nonabelian Carnot groups
cannot be bilipschitz embedded in Alexandrov spaces, generalizing
an earlier result of Semmes   \cite{david_semmes_book} (which was
based on Pansu's version of 
Rademacher's theorem for mappings between Carnot groups).  Another application was a 
second proof
\cite{cheeger_kleiner_metricdiff_monotone} of the fact that the Heisenberg group cannot be
biLipschitz embedded in $L_1$ (originally proved in
\cite{cheeger_kleiner_L1_BV}).

\subsection*{Differentiability spaces}
The main goal in this paper is to generalize metric differentiation to a large class
of metric measure spaces, namely    differentiability spaces.  
These were first
introduced and studied in \cite{cheeger} without being given a name;
see in particular, Theorem 4.38, Definition 4.42 and the surrounding
discussion.
There  it was shown that 
 PI spaces  --- 
 metric measure spaces that
are doubling and satisfy 
a Poincar\'e inequality in the sense of Heinonen-Koskela \cite{heinonen98}
--- 
are differentiability
spaces.  Differentiability
spaces were further studied in \cite{keith,bate-diff} (under slightly
different hypotheses), where
they were called spaces with a strong measurable differentiable structure, and Lipschitz
differentiability spaces, respectively. 
Examples of differentiability spaces include PI spaces such as Carnot groups
with Carnot-Caratheodory metrics, and more generally
 Borel subsets of  PI spaces, with the restricted measures. 
We recall  (see Section \ref{sec_preliminaries})
that a differentiability space
$(X,\mu)$  has a countable collection $\{(U_i,\phi_i)\}$
of charts, where $\cup_i\, U_i$ has full measure in $X$.
Also, there are canonically defined measurable tangent and cotangent bundles 
$TX$, $T^*X$, and  for any Lipschitz function $u:X\ra\R$, there is a 
well-defined differential $du$, which is   a 
measurable section  of $T^*X$.  

\begin{remark}
We emphasize that the cotangent  and tangent bundles are not on the same footing:
the existence of the cotangent bundle follows quite directly from definition of differentiability
space, whereas the tangent bundle is defined as 
the dual  of the cotangent bundle i.e. 
$TX=(TX^*)^*$. It was observed in 
\cite{cheeger_kleiner_radon} that for PI spaces, given a Lipschitz curve $\ga$, 
for certain parameter values, one can define a velocity vector $\ga'(t)\in T_{\ga(t)}X$
 and that such velocity vectors span the tangent space almost everywhere; in 
 \cite{cheeger_kleiner_metric_diff_unpublished} ``span'' was upgraded to 
 ``are dense''.  As will be seen below, this new geometric characterization
of tangent vectors
was crucial to subsequent developments including 
 the papers 
 \cite{cheeger_kleiner_metric_diff_unpublished}, \cite{bate_jams} and the main results of the present paper, a first example being 
 Theorem \ref{thm_intro_norms_agree}.
\end{remark}

For a Carnot group $\ccgroup$ with a Carnot-Caratheodory metric, the horizontal bundle can be canonically identified with the tangent bundle $T\ccgroup$ of $\ccgroup$ 
viewed as a PI space. This example indicates that in order to formulate a version of metric differentiation for a differentiability space $(X,\mu)$, 
one needs to identify a measurable seminorm  on the tangent bundle
$TX$ and a family of geodesics that will play the role of the family of horizontal geodesics.
We first discuss these in the case of the identity map $X\ra X$, initially
focussing on the measurable seminorm on $TX$; the treatment
in this special case may be viewed as part of the intrinsic 
structure theory of $X$ itself. 

For the remainder of the introduction $(X,\mu)$ will denote a differentiability
space. 

\subsection*{The canonical norm on $TX$}
We now consider several ways of  defining a seminorm on the tangent bundle
$TX$;
as indicated above, these will be used in the formulation of metric
differentiation in the special case of the identity map $X\ra X$.
In the first, we choose a countable dense set $\{x_i\}\subset X$,  and let 
$u_i:X\ra \R$ be  distance function $u_i(x)=d(x,x_i)$.  For every $i$, since the 
differential $du_i$ is a measurable section of the cotangent bundle, by 
duality it defines a measurable family of linear functions on the 
tangent spaces, and therefore $|du_i(\cdot)|$ defines a measurable family
of seminorms on $TX$; taking supremum  we may define
$$
\|v\|_1=\sup_i|du_i(v)|\,.
$$
 As a variations on this, we may define
$\|\cdot\|_2$ and $\|\cdot\|_3$ by replacing the
collection of distance functions $\{u_i\}$ with the collections of 
all distance functions and all 
$1$-Lipschitz functions, respectively; note that this requires a little care since
these collections are uncountable, see Lemma \ref{lem:sup_of_seminorms}.  Finally, it was observed in \cite{cheeger}
that the pointwise upper Lipschitz constant induces a canonical measurable norm
on the cotangent bundle $T^*X$, and by duality this yields
a norm $\|\cdot\|_4$ on $TX$.

\begin{theorem}[See Section \ref{sec:conse}]
\label{thm_intro_norms_agree} 
The seminorms described above agree almost everywhere.  In particular, they are all
norms, and 
$\|\cdot\|_1$ is independent of the choice of the countable dense subset.
\end{theorem}
We will henceforth use $\|\cdot\|$ denote the norms $\|\cdot\|_i$, $1\leq i\leq 4$ on the 
full measure set where they are well-defined and agree.

\subsection*{Generic curves and pairs}
We now discuss the role of curves  in differentiability spaces.  For this
we fix a particular choice of charts $\{(U_i,\phi_i)\}$ as above.  
If $\ga:I\ra X$ is a Lipschitz curve, then one would like to make sense, for 
almost every $t\in I$,  of the 
velocity $\ga'(t)$ and its norm $\|\ga'(t)\|$, where $\|\cdot\|$ is the norm from 
Theorem \ref{thm_intro_norms_agree}  (compare (\ref{eqn_finsler_length})).  
Clearly this is impossible for an arbitrary curve $\ga$, since
it could lie entirely in the complement of the set where the tangent bundle
$TX$ and the norm are well-defined.   To address this, we work with generic curves,
and generic pairs.   Roughly speaking
(see Section \ref{sec:gen_points_gen_veloc}
for the precise definition) if $\ga:I\ra X$ is Lipschitz curve and $t\in I$,
then the pair $(\ga,t)$ is {\em generic} if for some chart $(U_i,\phi_i)$
of the differentiable structure, the time $t$ is:
\begin{itemize}
\item A Lebesgue density
point of the inverse image $\ga^{-1}(U_i)$.
\item An approximate continuity point
of the measurable function $(\phi_i\circ\ga)':I\ra \R^{n_i}$.
\item A density point of $\ga^{-1}(Y)$, where $Y\subset X$ is a full measure 
subset of $\cup_i U_i$ where the norm $\|\cdot\|$ is well-defined.  
\end{itemize}
The curve $\ga$ is {\em generic} if the pair $(\ga,t)$ is generic for almost every
$t\in I$.  It follows readily from the definitions that 
for any generic pair $(\ga,t)$, both the velocity vector $\ga'(t)\in TX$ and 
its norm $\|\ga'(t)\|$ are well-defined.  More generally, we may use essentially the
same notions when $\ga$ is  a {\em curve fragment} rather than a curve, i.e.  a Lipschitz map
$\ga:C\ra X$, where $C\subset \R$ is closed subset; this additional generality is 
essential because a differentiability space might have no nonconstant Lipschitz
curves.    Also,  if $\F$ 
and $\C$ are countable collections of 
Lipschitz functions  and bounded Borel functions respectively,
we may impose
 the additional  requirement  that $t$ is an approximate continuity point
of $(f\circ\ga)'$ and $u\circ\ga$ for all $f\in \F$, $u\in\C$.

\subsection*{Metric differentiation along curves}
Using the notions of genericity above, we can formulate one aspect of
metric differentiation, which is a  statement about  curve fragments.
This uses the  concept of the length of a curve fragment, which is 
straightforward extension of the length of a curve.

\begin{theorem}
\label{thm_curve_metric_diff}
Suppose $\ga:C\ra X$ is a   curve fragment.
\begin{enumerate}
\setlength{\itemsep}{1ex}
\item 
If $(\ga,t)$ is  a generic pair, then $t$ is a point of 
metric differentiability of $\ga$ in the sense that 
(\ref{eqn_blow_up_distance}) holds with $F=\ga$, $x_0=t$,
and for pairs of points $x_1$, $x_2$ where
the right-hand side is defined, and moreover $\|\ga'(t)\|=\|\frac{\D}{\D t}\|_{\ga}(t)$. 

\item  If $\ga$ is generic, then the length of $\ga$ is given by
$$
\length(\ga)=\int_C\|\ga'(t)\|d\L\,,
$$
where $\|\cdot\|$ is the norm of Theorem \ref{thm_intro_norms_agree}.
\end{enumerate}
\end{theorem}
Theorem \ref{thm_curve_metric_diff}
is essentially just an application of 
Theorem \ref{thm_intro_norms_agree}, and the method of proof of the $1$-dimensional version of 
metric differentiation given in \cite{ambrosio_kirchheim}, which exploits a countable
collection of distance functions as in the definition of $\|\cdot\|_1$;
See (\ref{eq:sem_metric_diff}) and Theorem \ref{thm:met_diff_norm}.

\begin{remark}
 We point out that unlike in the Carnot group case (and in particular $\R^n$),
in a differentiability space (for instance the Laakso spaces \cite{laakso}) 
one can have, for a full measure set of points $x\in X$,
two generic pairs $(\ga_1,t_1)$, $(\ga_2,t_2)$ such that $\ga_i(t_i)=x$,
the velocity vectors $\ga_1'(t_1)$, $\ga_2'(t_2)$ coincide, but the curves
are not tangent to first order in the sense
that
 $\limsup_{s\ra 0}\frac{d(\ga_1(t_1+s),\ga_2(t_2+s))}{s}>0$.
Thus it somewhat surprising that the tangent vector alone controls the speed of 
the curve.
\end{remark}

\mbox{}

\subsection*{The density of
generic velocities in $TX$, and consequences}
While the definition of genericity is convenient for stating results about 
individual curve fragments, in order to use it
in   statements about $(X,\mu)$ that
hold at almost every point,  such as Theorem \ref{thm_curve_metric_diff},
 it is crucial to know that generic curve fragments exist
in abundance.  This not at all obvious 
because the definition of a differentiability space is based on the
behavior of Lipschitz functions and does not involve
curves explicitly; in particular  it is not even
clear why $X$  should contain any curve fragments
with positive length.
To deduce the needed abundance,
we invoke Bate's fundamental work on Alberti representations
and differentiability spaces.  
Bate's work shows that one can characterize differentiability
spaces by means of differentiability of Lipschitz functions
along curve fragments.  The main consequence that we will use here is that
for $\mu$-a.e. $p\in X$, the set of generic velocity vectors is dense in $T_pX$
 (see Theorem \ref{thm:Sus_dir_dens}).  Here a {\em generic 
velocity vector} is the velocity vector
$\ga'(t)\in T_{\ga(t)}X$ of a generic pair
$(\ga,t)$.

Theorem \ref{thm_intro_norms_agree} and the density of velocity vectors 
leads directly to the following:

\begin{corollary}
\label{cor_directional_derivative_Lip}
If $u:X\ra \R$ is a Lipschitz function, then for $\mu$-a.e. $p\in X$, the pointwise
upper Lipschitz constant $\Lip u(p)$ is the supremal normalized
directional derivative of $u$ over
generic pairs $(\ga,t)$ with $\ga(t)=p$:
$$
\Lip u(p)=\sup\left\{ \frac{(u\circ\ga)'(t)}{\|\ga'(t)\|}
=\frac{(u\circ\ga)'(t)}{\|\frac{\D}{\D t}\|_{\ga}(t)}
\;\mid \;(\ga,t)\;\text{generic},\;\ga(t)=p,\;\ga'(t)\neq 0\right\}\,.
$$
\end{corollary}
Corollary \ref{cor_directional_derivative_Lip}
has two further consequences.  The first is a new proof of
the characterization  
of the minimal generalized upper gradient in PI spaces as the pointwise
Lipschitz constant 
(see Section \ref{subsec:gf_new}); this was one of the main results in 
\cite{cheeger}.  The second is  a new proof of the following recent result
of the third author
(see   Section 
\ref{subsec:Liplip}).

\begin{theorem}[\cite{deralb}]
\label{thm_Lip_equals_lip}
If $(X,\mu)$ is a differentiability space, and $u:X\ra\R$ is a Lipschitz
function, then for $\mu$-a.e. $p\in X$ we have $\Lip u(p)=\lip u(p)$.
Here $\lip u(p)$ is the pointwise lower Lipschitz constant
(Definition  \ref{defn:Liplip}). 
\end{theorem}
We recall that \cite{keith} introduced the $\Lip$-$\lip$-condition for a metric
measure space, which says that for some $C\in \R$, and every Lipschitz function
$u:X\ra \R$, the upper and lower pointwise Lipschitz constants satisfy
$\Lip u\leq C\lip u$ almost everywhere.  Keith showed that under mild assumptions on 
the measure, a metric measure space satisfying a $\Lip$-$\lip$-condition is a
differentiability space.  Combining this
with Theorem  \ref{thm_Lip_equals_lip}, it follows that one may always take $C=1$.
We note that when $(X,\mu)$ is PI space, or more generally a Borel subset of a
PI space with the restricted measure, it followed from the earlier work 
\cite{cheeger} that $\Lip u=\lip u$ almost everywhere.  These results 
indicate a strong similarity between PI
spaces and differentiability spaces.  

For more discussion of these results we refer the reader to the corresponding Sections.

\subsection*{The structure of blow-ups}
\label{subsec:struc_blow}

For a general differentiability space, there is no natural rescaling as in the 
Carnot group case, so  to formulate an analog of the convergence of the
rescaled pseudodistances
(\ref{eqn_blow_up_distance}), we consider sequences of
rescalings of $X$ with the measure $\mu$  suitably renormalized,
and take pointed  Gromov-Hausdorff limits of the metric measure spaces,
as well as the chart functions and Alberti representations.  We give a
brief and informal account of this here, and refer the reader to Section \ref{sec:geom_blowups}
for more discussion.
For simplicity, in the following statement we assume in addition 
that $X$ is a doubling metric space.

\begin{theorem}
\label{tnm_intro_blow_up}
For $\mu$-a.e. $x\in X$, if $\{\la_j\}$ is any sequence of scale factors
with $\la_j\ra\infty$, and $x\in U_i$, then there is a sequence
$\{\la_j'\}$ such that the sequence 
$$
\{(\la_jX,\la_j'\mu,x)\stackrel{\phi_i}{\lra} (\la_j\R^{n_i},\phi_i(x))\}
$$
of pointed rescalings of the
chart $\phi_i:X\ra \R^{n_i}$ subconverges in the pointed measured Gromov-Hausdorff
sense to a pointed blow-up map
$$
\hat \phi_i:(\hat X,\hat \mu,\star)\ra (T_xX,0)\,,
$$
where $(\hat X,\hat \mu)$ is a doubling metric measure space.
Moreover:
\begin{enumerate}
\setlength{\itemsep}{1ex}
\item  When $T_xX$ is equipped with the norm $\|\cdot\|$ of Theorem 
\ref{thm_intro_norms_agree}, then the map $\hat \phi_i:\hat X\ra T_xX$ becomes
 a metric submersion (see Definition \ref{def_metric_submersion} below).
\item For every unit vector $v$ in the normed space $(T_xX,\|\cdot\|)$, 
there is an Alberti representation of $\hat \mu$ whose support is contained in the 
collection of unit
speed geodesics $\ga:\R\ra \hat X$ with the property that $\hat \phi_i\circ\ga:\R\ra T_xX$
has constant velocity $v$; furthermore, the measure associated to each such 
$\ga$ is just arclength.  This Alberti representation is obtained by 
blowing-up suitable Alberti representations in $X$.  
\end{enumerate}  
\end{theorem}

\begin{definition}
\label{def_metric_submersion}
A map $f:Y\ra Z$ between metric spaces is a metric submersion if it is a $1$-Lipschitz
surjection, and for every $y_1\in Y$, $z_2\in Z$, there is a $y_2\in f^{-1}(z_2)$
such that $d(y_1,y_2)=d(f(y_1),z_2)$.  Equivalently, given any two fibers
$f^{-1}(z_1),\,f^{-1}(z_2)\subset Y$, the distance function from the fiber
$f^{-1}(z_1)$ is  constant and equal to  $d(z_1,z_2)$ on the other fiber $f^{-1}(z_2)$.  
\end{definition}

 To aid the reader's intuition, it might be helpful to look at
the example $(\real^2,{\mathcal L}^2)$, where on $\real^2$ we consider
the $l^1$-norm; as this norm is not strictly convex, one can obtain an
Alberti representation of ${\mathcal L}^2$ by using unit-speed
geodesics in ${\mathcal L}^2$ with corners, i.e.~geodesics which do
not lie in straight lines. Blowing-up such representations at a
generic point, one obtains an Alberti representation of ${\mathcal
  L}^2$ whose transverse measure is concentrated on the set of
straight lines in $\real^2$. 

There are  precursors to 
Theorem \ref{tnm_intro_blow_up} in \cite{cheeger} in the case
of PI spaces.  In that case the blow-ups (tangent cones) are also PI spaces,  the
coordinate functions blow-up to generalized linear functions, and \cite{cheeger}
proved the surjectivity of the 
canonical map $Y\ra T_xX$.  Distinguished geodesics 
of a different sort were discussed
in \cite{cheeger}, namely the gradient lines of generalized linear functions; however,
unlike the curves in
the support of the Alberti representations of Theorem \ref{tnm_intro_blow_up}\,(2),
these need not be affine with respect to the blow-up chart $\hat \phi_i$.

\begin{remark}
The third author \cite{deralb} and David
\cite{david_difftang_ahlfors} also have  results related
to Theorem \ref{tnm_intro_blow_up}\,(1).  
They show that certain blow-up
maps are Lipschitz quotient maps, which is  a weaker version of the 
metric submersion property.
The paper \cite{deralb} is concerned with the relationship between Weaver derivations
\cite{weaver}
and Alberti representations without the assumption that one has a differentiability space,
so the setup there is much more general 
than the one considered here.
We point out that
our results in Section \ref{sec:geom_blowups} have natural counterparts in that
general context,
under the assumption that $\mu$ is asymptotically doubling.
We note that one of the main ingredients in Theorem \ref{tnm_intro_blow_up}
is a procedure for blowing-up Alberti representations, which has other applications.
In particular, it allows one to blow-up Weaver derivations under the
assumption that the background measure is asymptotically doubling. We
point out that, as the metric measure space $X$ does not need to
possess a group of dilations, it is not trivial to find a correct way
to rescale derivations and pass to a limit; however, by taking advantage
of the representation of Weaver derivations in terms of Alberti
representations proven in \cite{deralb}, one can use Theorem
\ref{thm:measured_blow_up} to blow-up a derivation at a generic point.
Moreover, as the blown-up Alberti representation is concentrated on
the set of geodesic lines, the blown-up derivation corresponds to a
$1$-normal current (in the sense of Lang) without boundary.
We refer the reader to Section \ref{sec:geom_blowups} 
(in particular Theorem \ref{thm:alberti_blow_up} and Remark \ref{rem:space_blow_up})
for more details.
\end{remark}

Theorem \ref{tnm_intro_blow_up}
 implies that the blow-up of any Lipschitz function at a generic point
is harmonic, in the following sense.

\begin{definition}
\label{def_p_lip_harmonic}
Suppose $(W,\zeta)$ is proper metric measure space, where $\zeta$ is a locally
finite Borel measure.
Then a Lipschitz function $u:W\ra\R$ is {\em $p$-$\Lip$-harmonic} if for every
ball $B(x,r)\subset W$, and every
Lipschitz function $v:W\ra\R$ that agrees with $u$ outside $B(x,r)$, we have
$$
\int_{B(x,r)}
(\Lip v)^p\,d\hat \mu\geq \int_{B(x,r)}(\Lip u)^p\,d\hat \mu\,.
$$
\end{definition}

Theorem \ref{tnm_intro_blow_up} yields:
\begin{corollary}
\label{cor_blow_ups_harmonic}
Suppose $u:X\ra \R$ is a Lipschitz function.  Then for $\mu$ a.e. $x\in X$, for
any blow-up sequence as in Theorem \ref{tnm_intro_blow_up} there is a blow-up limit
$\hat u:Y\ra \R$ such that:
\begin{enumerate}
\item $\hat u$ is {\em $p$-$\Lip$-harmonic} for all $p\geq 1$.
\item For any $y\in Y$, $r\in [0,\infty)$ 
we have $\var(u,y,r) =r\cdot \Lip(u)(x)$,
where $\var(u,y,r)$ is the variation of $\hat u$ over $B(y,r)$:
$$
\var(u,y,r)=\sup\{|u(z)-u(y)|\mid z\in B(y,r)\}\,.
$$  
In particular, $\Lip(\hat u)(y)=\lip(\hat u)(y)=\Lip(u)(x)$
for all $y\in Y$, and  $\Lip(u)(x)$ is also  the global Lipschitz constant of
$\hat u$.  
\end{enumerate}
\end{corollary}
We remark that 
in the terminology of \cite[Sec. 6]{keith},
part (2) of the corollary says that blow-ups are $1$-quasilinear; this
refines \cite[Sec. 6]{keith}, where it was shown  that blow-ups are $K$-quasilinear
for some $K$.

It is an open question whether a blow-up  of a differentiability
space must be
a PI space, or even a differentiability space.
Corollary \ref{cor_blow_ups_harmonic} may be compared with the result from
\cite{cheeger}, which asserts that blow-ups of Lipschitz functions are
generalized linear functions --- $p$-harmonic functions with constant norm
gradient.  The proof in \cite{cheeger} is quite different however --- it is based on 
asymptotic harmonicity and 
breaks down in differentiability spaces.

The results above all speak to the broader topic of the infinitesimal structure
of differentiability spaces.  There are a number of open questions here.
The present state of knowledge makes it difficult to formulate compelling conjectures
or questions
in a precise form, but one may ask the following:

\begin{question}
If $(X,\mu)$ is a differentiability space, is there a countable collection 
$\{U_i\}$ of Borel subsets of $X$, such that $\mu(X\setminus \cup_i\,U_i)=0$ and
every $U_i$ admits a measure-preserving isometric embedding in a PI space?
\end{question}
\noindent
If the answer is yes, then blow-ups of differentiability spaces at generic points
will also be PI spaces, so one may approach this question by trying to verify that
blow-ups have various properties of PI spaces, such as quasiconvexity,  a
differentiable structure, etc.  It is of independent interest to gain a better 
understanding of the structure of blow-ups in the PI space case.  Known examples
suggest that the blown-up Alberti representations may have accessibility properties
similar to the accessibility one has in Carnot groups.

\begin{remark}
Since the preprint version of this paper was posted, there has been significant 
progress  on these issues. It was
shown in \cite{schioppa_lip_lip_stable} that at almost every point, any blow-up 
of a differentiability space is again a differentiability space.
It was shown  in 
\cite{bate_li_rnp}
that at almost every point, every  blow-up of an RNP differentiability space
satisfies a 
non-homogeneous Poincare inequality, and consequently is a quasiconvex
RNP differentiability space; it was also shown
\cite{bate_li_rnp} that RNP differentiability spaces
may  be characterized by quantitative
connectedness properties of universal Alberti representations,
and by  asymptotic non-homogeneous  Poincar\'e
inqualities.
\end{remark}

\subsection*{The infinitesimal geometry of Lipschitz maps}
\label{subsec:inf_geom_lipmaps}

We now return to the general case of metric differentiation.
Consider a Lipschitz map $F:X\to Z$, where $Z$ is any
metric space, and let $\varrho=\varrho_F$ be the pullback distance function
(\ref{eq:Fpseudo}).   Our results in this
case parallel what has been discussed above for the special
case of the identity map $\id_X:X\ra X$, so we will be brief
and focus on  the novel features; see Section \ref{sec:lipmaps_metricdiff}
for the details.

The map $F$ gives rise to a distinguished subset of the Lipschitz functions
on $X$,  namely the set of pullbacks $u\circ F$, where $u:Z\ra\R$ is Lipschitz,
or equivalently, the set of functions $v:X\ra\R$ 
that are  Lipschitz with respect to the pseudodistance $\varrho$.  

\begin{theorem}[Theorem \ref{thm:distance_functions_span}]
There is a canonical subbundle $\W_\varrho\subset T^*X$ such that the differential
of any $\varrho$-Lipschitz function $v:X\ra \R$ belongs to $\W_\varrho$ $\mu$-almost
everywhere.  Moreover, for any countable dense subset $D_X\subset X$, the 
set of differentials of the corresponding
$\varrho$-distance functions span $\W_\varrho$.
\end{theorem}

One may  construct several seminorms on $TX$  analogous to the seminorms $\|\cdot\|_j$,
$1\leq j\leq 4$, of Theorem \ref{thm_intro_norms_agree}.  For instance, given a 
countable dense subset $D_X\subset X$, we may define a seminorm by 
$$
\|\cdot\|_{1,\varrho}=\sup\{|d\rho_x|\;|\; x\in D_X\}\,,
$$
where $\varrho_x$ is the $\varrho$-distance from $x$; analogs of the other three
seminorms
are defined similarly, using the pseudodistance $\varrho$ instead of the 
distance function $d_X$.  

\begin{theorem}[Theorem \ref{thm:can_seminorm}]
\label{thm_intro_pullback_norms_agree}
The seminorms agree almost everywhere, giving rise to
a canonical seminorm $\|\cdot\|_\varrho$
on $TX$.
\end{theorem}

Unlike in the case of the identity map, when $\varrho=d_X$, the canonical
seminorm need not 
be a norm.  Instead it induces a norm on the quotient bundle
$TX/\W_\varrho^\perp$
and a dual norm
$\|\cdot\|_\varrho^*$ on the canonical subbundle $\W_\varrho\subset T^*X$;
here $\W_\varrho^\perp\subset TX$ is the annihilator of the 
$\W_\varrho\subset T^*X$, .

There are two different ways to formulate metric differentiation in terms of
blow-ups.  In the first, we refine Theorem \ref{tnm_intro_blow_up} by bringing
in the sequence of rescaled pseudodistances  $\{\la_j\varrho\}$ as well.  After
passing to a subsequence, these will Gromov-Hausdorff converge (in a natural sense)
to a limiting pseudodistance $\hat\varrho$ on $Y$.  Then in addition to conclusions
(1) and (2) of Theorem \ref{tnm_intro_blow_up}, we have:

\begin{enumerate}
\setcounter{enumi}{2}
\setlength{\itemsep}{1ex}
\item 
When $Y$ and $T_pX$ are equipped with the pseudodistance 
$\hat\varrho$ and the seminorm $\|\cdot\|_\varrho$ of Theorem 
\ref{thm_intro_pullback_norms_agree} respectively,  the map $\hat \phi_i:Y\ra T_pX$ is
 a metric submersion.
\item For every unit vector $v$ in the normed space $(T_pX,\|\cdot\|)$, 
there is an Alberti representation of $\hat \mu$ whose support is contained in the 
collection of curves $\ga:\R\ra Y$ with the property that $\hat \phi_i\circ\ga:\R\ra T_pX$
has constant velocity $v$,  $\ga$ is a unit speed $d_Y$-geodesic, and
a constant $\|v\|_\varrho$-speed $\hat\varrho$-geodesic. 
\end{enumerate}

A second way to formulate the blow-up assertion is to take an ultralimit of the
map $F:X\ra Z$.  We refer the reader to Section \ref{subsec:metdiff_blowups} 
for the statements.

One  consequence of (4)  is that the blown-up Alberti representations appearing in 
Theorem \ref{tnm_intro_blow_up}\,(2) may be viewed as invariants of the  differentiability
space structure, in the following way.  The definitions readily imply that if
$(X,\mu)$ is a differentiability space, $(Z,\nu)$ is a metric measure space,
and $F:(X,\mu)\ra (Z,\nu)$ is a
bilipschitz homeomorphism that is also measure class preserving in the sense
that  pushforward measure $F_*\mu$ and $\nu$ are mutually absolutely continuous, then 
$(Z,\nu)$ is also a differentiability space.  When $X$ is doubling, for almost
every $p\in X$, we can then take a Gromov-Hausdorff limit of the the sequence of
rescalings of $F$ as in Theorem \ref{tnm_intro_blow_up}, 
to obtain a bilipschitz homeomorphism 
$$
\hat F:(\hat X,\hat p)\lra (\hat Z,\hat F(\hat p))\,.
$$
This blow-up map $\hat F$ will preserve the blow-up measures up to scale, and will
preserve the blow-up Alberti representations from Theorem \ref{tnm_intro_blow_up}(2)
up to a change of speed that depends only on the choice of tangent vector $v$.

\subsection*{Applications to embedding}
In Section \ref{sec:examples} we apply metric differentiation to Lipschitz maps
between Carnot groups, Alexandrov spaces with curvature bounded above or below,
and the inverse limit spaces in \cite{cheeger_inverse_poinc}, showing that such maps are 
strongly constrained on an infinitesimal level.  

\subsection*{Further discussion}
We now make some remarks about the evolution of some of the main ideas in this paper ---
generic velocities, the proof of abundance, the structure of blow-ups, and their
distinguished geodesics.

While \cite{cheeger} clarified many points at the foundation of PI spaces, the role
of curves remained somewhat mysterious, and in particular velocity vectors to curves
were not considered there.
In fact, although Lipschitz curves were used in the
original definition of a PI space by Heinonen-Koskela (which is based on upper gradients) 
there is an equivalent definition  in which curves do not appear at all \cite{keith_lip_pi}.

The first appearance of tangent vectors to curves in the context of PI spaces was 
in   \cite{cheeger_kleiner_radon}.   There a notion similar to generic velocity 
vectors was introduced, and it was shown that they span the tangent space at a typical
point; in addition, there was a new characterization of the minimal 
generalized upper
gradient, which may be viewed as a precursor to Corollary \ref{cor_directional_derivative_Lip}.
Metric differentiation for PI spaces was announced in 
\cite[p.1020]{cheeger_kleiner_radon}.   This was work of the first two authors, which led to an 
unpublished account of metric differentiation 
\cite{cheeger_kleiner_metric_diff_unpublished}
that was similar in several respects to the
present paper.  For instance,   it used a notion of generic velocity vectors, 
and contained a
blow-up statement like 
Theorem \ref{tnm_intro_blow_up} involving a distinguished family of geodesics with
constant velocity in the blow-up chart; however, it did not use Alberti representations.
We mention that
 is  easy to see that  the collection of nongeneric Lipschitz curves 
$\ga:I\ra X$ has zero $p$-modulus, for every $p$.
This yields a weak form of abundance of generic curves in the PI space case.  
A key ingredient in 
\cite{cheeger_kleiner_metric_diff_unpublished} was a proof of the density of 
the directions of generic velocity vectors based on a much deeper
argument that borrowed ideas --- a renorming argument and the equality
$\Lip u$ and the minimal  generalized upper gradient --- from \cite{cheeger}.

Bate's beautiful work on Alberti representations \cite{bate-diff,bate_jams}
greatly strengthened the connection between curves and differentiability, 
providing several different alternate characterizations of differentiability 
spaces in terms of Alberti representations.  His approach was partly motivated by
the work of Alberti-Csornyei-Preiss on differentiability for subsets of $\R^n$, 
and  an observation of Preiss that the characterization of the minimal
generalized upper gradient in \cite{cheeger_kleiner_radon} implied the existence of
Alberti representations for PI spaces \cite[Sec. 10]{bate_jams}.  

When \cite{bate-diff} appeared, the third author used it to give a proof of
$\Lip=\lip$ based on a renorming construction, without being aware of the contents
of \cite{cheeger_kleiner_metric_diff_unpublished}.   Independently,
the first two authors  recognized that \cite{bate-diff} could
be used to give a stronger and more general treatment of metric differentiation,
and proposed writing the present paper.

\subsection*{Acknowledgements}
We would like to thank David Bate, Guy David,  and Sean Li for drawing our attention to
errors in an earlier version of this paper.

\section{Preliminaries}
\label{sec_preliminaries}
\newcount\maskprel
\maskprel=0
\ifnum\maskprel>0{
\begin{enumerate}
\item Standing assumptions about differentiability spaces, and properties.
\item Metric derivative for curve fragments (rectifiable sets).  Ambrosio-Kirchheim.
\item Alberti representations.
\item Bate's theorem.
\item Gromov-Hausdorff limits of spaces, and spaces with baggage (measures,
collections of functions, or paths).
\end{enumerate}
}\fi
 \subsection{Standing assumptions and review of differentiability
   spaces}
 \label{subsec:stand_ass}
Throughout this paper, the pair $(X,\mu)$ will denote
a {\bf differentiability space}; this means that
$(X,\xdname)$  is a complete, separable metric space, $\mu$ is
a  Radon measure, and the pair
$(X,\mu)$  admits a measurable differentiable structure 
as recalled
below, cf. \cite{cheeger,keith}.

 We briefly highlight the main features of a differentiability
space, see below for
more discussion:
\begin{enumerate}
\item There is a countable collection of charts
  $\{(U_\alpha,\phi_\alpha)\}_\alpha$, 
  where $U_\al\subset X$ is measurable and $\phi_\al$ is Lipschitz,
  such that $X\setminus
  (\cup_\alpha U_\alpha)$ is $\mu$-null, and each real-valued
  Lipschitz function $f$ admits a first order Taylor expansion with
  respect to the components of $\phi_\alpha:X\to\real^{N_\alpha}$
  at
  generic points of $U_\alpha$, i.e.
  there exist a.e. unique measurable functions $\frac{\partial
      f}{\partial\phi_\alpha^i}$ on $U_\al$ such that:
  \begin{multline}
    \label{eq:taylor_exp}
    f(x)=f(x_0)+\sum_{i=1}^{N_\alpha}\frac{\partial
      f}{\partial\phi_\alpha^i}(x_0)\left(\phi_\alpha^i(x)-\phi_\alpha^i(x_0)\right)\\
    +o\left(\xdst x,x_0.\right)\quad(\text{for $\mu$-a.e.~$x_0\in U_\alpha$}).
  \end{multline}
  \vskip2mm
\item There are  measurable cotangent and tangent bundles
  $T^*X$ and $TX$ (see also subsection
  \ref{subsec:meas_vec_bund}). The fibres of $T^*X$ are generated by
  the differentials of Lipschitz functions, and the tangent bundle of
  $TX$ is defined \emph{formally} by duality: part of the motivation
  of the present work is to give a concrete description of $TX$ by
  using velocity vectors of Lipschitz curves.
  \vskip2mm
\item 
Natural dual norms $\|\cdot\|_{\Lip}$ and $\|\cdot\|_{\Lip}^*$ on $T^*X$ and
  $TX$ respectively. The norm $\|\cdot\|_{\Lip}$ is induced by the pointwise upper
  Lipschitz constant, i.e.~for any real-valued Lipschitz functions
  $f$ we have $\|df\|_{\Lip}=\biglip f(x)$ for $\mu$-a.e.~$x\in X$.
\end{enumerate}
We recall that $\biglip f(x)$ denotes the \textbf{(upper) pointwise
  Lipschitz constant of $f$ at $x$}, that is:
     \begin{equation}
       \label{eq:Liplip2}
       \biglip f(x)=\limsup_{r\searrow0}\sup\left\{\frac{\left|f(y)-f(x)\right|}{r}:\xdst
         x,y.\le r\right\}.
     \end{equation}
\par We now give a brief review of some definitions from
\cite{cheeger,keith}; an exposition can be found in \cite{kleiner_mackay}.
Let $(Z,\nu)$ be a metric measure space.
Let $\U$ be a (countable) collection
of Lipschitz functions on $Z$.
  Then $\U$ is {\bf dependent at $x\in Z$} if
some finite nontrivial linear combination $v$ of elements of $\U$
is constant to first order at $x$, i.e.
$|v(y)-v(x)|=o(\xdst x,y.))$. Alternatively, one can say that the pointwise upper
  Lipschitz constant $\biglip v$ of $v$ vanishes at $x$.  The {\bf dimension of $\U$ at $x$} is the
supremal cardinality of a subset that is linearly independent at $x$;  the 
dimension function 
$\dim_{\U}:Z\ra \natural\cup\{\infty\}$ is Borel whenever $\U$
is a countable collection. Suppose that $U\subset Z$ is a Borel set
with positive $\nu$-measure, and that
$\phi:U\ra \R^n$ is Lipschitz. The pair $(U,\phi)$
is a {\bf chart} if the component functions
$\phi_1,\ldots,\phi_n$ of $\phi$
are independent
at $\nu$-a.e.~$x\in U$, and if for each real-valued Lipschitz function
$f$, the $(n+1)$-tuple
$\left(\phi_1,\ldots,\phi_n,f\right)$ is dependent at $\nu$-a.e.~$x\in
U$. In particular, there are, unique up to $\nu$-null sets, Borel
functions $\frac{\partial f}{\partial \phi_i}:U\to\real$ such that the
Taylor expansion (\ref{eq:taylor_exp}) holds for $\nu$-a.e.~$x_0\in
U$; in this case we also say that $f$ is \textbf{differentiable at
  $x_0$ with respect to the
  $\{\phi_i\}_{i=1}^n$}. 

A metric measure space $(Z,\nu)$
 \textbf{admits a measurable differentiable
structure} if there exists an countable collection of charts
$\left\{(U_\alpha,\phi_\alpha)\right\}_\alpha$ such that
$Z\setminus (\cup_\alpha U_\alpha)$ is $\nu$-null. 
Without loss of
generality, we will always assume that
for each pair $(\alpha,\beta)$, at each point of $U_\alpha\cap
U_\beta$ the functions $\phi_\alpha$ are differentiable with respect
to the functions $\phi_\beta$.

One says that a metric measure space $(Z,\nu)$ is {\bf (almost everywhere) finite dimensional} if for any countable
collection $\U$ of Lipschitz functions, the dimension $\dim_{\U}$ is finite 
almost everywhere. It follows from a selection argument 
\cite{cheeger,keith} that when $\nu$ is $\si$-finite, then $(Z,\nu)$
admits a measurable differentiable structure if and only if it is finite dimensional. 
Thus, apart from being a standard condition on a measure,
$\si$-finiteness is a natural assumption in the present topic.
As the measure $\nu$ only enters through its sets of measure zero, one really only
cares about the \emph{measure class} of $\nu$;  hence if $\nu$ is $\si$-finite, 
then without loss
of generality one may take $\nu$ to  be finite.
\par 
We finally give a brief justification 
of why we assume $X$ to be
complete in the definition of a differentiability space, 
which was also a working assumption in
\cite{bate_speight,bate-diff}. Suppose 
$(Z,\nu)$ is a metric measure space, where $Z$ is not necessarily 
complete.  Denote by $\bar Z$ its completion, and let $\bar\nu$ be the pushforward of $\nu$
under the inclusion $Z\ra \bar Z$.  Then any Lipschitz function $u\in \Lip(Z)$
extends uniquely to $\bar Z$, and since $Z$ is dense in $\bar Z$, the notions of
dependence and dimension for a collection $\U\subset\Lip(Z)$ at any $x\in Z$
agree with the notions for the corresponding collection $\bar\U\subset\Lip(\bar Z)$. 
Hence $(Z,\nu)$ has a differentiable structure if and only if $(\bar Z,\bar\nu)$ has a 
measurable differentiable structure.
\subsection{The metric derivative for $1$-rectifiable sets}
\label{subsec:metric_deriv}
Let $Z$ be a separable metric space and denote by $\zdname$ the metric
on $Z$. We say that a pseudometric $\rdname$ on $Z$ is \textbf{Lipschitz
  compatible} if there is a nonnegative constant $C$ such that:
\begin{equation}
  \label{eq:comp_sem}
  \rdname\le C\zdname;
\end{equation}
we say that a function $f:Z\to W$ is $\rdname$-Lipschitz if there is
a nonnegative $C$ such that:
\begin{equation}
  \label{eq:lip_sem}
  \wdst f(z_1),f(z_2).\le C\rdst z_1,z_2.\quad(\forall z_1,z_2\in Z).
\end{equation}
Note that $\rdname$-Lipschitz functions are necessarily
$\zdname$-Lipschitz; when referring to the background metric $\zdname$ we will
simply use the term Lipschitz. We denote by $\hmeas .$ the
$1$-dimensional Hausdorff measure on $Z$ and by $\rhmeas .$ the
$1$-dimensional Hausdorff measure associated to the pseudometric $\rdname$.
\par We now recall metric differentiation results of  \cite{kirchheim_metric_diff,ambrosio_kirchheim,ambrosio_tilli} 
in the case of $1$-rectifiable sets.
\par Let $Y$ be a Lebesgue measurable subset of $\real$ and let $\ga:Y\ra Z$ be a Lipschitz map.
We fix a countable dense subset $\{z_i\}$ of $Z$, and 
let $u_i$ be the 
pullback of the pseudodistance function $\rdstp
z_i.(\cdot)=\rdst\cdot,z_i.$ by the map $\gamma$.
Then $\gamma$ has a \textbf{$\rdname$-metric differential}
$\rmd\ga:Y\to[0,\infty)$, which is uniquely determined for $\lebmeas$
a.e. $t\in Y$, and which has the following properties:
\begin{description}
\item[(MD1)] Rescalings of the pullback pseudometric $\ga^*\rdname$ at $t$ converge uniformly on 
compact sets to $\rmd\ga(t)\,d_{\R}$, that is, the Euclidean distance
scaled by the factor $\rmd\ga(t)$.
\vskip2mm
\item[(MD2)] Consider a point $t\in Y$ such that:
\vskip2mm
  \begin{enumerate}
  \item   The point $t$ is a Lebesgue density point of $Y$;
    \vskip2mm
  \item  The derivatives
    of the functions $\{u_i\}_i$ exist at $t$;\vskip2mm
  \item The derivatives $\{u'_i\}_i$ are approximately continuous at $t$;\vskip2mm
  \item The function $\sup_i|u_i'|$ is approximately continuous at $t$.\vskip2mm
  \end{enumerate}
 Then the $\rdname$-metric differential exists at $t$ and is given by:
  \begin{equation}
    \label{eq:sem_metric_diff}
    \rmd\gamma(t)=\sup_i\left|u'_i(t)\right|.
  \end{equation}
\item[(MD3)] One has an area formula \cite[Thm.~7]{kirchheim_metric_diff}:
  \begin{equation}
    \label{eq:sem_metric_area}
    \int_Z\#\left\{t\in Y: \gamma(t)=z\right\}\,d\rhmeas .(z)=\int_Y \rmd\gamma(t)\,d\lebmeas(t).
  \end{equation}
\end{description}
In the case in which the metric differential refers to the metric
$\zdname$ we will use the symbol $\md\gamma$ instead of $\text{$\zdname$-md}\gamma$.
\subsection{Alberti representations}
\label{subsec:alb_rep}
\par Alberti representations were introduced in
\cite{alberti_rank_one} to prove the so-called rank-one property for BV
functions; they were later applied to study the differentiability
properties of Lipschitz functions $f:\real^N\to\real$
\cite{acp_plane,acp_proceedings} and have recently been used to obtain a
description of measures in differentiability spaces
\cite{bate-diff}. We first give an informal definition.
\par An \textbf{Alberti representation} of a Radon measure $\mu$ is a
generalized Lebesgue decomposition of $\mu$ in terms of
$1$-rectifiable measures: i.e.~one writes $\mu$ as an integral:
\begin{equation}
  \label{eq:albrep_informal}
  \mu=\int \nu_\gamma\,dP(\gamma),
\end{equation}
where $\{\nu_\gamma\}$ is a family of $1$-rectifiable
measures. The standard example is offered by Fubini's Theorem; given
$x\in\real^{N-1}$, denote by $\gamma(x)$ the parametrized line in
$\real^{N}$ given by $\gamma(x)(t)=x+te_{n+1}$; then an Alberti
representation of the Lebesgue measure ${\mathcal L}^{N}$ is given by:
\begin{equation}
  \label{eq:fubini}
  {\mathcal L}^{N}=\int_{\real^{N-1}}\hmeas._{\gamma(x)}\,d{\mathcal L}^{N-1}(x).
\end{equation}
\par To make the previous account more precise we introduce more
terminology. For more details we refer the reader to \cite{bate-diff}
and \cite[Sec.~2.1]{deralb}; note however, that we slightly diverge
from the treatments in \cite{bate-diff,deralb} because we discuss also
unbounded $1$-rectifiable sets: the need to do so becomes apparent in
Section \ref{sec:geom_blowups}.
\begin{definition}\label{defn:frags}
A {\bf fragment in $X$}
is a Lipschitz map $\ga:C\ra X$, where $C\subset\real$ is closed. The
set of fragments in $X$ will be denoted by $\frag(X)$.
\end{definition}
We need to topologize $\frags(X)$; let $F(\real\times X)$ denote the set of
closed subsets of $\real\times X$ with the Fell topology
\cite[(12.7)]{kechris_desc}; we recall that a basis of the Fell
topology consists those sets of the form:
\begin{equation}
  \label{eq:felltop}
  \left\{F\in F(\real\times X): F\cap K=\emptyset, F\cap
    U_i\ne\emptyset\;\text{for $i=1,\ldots,n$}\right\},
\end{equation}
where $K$ is a compact subset of $\real\times X$, and 
$\{U_i\}_{i=1}^n$ is a finite collection of open subsets of $\real\times X$.
Note that
the empty set $\emptyset$ is included in $F(\real\times X)$ and that,
if $X$ is locally compact, the topological space $F(\real\times X)$ is
compact. We now consider the set $\fellc(\real\times X)=F(\real\times
X)\setminus\{\emptyset\}$ which is, if $X$ is locally compact, a
$K_\sigma$, i.e.~a countable union of compact sets. Each fragment $\gamma$ can be identified with an element
of $\fellc(\real\times X)$ and thus $\frag(X)$ will be topologized as
a subset of $\fellc(\real\times X)$. We will use fragments to
parametrize $1$-rectifiable subsets of $X$.
\par We now briefly discuss the
topology on Radon measures that allows to make sense of an integral
like (\ref{eq:albrep_informal}). Let $C_c(X)$ denote the set of continuous function defined on $X$
with compact support; recall that the set $C_c(X)$ is a Fr\'echet space.
We denote by $\radon(X)$ the set of (nonnegative) Radon measures on $X$; as $\radon(X)$
can be identified with a subset of the dual of $C_c(X)$, we will
topologize it with the restriction of the weak* topology. In
particular, when we assert that a map $\psi:Z\to\radon(X)$ is Borel,
we mean that for each $g\in C_c(X)$, the map:
\begin{equation}
  \label{eq:radon_borel}
  z\mapsto\int_X g(x)\,d\left(\psi(z)\right)(x)
\end{equation}
is Borel.
\begin{defn}
  \label{defn:alberti_reps}
  An Alberti representation of the measure $\mu$ is a pair $(P,\nu)$
  such that:
  \begin{description}
  \item[(Alb1)] $P$ is a Radon measure on $\frags(X)$;
    \vskip2mm
  \item[(Alb2)] The map $\nu:\frags\to\radon(X)$ is Borel and, for each
    $\gamma\in\frags(X)$, we have $\nu_\gamma\ll\hmeas ._\gamma$,
    where $\hmeas._\gamma$ denotes the
    $1$-dimensional Hausdorff measure on the image of $\gamma$;\vskip2mm
  \item[(Alb3)] The measure $\mu$ can be represented as $\mu=\int_{\frags(X)}\nu_\gamma\,dP(\gamma)$;\vskip2mm
  \item[(Alb4)] For each Borel set $A\subset X$ and all real numbers
    $b\geq a$, the map
    $\gamma\mapsto\nu_\gamma\left(A\cap\gamma(\dom\gamma\cap[a,b])\right)$
    is Borel.
  \end{description}
\end{defn}
\par We now recall some definitions regarding additional properties of
Alberti representations.
\begin{defn}
   \label{defn:alberti_constants}
   An Alberti representation $\albrep.=(P,\nu)$ is said to be
   \textbf{$C$-Lipschitz (resp.~$(C,D)$-biLipschitz)} if
   $P$-a.e.~$\gamma$ is $C$-Lipschitz (resp.~$(C,D)$-biLipschitz).
 \end{defn}
\begin{defn}
   \label{defn:alberti_speed}
   Let $\sigma:X\to[0,\infty)$ be Borel and $f:X\to\real$ be
   Lipschitz. An Alberti representation $\albrep.=(P,\nu)$  is said to
   be \textbf{have $f$-speed $\ge\sigma$ (resp.~$>\sigma$)} if for $P$-a.e.~$\gamma\in\frags(X)$ and
   $\lebmeas\mrest\dom\gamma$-a.e.~$t$ one has
   $(f\circ\gamma)'(t)\ge\sigma(\gamma(t))\metdiff\gamma(t)$ (resp.~$(f\circ\gamma)'(t)>\sigma(\gamma(t))\metdiff\gamma(t)$).
 \end{defn}
Another property regards the \emph{direction}, with respect to a
finite tuple of Lipschitz functions, of the fragments used in an
Alberti representation. To measure the direction one can use the
notion of Euclidean cone:
\begin{defn}\label{defn:cone}
  Let $\theta\in(0,\pi/2)$, $v\in{\mathbb S}^{n-1}$; the \textbf{open
    cone} $\cone(v,\theta)\subset\real^n$ with axis $v$ and opening
  angle $\theta$ is:
\begin{equation}
    \cone(v,\theta)=\{u\in\real^q: \tan\theta\langle
    v,u\rangle>\|\pi_v^\perp u\|_2\},
  \end{equation} where $\pi_v^\perp$ denotes the orthogonal projection
  on the orthogonal complement of the line $\real v$.
\end{defn}
\begin{defn}
   \label{defn:alberti_direction}
   \par Given a Lipschitz function $f:X\to\real^n$, an Alberti
   representation $\albrep.=(P,\nu)$ is said to be \textbf{in the
     $f$-direction of the open cone $\cone(v,\theta)$} if for
   $P$-a.e.~$\gamma\in\frags(X)$ and
   $\lebmeas\mrest\dom\gamma$-a.e.~$t$ one has $(f\circ\gamma)'(t)\in\cone(v,\theta)$.
 \end{defn}
\par For the purpose of this paper it will be convenient to obtain
Alberti representations with biLipschitz constants close to $1$. We
will thus use the following result \cite[Thm.~2.64]{deralb}:
\begin{thm}
   \label{thm:alberti_rep_prod}
   Let $X$ be a complete separable metric space and $\mu$ a Radon
   measure on $X$. Then the following are equivalent:
   \begin{enumerate}
   \item The measure $\mu$ admits an Alberti representation $\albrep.$
     in the
     $f$-direction of $\cone(v,\theta)$ with $g$-speed $>\sigma$;  \vskip2mm
   \item For each $\varepsilon>0$ the measure $\mu$ admits a
     $(1,1+\varepsilon)$-biLipschitz Alberti representation $\albrep.$
     in the
     $f$-direction of $\cone(v,\theta)$ with $g$-speed $>\sigma$.  
   \end{enumerate}
   Moreover, one can always assume that the Alberti representation is
   of the form $\albrep.=(P,\nu)$, where $P$ is a finite Radon measure
   concentrated on the set of fragments with compact
   domain. Additionally, one can assume that $\nu=h\Psi$ where $h$ is
   a nonnegative Borel function of $X$ and:
   \begin{equation}
     \label{eq:alberti_rep_prod_s1}
     \Psi_\gamma=\mpush\gamma.\left(\lebmeas\mrest\dom\gamma\right),
   \end{equation}
   i.e.~the push-forward of the restriction of the Lebesgue measure to
   the domain of $\gamma$.
 \end{thm}
\par Sometimes we will find it useful to \textbf{restrict an Alberti
representation $\albrep.=(P,\nu)$} to a Borel
set  $U\subset X$ by letting $\albrep.\mrest
 U=(P,\nu\mrest U)$. Other times one knows the existence of Alberti
 representations on subsets $\{U_\alpha\}_\alpha$ and would like to
 glue them together. This is accomplished by the following
 \emph{gluing principle} \cite[Thm.~2.46]{deralb}:
 \begin{thm}
   \label{thm:alberti_glue}
   Let $\{U_\alpha\}_\alpha$  be Borel subsets and suppose that for
   each $\alpha$ the measure $\mu\mrest U_\alpha$ admits a
   $(C,D)$-biLipschitz Alberti representation in the $f$-direction of
   $\cone(v,\theta)$ with $f$-speed $\geq\sigma$ (or $>\sigma$); then
   the measure $\mu\mrest\bigcup_\alpha U_\alpha$ also admits a
   $(C,D)$-biLipschitz Alberti representation in the $f$-direction of
   $\cone(v,\theta)$ with $f$-speed $\geq\sigma$ (or $>\sigma$).
 \end{thm}
 \subsection{Results from Bate and Speight}
 \label{subsec:result_bate_speight}
 We now recall some results \cite{bate_speight,bate-diff} on the
 structure of measures in differentiability spaces. The original
 Theorems \cite{cheeger,keith} on the existence of differentiable
 structures required the measure $\mu$ to be doubling. Bate and
 Speight \cite{bate_speight} found  a partial converse of this:
 \begin{thm}
   \label{thm:as_doubling}
   If $(X,\mu)$ is a differentiability space, then:
   \begin{itemize}
   \item The measure $\mu$ is
   \textbf{asymptotically doubling}, i.e.~for $\mu$-a.e.~$x$ there are
   $(C_x,r_x)\in(0,\infty)^2$ such that:
   \begin{equation}
     \label{eq:as_doubling}
     \mu\left(\ball
       x,2r.\right) \le C_x
     \mu\left(\ball x,r.\right)\quad (\forall r\le r_x).
   \end{equation}
   As a consequence, $(X,\mu)$ is a
   \emph{Vitali space}, i.e.~the Vitali Covering Theorem
   holds in $(X,\mu)$,
   and thus also Lebesgue's Differentiation Theorem holds
   for $\mu$.
   \item Every  porous subset is $\mu$-null.
   \end{itemize}
   
 \end{thm}
 
 It was shown in \cite[Lemma 8.3]{bate_jams} that if
 if $(X,\mu)$ is asymptotically doubling, there are
   countably many Borel sets $\{U_\alpha\}_\alpha$ such that
   $\mu(X\setminus\bigcup_\alpha U_\alpha)=0$ and such that each
   $U_\alpha$ is doubling as a metric space. Moreover, the sets $\{U_\alpha\}_\alpha$
   might be assumed to be closed or compact.

 \begin{remark}
   \label{rem:intro_blow_ups}
 In particular, at generic points of each $U_\alpha$,
 one can obtain blow-ups/tangent cones of $(X,\mu)$ by using Gromov's
 Compactness Theorem (see Section \ref{sec:geom_blowups}). In fact, as
 porous sets are $\mu$-null, blowing-up $(U_\alpha,\mu\on U_\alpha)$
 at a point $p$ of $U_\alpha$ which is a Lebesgue density point for $\mu$,
 and is also a point at which $U_\alpha$ is not porous in the ambient space $X$,
 will yield the same metric measure spaces as blowing-up $(X,\mu)$.
\end{remark}

 \par Recently Bate \cite{bate-diff} made a deep study of the
 structure of measures in differentiability spaces by using Alberti
 representations; in particular, he was able to obtain several
 characterizations of these spaces. For the sake of brevity we just
 summarize one characterization as follows:
 \begin{thm}
   \label{thm:bate_char}
   The metric measure space $(X,\mu)$ is a differentiability space if
   and only if:
   \begin{enumerate}
   \item The measure $\mu$ is asymptotically doubling and porous sets
     are $\mu$-null;\vskip2mm
   \item There is a Borel function $\tau:X\to(0,\infty)$ such that,
     for each real-valued Lipschitz function $f$, the measure $\mu$
     admits an Alberti representation with $f$-speed $\geq\tau\biglip f$.
   \end{enumerate}
 \end{thm}
 In \cite{deralb} it was shown that one may take $\tau=1$: in
 subsection \ref{subsec:Liplip} we provide a proof of this fact which
 is independent of the results in \cite{deralb}. To put this in
 perspective we recall the following definition:
 \begin{defn}
   \label{defn:Liplip}
   Let $f:X\to\real$ be Lipschitz. The \textbf{lower pointwise Lipschitz
     constant of $f$ at $x$} is:
   \begin{equation}
     \label{eq:Liplip1}
     \smllip
     f(x)=\liminf_{r\searrow0}\sup\left\{\frac{\left|f(y)-f(x)\right|}{r}:\xdst
         x,y.\le r\right\}.
     \end{equation}
   \end{defn}
   In \cite{keith} it was shown that the existence of a measurable
   differentiable structure follows under the assumption that
   $(X,\mu)$ satisfies a \textbf{Lip-lip inequality}: this means that
   there is a $K\geq1$ such that, for each real-valued Lipschitz
   function $f$, one has:
   \begin{equation}
     \label{eq:Liplip_ineq}
     \biglip f(x)\le K\smllip f(x)\quad(\text{for $\mu$-a.e.~$x$}).
   \end{equation}
   In particular, Theorem \ref{thm:bate_char} implies that in a
   differentiability space the Lip-lip inequality holds by
   replacing the constant $K$ with the function $\tau$; thus, showing that one can
   take $\tau=1$ implies that the Lip-lip inequality self-improves to
   an equality. For the case of PI-spaces, the Lip-lip equality was a
   main result of \cite{cheeger}, which followed from the more general
   result that, for $p>1$, $\biglip f$ is a representative of the
   \emph{minimal generalized upper gradient} of $f$.
   \par The result of \cite{bate-diff} that we will mainly use is the
   existence of Alberti representations in the directions of arbitrary cones:
   \begin{thm}
  \label{thm:bate_alb_arb_dir}
  Let $(U,\psi)$ be an $N$-dimensional differentiability chart for
    the differentiability space $(X,\mu)$; then for each
    $v\in\mathbb{S}^{N-1}$ and each $\theta\in(0,\pi/2)$, the measure
    $\mu\mrest U$ admits an Alberti representation in the
    $\psi$-direction of\/ $\cone(v,\theta)$.
\end{thm}

 \subsection{Measurable Vector Bundles}
\label{subsec:meas_vec_bund}

In this paper we will work with measurable subbundles of the tangent
and cotangent bundles associated to a differentiability space. Since
we deal with different (measurable) seminorms on these subbundles, we
need to introduce a bit of terminology to make the treatment
precise. Let $(\Omega,\Sigma)$ be a measurable space; a
\def\bs{\V}
\textbf{$\Sigma$-measurable vector bundle over $\Omega$} is a quadruple
\def\bi{I_{\bs}}
\def\bcap{I_{\bs,\cap}}
\def\gtrans#1.{g_{\setbox0=\hbox{$#1\unskip$}\ifdim\wd0=0pt \alpha,\beta
    \else #1\fi}}
\def\glingroup#1.{\text{\normalfont GL}\left({\setbox0=\hbox{$#1\unskip$}\ifdim\wd0=0pt \real^{N_\alpha}
    \else #1\fi}\right)}
$(\bi,\{N_\alpha\}_{\alpha\in\bi},\{U_\alpha\}_{\alpha\in\bi},\{\gtrans.\}_{(\alpha,\beta)\in\bcap})$
such that:
\begin{enumerate}
\item The index set $\bi$ is countable and
  $\{U_\alpha\}_{\alpha\in\bi}$ is a cover of $\Omega$ consisting of
  $\Sigma$-measurable sets;
  \vskip2mm
\item Each $N_\alpha$ is a nonnegative integer and if $U_\alpha\cap
  U_\beta\ne\emptyset$, then $N_\alpha=N_\beta$;
    \vskip2mm
\item The (possibly empty set) $\bcap$ consists of those pairs
  $(\alpha,\beta)\in\bi\times\bi$ such that $U_\alpha\cap
  U_\beta\ne\emptyset$;
    \vskip2mm
\item Each $\gtrans.$ is a $\Sigma$-measurable map
  $\gtrans.:U_\alpha\cap U_\beta\to\glingroup.$.
\end{enumerate}
If $N=\sup_\alpha N_\alpha<\infty$ the bundle $\bs$ is said \textbf{to
  have finite dimension $N$}.
\par A \textbf{section $\sigma$ of $\bs$} is a collection
$\{\sigma_\alpha\}_{\bi}$ of $\Sigma$-measurable maps
$\sigma_\alpha:U_\alpha\to\real^{N_\alpha}$ such that:
\begin{equation}
  \label{eq:sec_compa}
  \gtrans.\circ\sigma_\alpha=\sigma_\beta.
\end{equation}
\par A measurable subbundle of $\bs$ is a measurable choice of a
hyperplane in each fibre. \def\grass#1,#2.{\text{\normalfont Gr}\left({\setbox0=\hbox{$#1\unskip$}\ifdim\wd0=0pt \real^{N_\alpha}
      \else #1\fi},{\setbox1=\hbox{$#2\unskip$}\ifdim\wd1=0pt M_\alpha{} \else #2\fi}\right)}
More precisely, let $\grass \real^N,k.$ denote the Grassmanian of unoriented $k$-dimensional
planes in $\real^{N}$; then a \textbf{subbundle} \def\sbs{\W}
$\sbs$ of $\bs$ is a pair $(\{M_\alpha\}_{\bi},\{\phi_\alpha\}_{\bi})$ such that:
\begin{enumerate}
\item Each nonnegative integer $M_\alpha$ satisfies $M_\alpha\le
  N_\alpha$ and if $(\alpha,\beta)\in\bcap$, then $M_\alpha=M_\beta$;\vskip2mm
\item Each $\phi_\alpha$ is a $\Sigma$-measurable map
  $\phi_\alpha:U_\alpha\to\grass,.$;\vskip2mm
\item For each pair $(\alpha,\beta)\in\bcap$ the following
  compatibility condition holds:
  \begin{equation}
    \label{eq:subb_comp}
    \gtrans.\left(\phi_\alpha(x)\right)=\phi_\beta(x)\quad(\forall x\in
    U_\alpha\cap U_\beta).
  \end{equation}
\end{enumerate}
\par We now turn to the construction of seminorms on $\bs$ (or on a
subbundle). Let \def\semnrm#1.{\text{\normalfont Sem}\left({\setbox0=\hbox{$#1\unskip$}\ifdim\wd0=0pt \real^{N_\alpha}
      \else #1\fi}\right)}
\def\gsemnrm#1.{\text{\normalfont Sem}_{+\infty}\left({\setbox0=\hbox{$#1\unskip$}\ifdim\wd0=0pt \real^{N_\alpha}
      \else #1\fi}\right)}
$\semnrm \real^N.$ denote the set of seminorms on $\real^N$, and let
$\gsemnrm.=\semnrm.\cup\{+\infty\}$, which is viewed as the
one-point compactification of $\semnrm.$. The element $+\infty$ is
interpreted as the real-valued function on $\real^N$ which assigns
value $+\infty$ to any non-zero vector and value $0$ to the null vector. A
seminorm (resp.~a generalized seminorm)
$\gennrmname$ on $\bs$ is a collection $\{\alnrmname\}_{\alpha\in\bi}$
of $\Sigma$-measurable maps $\alnrmname:U_\alpha\to\semnrm.$
(resp.~$\gsemnrm.$) which
satisfy, for each $(\alpha,\beta)\in\bcap$ and each
$v\in\real^{N_\alpha}$, the following compatibility condition:
\begin{equation}
  \label{eq:sem_comp}
  \alnrm v.(x)=\benrm{\gtrans.(v)}.(x)\quad(\forall x\in U_\alpha\cap U_\beta).
\end{equation}
\par We will essentially work with measurable bundles where
$\Omega=X$, a complete separable metric space, and where $\Sigma$ is
the Borel $\sigma$-algebra. However, in the case of a metric measure
space $(X,\mu)$, we implicitly identify vector bundles, sections and
seminorms which agree $\mu$-a.e. For example, consider two Borel vector
bundles
$\bs=(\bi,\{N_\alpha\}_{\alpha\in\bi},\{U_\alpha\}_{\alpha\in\bi},\{\gtrans.\}_{(\alpha,\beta)\in\bcap})$
and 
$\bs'=(\bi',\{N'_{\alpha'}\}_{\alpha'\in\bi'},\allowbreak\{U'_{\alpha'}\}_{\alpha'\in\bi'},\{\gtrans.'\}_{(\alpha',\beta')\in\bcap'})$ over $X$; we identify them if:
\begin{enumerate}
\item Whenever $\mu(U_\alpha\cap U'_{\alpha'})>0$ one has
  $N_\alpha=N'_{\alpha'}$;
    \vskip2mm
\item Whenver $\mu(U_\alpha\cap U'_{\alpha'})>0$ there are a
  $\mu$-full measure subset $V_{\alpha,\alpha'}\subset U_\alpha\cap
  U'_{\alpha'}$ and a Borel map
  $G_{\alpha,\alpha'}:V_{\alpha,\alpha'}\to\glingroup.$, such that,
  if $\mu(U_\beta\cap U'_{\beta'})>0$, one has:
  \begin{equation}
    \label{eq:equiv_bundles}
    G_{\beta,\beta'}\circ\gtrans.(x)=\gtrans\alpha',\beta'.\circ
    G_{\alpha,\alpha'}(x)\quad\text{(for $\mu$-a.e.~$x\in
      V_{\alpha,\alpha'}\cap V_{\beta,\beta'})$}.
  \end{equation}
\end{enumerate}
\par To construct seminorms on measurable vector bundles we will use
often the following lemma.
\begin{lemma}
  \label{lem:sup_of_seminorms}
  Let $\bs$ be a measurable vector bundle over $X$ and let $\left\{\tanrmname\right\}_{\tau\in T}$ be a
  countable collection of seminorms on $\bs$. Then for $x\in U_\alpha$ and
  $v\in\real^{N_\alpha}$ let
  \begin{equation}
    \label{eq:sup_of_seminorms_s2}
    \left\|v\right\|_{T,\alpha}(x)=\sup_{\tau\in T}\left\|v\right\|_{\tau,\alpha}(x);
  \end{equation}
  then $\left\{\|\,\cdot\,\|_{T,\alpha}\right\}_{\alpha\in\bi}$
  defines a seminorm $\Tnrmname$ on $\bs$, which we call the
  \textbf{supremum of the seminorms $\left\{\Tnrmname\right\}_{\tau\in
      T}$}. Moreover, suppose that there are a seminorm
  $\gennrmname$ on $\bs$  and a 
  $C\ge0$ such that:
  \begin{equation}
    \label{eq:sup_of_seminorms_s1}
    \tanrmname\le C\gennrmname
  \end{equation}
  holds uniformly in $\tau$.  Then $\Tnrmname$ is a seminorm on $\bs$
  and one has:
  \begin{equation}
    \label{eq:sup_of_seminorms_s3}
    \Tnrmname\le C\gennrmname.
  \end{equation}
  \par Suppose now that $\mu$ is a $\sigma$-finite Borel measure on
  $X$ and let $\{\omnrmname.\}_{\omega\in\Omega}$ be a collection of
  seminorms on $\bs$ which is allowed to be uncountable. Then there is
  a $\mu$-a.e.~unique generalized seminorm $\Omnrmname$, called the
  \textbf{essential supremum} of the collection
  $\{\omnrmname.\}_{\omega\in\Omega}$, which satisfies the following
  properties:
  \begin{description}
    \item[(Ess-sup1)] For each section $\sigma$ of $\bs$ and each
      $\omega\in\Omega$ one has:
  \begin{equation}
    \label{eq:sup_of_seminorms_s4}
    \Omnrm\sigma.\ge\omnrm\sigma.\quad{\text{$\mu$-a.e.;}}
  \end{equation}
\item[(Ess-sup2)] If $\Omnrmname'$ is another generalized seminorm
  satisfying (\ref{eq:sup_of_seminorms_s4}), then one has:
  \begin{equation}
    \label{eq:sup_of_seminorms_s5}
    \Omnrmname'\ge\Omnrmname\quad{\text{$\mu$-a.e.}}
  \end{equation}
\end{description}
Moreover, if there are a seminorm
  $\gennrmname$ on $\bs$  and a 
  $C\ge0$ such that:
  \begin{equation}
    \label{eq:sup_of_seminorms_s6}
    \omnrmname\le C\gennrmname
  \end{equation}
  holds $\mu$-a.e.~and uniformly in $\omega$, then $\Omnrmname$ can be taken to be a
  seminorm satisfying:
  \begin{equation}
    \label{eq:sup_of_seminorms_s7}
    \Omnrmname\le C\gennrmname\quad\text{$\mu$-a.e.}
  \end{equation}
\end{lemma}
\begin{proof}
  The proof that $\Tnrmname$ defines a generalized seminorm, which is
  also a norm under the additional assumption
  (\ref{eq:sup_of_seminorms_s1}), follows by unwinding the definition
  of a measurable vector bundle. To prove the second part of Lemma
  \ref{lem:sup_of_seminorms} we use the approach of
   \cite[Prop.~5.4.7]{elliott_fin}.
  \par We first observe that \textbf{(Ess-sup1)} and
  \textbf{(Ess-sup2)} are properties that hold up to $\mu$-null sets,
  and thus we can construct $\Omnrmname$ independently on each
  $\bs|U_\alpha$, where $\bs|U_\alpha$ denotes the union of the fibres
  of $\bs$ over the points $x\in U_\alpha$. We can therefore assume
  that the cardinality of $\bi$ is one and identify $\bs$ with the
  product $U\times \real^N$. Without loss of generality we can also
  assume that $\mu$ is a probability measure on $U$. We take a norm
  $\emptynrmname$ on $\real^N$, and denote by $S^{N-1}$ and $\hmeas
  N-1.$ the corresponding unit ball and $(N-1)$-dimensional Hausdorff
  measure. We finally let $\pi$ be the probability measure
  \begin{equation}
    \label{eq:sup_of_seminorms_p1}
    \mu\otimes\hmeas N-1.\mrest S^{N-1}/\hmeas N-1.(S^{N-1}).
  \end{equation}
  We also observe that, by possibly increasing $\Omega$, we can assume
  that the collection $\{\omnrmname.\}_{\omega\in\Omega}$ is
  \emph{upward-filtering}, i.e.~for all pairs$(\omega,\omega')\in \Omega^2$
  there is an $\omega''\in\Omega$ satisfying:
  \def\fknrm{\left\|\,\cdot\,\right\|}
  \begin{equation}
    \fknrm_{\omega''}=\max\left\{\fknrm_{\omega},\fknrm_{\omega'}\right\}.
  \end{equation}
  \par We now consider the increasing homeomorphism:
  \begin{equation}
    \label{eq:sup_of_seminorms_p2}
    \begin{aligned}
      \varphi:&\real\to(0,1)\\
      t&\mapsto\frac{e^t}{e^t+1},
    \end{aligned}
  \end{equation}
  and observe that the random variables
  $\left\{\varphi\left(\omnrmname\right)\right\}_{\omega\in\Omega}$
  are all nonnegative and have $\pi$-expectations satisfying:
  \begin{equation}
    \label{eq:sup_of_seminorms_p3}
    E\left[\varphi\left(\omnrmname\right)\right]\le 1.
  \end{equation}
  Thus the supremum:
  \begin{equation}
    \label{eq:sup_of_seminorms_p4}
    q=\sup\left[E\left[\varphi\left(\omnrmname\right)\right]: \omnrmname\in\Omega\right\}
  \end{equation}
  is finite, and we let $T=\{\fknrm_{\omega_n}\}$ denote a maximizing
  sequence:
  \begin{equation}
    \label{eq:sup_of_seminorms_p5}
    \lim_{n\to\infty}E\left[\varphi\left(\fknrm_{\omega_n}\right)\right]=q.
  \end{equation}
  The proof is completed by showing that $\Tnrmname$ satisfies
  \textbf{(Ess-sup1)} and \textbf{(Ess-sup2)}.
  \par We first address \textbf{(Ess-sup1)}: suppose that one has
  $\omnrmname>\Tnrmname$ on a set of positive measure. Then,
  considering the sequence of norms
  $\left\{\max\left\{\fknrm_{\omega_n},\omnrmname\right\}\right\}\subset
  \Omega$
  one contradicts (\ref{eq:sup_of_seminorms_p5}).
  \par We now address \textbf{(Ess-sup2)} and take a norm
  $\Omnrmname'$ satisfying (\ref{eq:sup_of_seminorms_s4}). Let
  $A\subset U$ be a set of positive $\mu$-measure: we claim that
  one has:
  \begin{equation}
    \label{eq:sup_of_seminorms_p6}
    \lim_{n\to\infty}E\left[\chi_{A\times\real^N}\varphi\left(\fknrm_{\omega_n}\right)\right]
    =\sup\left\{E\left[\chi_{A\times\real^N}\varphi\left(\omnrmname\right)\right]:\omega\in\Omega\right\}
    =E\left[\chi_{A\times\real^N}\varphi\left(\Tnrmname\right)\right].
  \end{equation}
  In fact, if any of the equalities in (\ref{eq:sup_of_seminorms_p6})
  failed, using that $\varphi$ is positive and that the collection $\{\omnrmname.\}_{\omega\in\Omega}$ is
  upward-filtering, one would contradict
  (\ref{eq:sup_of_seminorms_p5}). As $\varphi$ is increasing, we have
  \begin{equation}
    \label{eq:sup_of_seminorms_p7}
    E\left[\chi_{A\times\real^N}\varphi\left(\Omnrmname'\right)\right]\ge E\left[\chi_{A\times\real^N}\varphi\left(\fknrm_{\omega_n}\right)\right],
  \end{equation}
  and from (\ref{eq:sup_of_seminorms_p6}) it follows that:
  \begin{equation}
    \label{eq:sup_of_seminorms_p8}
    E\left[\chi_{A\times\real^N}\varphi\left(\Omnrmname'\right)\right]\ge
    E\left[\chi_{A\times\real^N}\varphi\left(\Tnrmname\right)\right],
  \end{equation}
  from which (\ref{eq:sup_of_seminorms_s5}) follows.
\end{proof}
\section{Generic points and generic velocities}
\label{sec:gen_points_gen_veloc}
In this Section we fix a complete separable metric space $X$ and
introduce a notion of genericity for pairs \def\dset{D_X}
\def\gfun{\F}\def\gbor{\C}\def\gsem{\S}\def\gtrip{(\gfun,\gbor,\gsem)}\def\gpair{(\gfun,\gbor)}\def\gquad{(\gfun,\gbor,\gsem,\dset)}
\def\gpairset{(\gfun,\gbor,\dset)}
$(\gamma,t)\in\frags(X)\times\real$; this notion of genericity will be
specified in terms of a quadruple $\gquad$ such that: $\gfun$ is a
countable collection of real-valued Lipschitz functions defined on
$X$, $\gbor$ is a countable collection of real-valued bounded Borel
functions defined on $X$, $D_X$ is a countable dense subset of $X$, and $\gsem$ is a countable collection of
Lipschitz compatible
pseudometrics on $X$ which will always include the metric $\xdname$.

\begin{definition}
  \label{def_generic_pair}
We say that the pair  $(\ga,t)$  is $\gquad$-{\bf generic} if:
 \begin{description}
\item[(Gen1)] The point $t$ is a Lebesgue density point of $\dom(\ga)$;\vskip2mm
\item[(Gen2)] For each $f\in\gfun$ the derivative $(f\circ \ga)'$ exists
  and is  approximately continuous at $t$;\vskip2mm
\item[(Gen3)] For each $u\in\gbor$ the function $u\circ\gamma$ is
  approximately continuous at $t$;\vskip2mm
\item[(Gen4)] For each $x\in\dset$ and each $\rdname\in\gsem$ the
  derivative $(\rdstp x.\circ\gamma)'$ exists
  and is  approximately continuous at $t$;\vskip2mm
\item[(Gen5)] For each $\rdname\in\gsem$ the function
  $\sup_{x\in\dset}\left|(\rdstp x.\circ\gamma)'(t)\right|$ is approximately
  continuous at $t$.
\end{description}
In the case in which $\gsem$ consists only of $\xdname$ we will just
write $\gpairset$. Whenever a default choice of the set $\dset$ is assumed, we
will omit $\dset$ from the notation.
\end{definition}
\begin{remark}
We remark that the proof of \cite[Thm.~4.1.6]{ambrosio_tilli} shows
that at a point $t$ where \textbf{(Gen1)}, \textbf{(Gen4)} and
\textbf{(Gen5)} hold, the $\varrho$-metric derivative $\rmd\gamma(t)$
exists and equals $\sup_{x\in\dset}\left|(\rdstp
  x.\circ\gamma)'(t)\right|$. Thus, if $(\gamma,t)$ is
$\gquad$-generic, for each $\rdname\in\gsem$ the $\varrho$-metric
derivative exists and is approximately continuous at $t$.
\end{remark}
\par We point out that, in the case
of a differentiability space $(X,\mu)$, Definition
\ref{def_generic_pair} has a natural interpretation in terms of the
$\mu$-tangent bundle $TX$. Let $\left\{(U_\alpha,\phi_\alpha)\right\}$
be an atlas for $(X,\mu)$ and suppose that $\gfun$ contains the components
of all the coordinate functions $\{\phi_\alpha\}$, and that $\gbor$
contains all the characteristic functions
$\{\chi_{U_\alpha}\}$.  Suppose now that $(\gamma,t)$ is
$\gquad$-generic and that $\gamma(t)\in\bigcup_\alpha U_\alpha$; then
$\gamma'(t)$ is a well-defined element of $TX$. We are thus led to the
following definition.
\begin{definition}
  \label{def:gen_vel}
A $\gquad$-{\bf generic velocity vector} is an element of $TX$ of the form $\ga'(t)$,
where $(\ga,t)$ is $\gquad$-generic and  $\ga(t)\in
\cup_\alpha\,U_\alpha$.  As above, in the case in which $\gsem$ consists only of $\xdname$, we will just
write $\gpairset$, and we will omit $D_X$ from the notation if a default
choice of the set $D_X$ is assumed.
\end{definition}  

We now establish measurability for generic pairs.
\begin{lemma}\label{lem:generic_borel}
  The set \begin{equation}
    G\gquad=\left\{(\ga,t): \text{\normalfont $(\gamma,t)$ is $\gquad$-generic}\right\}
  \end{equation} is a Borel subset of $\frag(X)\times\R$.
\end{lemma}
\begin{proof} \def\Domset{\text{\normalfont DOM}}
  \def\Isolset{\text{\normalfont ISOL}}
  \def\Mdiffset{\text{\normalfont MDIFF}}
  \def\Mdiffmap{\text{\normalfont MDer}}
  \def\Diffset{\text{\normalfont DIFF}}
  \def\Diffmap{\text{\normalfont Der}}
  \def\Subset{\text{\normalfont SUB}}
  \def\Lebmap{\text{\normalfont Leb}}
  \def\Appset{\text{\normalfont ACONT}}
  \def\Eval{\text{\normalfont Ev}}
  \def\Lebset{\text{\normalfont LEBDENS}}
  We prove that $G\gquad$ is Borel by showing that certain sets are
  Borel. Let $\Domset$ denote the set of pairs $(\gamma,t)$ such that
  $t\in\dom\gamma$:
  \begin{equation}
    \label{eq:generic_borel_p1}
    \Domset=\left\{(\gamma,t)\in\frags(X)\times\real: t\in\dom\gamma\right\};
  \end{equation}
  then $\Domset$ is closed. Fix $\delta>0$ and consider the set of
  pairs $(\gamma,t)$ where $t$ becomes isolated below scale $\delta$:
  \begin{equation}
    \label{eq:generic_borel_p2}
    \Isolset(\delta)=\left\{(\gamma,t)\in\Domset:
      \text{$\dom\gamma\cap(t-\delta,t+\delta)$ contains only one point}\right\};
  \end{equation}
  then $\Isolset(\delta)$ is closed and
  $\Isolset=\bigcup_{\delta\in\rational_{>0}}\Isolset(\delta)$ is Borel
  and consists of the pairs $(\gamma,t)$ where $t$ is an isolated
  point of $\dom\gamma$. We can thus attempt to define, for a
  Lipschitz compatible pseudometric $\varrho$, the $\varrho$-metric
  derivative and, for $f$ Lipschitz, the derivative of $f$ at 
  pairs in $\Domset\setminus\Isolset$. Consider the set:
  \begin{equation}
    \label{eq:generic_borel_p3}
    \begin{split}
    \Mdiffset(\varrho)&=\left\{(\gamma,t)\in\Domset\setminus\Isolset:\text{
        $\rmd\gamma(t)$ exists}\right\}\\
    &=\bigcap_{\varepsilon\in\rational_{>0}}\bigcup_{(\delta,\theta)\in\rational_{>0}\times\rational_{\ge0}}\biggl\{(\gamma,t)\in\Domset\setminus
    \Isolset:\forall s_1,s_2\in(t-\delta,t+\delta)\cap\dom\gamma, \\
    &\left|\rdst\gamma(s_1),\gamma(s_2).-\theta|s_1-s_2|\right|\le\varepsilon|s_1-s_2|\biggr\};
  \end{split}
\end{equation}
this set is Borel as all the sets in the curly brackets are closed in
$\Domset\setminus\Isolset$. Modifying the definition of $\Mdiffset$ by
constraining $\theta$ to lie in a specified interval we also conclude
that the map:
\begin{equation}
  \label{eq:generic_borel_p3bis}
  \begin{aligned}
    \Mdiffmap(\varrho):\Mdiffset&\to[0,\infty)\\
    (\gamma,t)&\mapsto\rmd\gamma(t)
  \end{aligned}
\end{equation}
is Borel. Consider now a real-valued Lipschitz function $f$ defined on
$X$; the set $\Diffset(f)$ where $(f\circ\gamma)'(t)$ exists is Borel
because we can write it as:
\begin{equation}
  \label{eq:generic_borel_p4}
  \begin{split}
    \Diffset(f)&=\left\{(\gamma,t)\in\Domset\setminus\Isolset:\text{
        $f\circ\gamma$ is differentiable at $t$}\right\}\\
      &=\bigcap_{\varepsilon\in\rational_{>0}}\bigcup_{(\delta,\theta)\in\rational_{>0}\times\rational}\biggl\{(\gamma,t)\in\Domset\setminus
    \Isolset:\forall s_1,s_2\in(t-\delta,t+\delta)\cap\dom\gamma, \\
    &\left|f\circ\gamma(s_1)-f\circ\gamma(s_2)-\theta(s_1-s_2)\right|\le\varepsilon|s_1-s_2|\biggr\}
  \end{split}
\end{equation}
where the sets in curly brackets are closed in
$\Domset\setminus\Isolset$. Constraining the $\theta$ appearing in the
definition of $\Diffset(f)$ to lie in a given interval we conclude
that the map:
\begin{equation}
  \label{eq:generic_borel_p5}
  \begin{aligned}
    \Diffmap(f):\Diffset(f)&\to\real\\
    (\gamma,t)&\mapsto(f\circ\gamma)'(t)
  \end{aligned}
\end{equation}
is Borel. Regarding condition \textbf{(Gen5)} we need also to take a
sup of derivatives when they exist; so let $\Omega$ be a countable set
of Lipschitz functions; then the set:
\begin{equation}
  \label{eq:generic_borel_p6}
  \Diffset(\Omega)=\bigcap_{f\in\Omega}\Diffset(f)
\end{equation}
and the map:
\begin{equation}
  \label{eq:generic_borel_p7}
  \begin{aligned}
    |\Diffmap(\Omega)|:\Diffset(\Omega)&\to\real\\
    (\gamma,t)&\mapsto\sup_{f\in\Omega}\left|(f\circ\gamma)'(t)\right|
  \end{aligned}
\end{equation}
are Borel.
\par We now turn to questions pertaining to the approximate
continuity of a function at a point in the domain of a fragment. For $L\ge0$ we will denote by $\Subset(L)$ the closed set of
those fragments whose domain lies in $[-L,L]$. Suppose now that we are
given a Borel set $B\subset\Domset$ and a Borel map $\psi:B\to\real$.
For $(\varepsilon,\delta,L)\in(\rational_{>0})^3$ let:
\begin{equation}
  \label{eq:generic_borel_p8}
  \begin{split}
    \tilde\Psi(\varepsilon,\delta,L,B,\psi)&=\biggl\{(\gamma,t,s)\in\frags(X)\times\real^2:\gamma\in\Subset(L),\\
    &\text{$(\gamma,t),(\gamma,s)\in B$, $|t-s|\le\delta$ and $\left|\psi(\gamma,t)-\psi(\gamma,s)\right|\le\varepsilon$}\biggr\};
  \end{split}
\end{equation}
the set $\tilde\Psi(\varepsilon,\delta,L,B,\psi)$ is Borel and
\cite[Thm.~17.25]{kechris_desc} shows that the map:
\begin{equation}
  \label{eq:generic_borel_p9}
  \begin{aligned}
  \Lebmap(\varepsilon,\delta,L,B,\psi)&:B\to\real\\
  (\gamma,t)&\mapsto\lebmeas\left(\left(\tilde\Psi(\varepsilon,\delta,L,B,\psi)\right)_{(\gamma,t)}\right)
\end{aligned}
\end{equation}
is Borel. It is then easy to prove that the sets of pairs $(\gamma,t)$
where some map is approximately continuous at $t$ is Borel; in fact,
first define:
\begin{equation}
  \label{eq:generic_borel_p10}
  \begin{split}
    \Appset(\psi)&=\bigcup_{L\in\rational_{>0}}\bigcap_{\varepsilon\in\rational_{>0}}\bigcup_{\delta\in\rational_{>0}}
    \bigcap_{r\in\rational_{>0}}\biggl\{(\gamma,t)\in
    B:\gamma\in\Subset(L),\\
    &\text{and for each $r\le\delta$ one has $\Lebmap(\varepsilon,r,L,B,\psi)(\gamma,t)\ge2(1-\varepsilon)r$}\biggr\},
  \end{split}
\end{equation}
which is a Borel set;
then, for example, $\Appset(\Mdiffmap(\xdname))$ consists of the pairs
$(\gamma,t)$ where $\metdiff\gamma$ exists and is approximately
continuous at $t$. In order to handle the approximate continuity for a
Borel map $u:X\to\real$ we introduce the notation $\Eval(u)$ to denote
the Borel map which evaluates $u$ at $\gamma(t)$:
\begin{equation}
  \label{eq:generic_borel_p11}
  \begin{aligned}
    \Eval(u):\Domset&\to\real\\
    (\gamma,t)&\mapsto u\circ\gamma(t).
  \end{aligned}
\end{equation}
Letting $\psi:\Domset\to\real$ to be the function which trivially maps
each pair $(\gamma,t)$ to $0$, we see that $\Appset(\psi)=\Lebset$ is the set
of pairs $(\gamma,t)$ where $t$ is a Lebesgue density point of
$\dom\gamma$.
\par We finally conclude that $G\gquad$ is Borel by observing that:
\begin{equation}
  \label{eq:generic_borel_p12}
  \begin{split}
    G\gquad&=\bigcap_{\varrho\in\gsem}\Appset(\Mdiffmap(\varrho))\cap\bigcap_{f\in\gfun}\Appset(\Diffmap(f))\\
    &\cap\bigcap_{u\in\gbor}\Appset(\Eval(u))\cap\bigcap_{x\in\dset,\varrho\in\gsem}\Appset(\Diffmap(\rdstp
    x.))\\
    &\cap\bigcap_{\varrho\in\gsem}\Appset\left(|\Diffmap|(\{\rdstp x.\}_{x\in\dset})\right)\cap\Lebset.
  \end{split}
\end{equation}
\end{proof}

\section{Metric Differentials and seminorms on $TX$}\label{sec:met_der_frag}
\newcount\maskmet
\maskmet=0
\ifnum\maskmet>0{
  By Ambrosio-Kirchheim \cite{ambrosio_kirchheim}, 
\cite[Thm.~4.1.6]{ambrosio_tilli} the metric derivative for a curve
fragment $\ga$
can be detected by looking at the derivatives  $d_x\circ \ga$ for
a countable family of distance functions at an approximate continuity
point of the supremum $\sup_x|(d_x\circ\ga)'|$ of all these derivatives.
We begin by pointing out that in the context of differentiability spaces, 
Ambrosio-Kirchheim implies that the metric derivative is actually given by a norm on the 
tangent bundle.
Let $(U,\phi)$ be a chart for a metric measure space $(X,d,\mu)$.  Choose a
countable dense subset  $\{x_j\}\subset X$, and let $\F_d$ be the corresponding 
collection of distance functions $\{d_{x_j}\}$.  Now take $V\subset U$ to be the 
full measure Borel subset of $U$ where the functions in $\F_d$ are differentiable with
respect to $\phi$. Let $\F=\F_d\cup\{\phi\}$ and $\C$ arbitrary.
We now define a seminorm $\|\cdot\|_{\F_d}$ 
on the restriction $TX\restr_V$ by
$$
\|v\|_{\F_d} =\sup\{|du(v)|\mid u\in \F_d\}\,.
$$
}\fi
\par In this Section we discuss the first instance of metric
differentiation, Theorem \ref{thm:met_diff_norm}. The point is that in
the presence of a differentiable structure, the $\rhmeas.$-measure of a 
fragment $\gamma$ can be recovered using a seminorm (canonically
associated to $\rdname$) on the tangent
bundle $TX$ associated to the differentiable structure. Let $(X,\mu)$
be a differentiability space with atlas
$\left\{(U_\alpha,\phi_\alpha)\right\}$, and fix a
countable dense set $\dset\subset X$.
\begin{defn}
  \label{defn:diff_set}
  Let $\Phi$ be a countable collection of Lipschitz functions on $X$;
  we say that a Borel subset $V\subset\bigcup_\alpha U_\alpha$ is a
  \textbf{$\Phi$-differentiability} set if:
  \begin{description}
  \item[(Diff1)] The set $V$ has full $\mu$-measure:
    $\mu\left(\bigcup_\alpha U_\alpha\setminus V\right)=0$;\vskip2mm
  \item[(Diff2)] For each $(x,f)\in V\times\Phi$, if $x\in U_\alpha$,
    then $f$ is differentiable at $x$ with respect to the coordinate
    functions $\phi_\alpha$.
  \end{description}
\end{defn}
Let $\Phi_{\dset,\rdname}=\left\{\rdstp x.:x\in\dset\right\}$ and let
$V$ be a $\Phi_{\dset,\rdname}$-differentiability set. Using Lemma
\ref{lem:sup_of_seminorms}, we obtain a seminorm $\dxnrmname$ on $TX$
by defining, for $y\in V$ and $v\in T_yV$:
\begin{equation}
  \label{eq:semidist_seminorm}
  \dxnrm v.=\sup_{x\in\dset}\left|d\rdstp x.\mid_y(v)\right|.
\end{equation}

\begin{theorem}\label{thm:met_diff_norm}
  Let $\gquad$ be as in {\normalfont Section
    \ref{sec:gen_points_gen_veloc}} and let $V$ be a
  $\Phi_{\dset,\rdname}$-differentiability set.  Assume that $\gfun$
  contains all the components of the coordinate functions
  $\phi_\alpha$, that $\gbor$ contains the characteristic functions
  $\{\chi_{U_\alpha}\}_\alpha\cup\{\chi_V\}$, and that $\rdname\in\gsem$. If
  $\gamma'(t)$ is an $\gquad$-generic velocity vector and if
  $\gamma(t)\in V$, then the metric differential $\rmd\gamma(t)$ exists and equals
  $\dxnrm\gamma'(t).$. In particular, if a fragment $\gamma$ lies in
  $V$, we have:
  \begin{equation}
    \label{eq:met_diff_norm_s1}
\rhmeas .(\im\gamma) = \int_{\dom\gamma}\dxnrm\gamma'(t).\,dt.
  \end{equation}
\end{theorem}
\begin{proof}
To fix the ideas suppose that $\gamma(t)\in U_\alpha$. Because of conditions \textbf{(Gen4)}--\textbf{(Gen5)},
the argument in \cite[Thm.~4.1.6]{ambrosio_tilli} implies that 
\begin{equation}
\rmd\gamma(t)=\sup_{x\in\dset}\left|(\rdstp x.\circ\gamma)'(t)\right|;
\end{equation}
as $\gamma(t)\in V$, for each $x\in\dset$ the pseudodistance function
$\rdstp x.$ is differentiable at $\gamma(t)$ with respect to the
coordinate functions $\phi_\alpha$; note also that $\phi_\alpha\circ\gamma$ is differentiable at
$t$ by condition \textbf{(Gen2)}. Thus,
\begin{equation}
  (\rdstp x.\circ\gamma)'(t)=\sum_i\frac{\partial \rdstp
    x.}{\partial\phi^i_\alpha}(\gamma(t))\left(\phi^i_\alpha\circ\gamma\right)'(t)
=d\rdstp x.\mid_{\gamma(t)}(\gamma'(t)),
\end{equation}
which implies
\begin{equation}
\sup_{x\in\dset}\left|(\rdstp x.\circ\gamma)'(t)\right|=\dxnrm\gamma'(t)..
\end{equation}
Formula (\ref{eq:met_diff_norm_s1}) follows from the area
formula (\ref{eq:sem_metric_area}) for the pseudometric $\rdname$ by observing that
for a fragment $\gamma$ which lies in $V$, for $\lebmeas $-a.e.~$t\in\dom\gamma$,
the velocity vector $\gamma'(t)$ is $\gquad$-generic.
\end{proof}\pcreatenrm{local}{\tilde D_X,\rdname}%
In Section \ref{sec:lipmaps_metricdiff} (Theorem
\ref{thm:can_seminorm}) we will show that for different
choices $\dset$ and $\tilde D_X$ of the countable dense set, the seminorms
$\dxnrmname$ and $\localnrmname$ are the same. The proof uses the
density of directions at generic points which is discussed in the next
Section. For the case in which $\rdname=\xdname$ this follows from
Theorem \ref{thm:metdiff=TX}.

\section{Density of generic directions at generic points}\label{sec:dir_dens}

\newcount\maskdensdir
\maskdensdir=0
\ifnum\maskdensdir>0{
\begin{enumerate}
\item Recall Bate  \cite[Theorem 9.5]{bate-diff}.  Or keep it in the preliminaries and
refer to it.
\item State and prove density of generic directions.
\item Discuss alternate proof in PI space case.
\end{enumerate}}\fi
In this Section we show that for $\mu$-a.e.~$x\in X$ the set of
vectors in $T_xX$ which can be represented by $\gquad$-generic
velocity vectors contains a dense set of ``directions'' in $T_xX$. We
make this idea precise with the following definition:
\begin{defn}\label{defn:dens_dir}
  If $V$ is a finite-dimensional vector space, we say that a subset
  $W\subset V$ \textbf{contains a dense set of directions} if:
  \begin{equation}   
\ol{[0,\infty)W}=\ol{\{tw\mid t\in [0,\infty), \;w\in W\}}=V\,.
  \end{equation}
\end{defn}
\par  We now fix an atlas $\{(U_\al,\phi_\al)\}_\alpha$ for the
differentiability space $(X,\mu)$ and let $N_\al$ denote the
dimension of the chart $(U_\al,\phi_\al)$. For each $\al$ let
$\left\{\cone(v_{\al,k},\theta_{\al,k})\right\}_{k\in\N}$ denote a
collection of open cones with
$\{v_{\al,k}\}\subset{\mathbb S}^{N_\al-1}$ dense in the unit
sphere
and
$\lim_{k\to\infty}\theta_{\al,k}=0$. Using
Theorem \ref{thm:bate_alb_arb_dir}, we find Alberti 
representation $\A_k=(P_k,\nu_k)$  of $\mu$ such that, for each $\al$, the restriction
 $\A_k\on U_\al$ is in the $\phi_\al$-direction of
$\cone(v_{\al,k},\theta_{\al,k})$.
\begin{thm}\label{thm:Sus_dir_dens}
Let $\Ga_0\subset\frag(X)$ be a Borel set such that, for each $k$ one has $P_k\left(\frag(X)\setminus\Ga_0\right)=0$;
and let $\gquad$ be as in Definition \ref{def:gen_vel}. 
Then there is
$\mu$-measurable subset $Y\subset X$ with 
 full $\mu$-measure   such that, 
for each
$x\in Y$, the set of velocity vectors
\begin{multline}\label{eq:Sus_dir_dens_s1}
G_x=\biggl\{v\in T_xX\mid \text{\normalfont $v=\ga'(t)$ for
  $\gamma\in\Gamma_0$}\\\text{such that $\gamma'(t)$ is $\gquad$-generic}\biggr\},
\end{multline}
contains a dense set of directions in $T_xX$.
\end{thm}
\begin{proof}
 Let $Z_k\subset X\times\frag(X)\times\R$ consist of those
  triples $(x,\ga,t)$ satisfying:
  \begin{enumerate}
  \item $\gamma'(t)$ is an $\gquad$-generic velocity vector;\vskip2mm
\item $\ga(t)=x$ and $\gamma\in\Ga_0$;\vskip2mm
\item If $\gamma'(t)\in U_\alpha$, then $(\phi_\alpha\circ\ga)'(t)\in\cone(v_{\alpha,k},\theta_{\alpha,k})$.
  \end{enumerate}
Using Lemma \ref{lem:generic_borel} we conclude that 
 $Z_k$ is Borel, and therefore
its projection $Y_k\subset U$ on $X$ is 
Suslin \cite{kechris_desc}, and
hence $\mu$-measurable.
 Note that
for each $\ga\in\Ga_0$, as $\nu_k$ is absolutely continuous with
respect to $\hmeas ._\gamma$, one has $\nu_k(\gamma)(X\setminus
Y_k)=0$, and therefore $\mu(X
\setminus Y_k)=0$. We conclude that $Y=\bigcap_kY_k$ is a $\mu$-full
measure 
$\mu$-measurable
 subset of $X$. Let $x\in Y\cap U_\alpha$, and let
$v\in T_xX$; then
for each $\epsi>0$ we can find a $k$ such that,
for each $w\in\cone(v_{\alpha,k},\theta_{\alpha,k})$, there is a $t_w\in[0,\infty)$ with:
  \begin{equation}
    \|v-t_ww\|_{l^2}\le\epsi\|v\|_{l^2};
  \end{equation}
  but as $x\in Y\cap U_\alpha$, there are a fragment
  $\gamma_k\in\Gamma_0$ and a $t_k\in\R$ such that the vector
  $\gamma_k'(t_k)\in T_xX$ is $\gquad$-generic and
  $(\phi_\alpha\circ\gamma_k)'(t_k)\in\cone(v_{\alpha,k},\theta_{\alpha,k})$;
  thus there is an $s_k\in[0,\infty)$ with
  \begin{equation}
    \|v-s_k(\phi_\alpha\circ\gamma_k)'(t_k)\|_{l^2}\le\epsi\|v\|_{l^2},
  \end{equation} which implies $\ol{[0,\infty)G_x}=T_xX$.
\end{proof}

\section{Consequences of density of generic directions}
\label{sec:conse}
\newcount\maskdens
\maskdens=0
\ifnum\maskdens>0 {
\begin{enumerate}
\item Other norms on $TX$ and $T^*X$.
\item Equality of all norms.
\item $\Lip=\lip$.  Generalization of $\Lip(u)=\rho_u$ to differentiability spaces.
\end{enumerate}
}\fi

In this section we prove the equality
of various seminorms on $TX$ (Theorem \ref{thm_intro_norms_agree}), 
the equality $\Lip u = \lip u$ a.e. (Theorem \ref{thm_Lip_equals_lip}), and give 
a new proof that in PI spaces the minimal generalized upper gradient agrees
with the pointwise Lipschitz constant.

\subsection{Equality of natural seminorms on $TX$}
\label{subsec:eq_nat_norms}
The main result in this subsection is the proof of Theorem \ref{thm_intro_norms_agree},
which is based on the following result.
\begin{theorem}\label{thm:metdiff=TX}
  Let $(X,\mu)$ be a differentiability space and $\dset\subset X$ a
  countable dense set. Then the seminorm $\dxdxnrmname$ on $TX$
  provided by (\ref{eq:semidist_seminorm}) (taking $\rdname=\xdname$)
  coincides with the norm 
  $\|\cdot\|^*_{\Lip}$ (see Section \ref{subsec:stand_ass}); 
  in particular, the norm $\dxdxnrmname$
  does not depend on the choice of $\dset$.
\end{theorem}
{\bf Notation:}
After proving Theorem \ref{thm_intro_norms_agree}, 
we will change to the notation $\|\cdot\|_{TX}$, $\|\cdot\|_{T^*X}$
or simply
$\|\cdot\|$ to denote the canonical norms on $TX$ and $T^*X$.

Theorem \ref{thm:metdiff=TX} can be regarded as an infinitesimal
version of metric differentiation for the identity map ${\rm id}:X\to
X$; its proof uses the following lemma:
\begin{lemma}\label{lem:finitedim_norms}
Suppose that $(V,\|\cdot\|)$ is a finite dimensional normed vector space, with dual
space $(V^*,\|\cdot\|^*)$.  Let $W$ be a subset of the closed unit ball 
$\ol{B(\|\cdot\|)}\subset V$, such that:
\begin{description}
\item[(H1)] For every $w\in W$, there is a linear functional $\al_w\in V^*$ with 
$\|\al_w\|^*\leq 1$, such that $\al_w(w)=1$;\vskip2mm
\item[(H2)] The set $W$ contains a dense set of directions.
\end{description}
Then:
\begin{enumerate}
\item For all $w\in W$ one has $\|\al_w\|^*=1$;\vskip2mm
\item The set $W$ is a dense subset of the unit sphere $S(\|\cdot\|)$;\vskip2mm
\item The seminorm on $V$ defined by 
$
\sup_{w\in W}|\al_w(\cdot)|
$
agrees with $\|\cdot\|$.
\end{enumerate}
\end{lemma}
\begin{proof}
  Note that by \textbf{(H1)} each $\al_w$ has unit norm (which
  implies  (1)) and that each vector
  $w\in W$ has unit norm, which implies that $W\subset S(\|\cdot\|)$. Let $v\in
  S(\|\cdot\|)$; by \textbf{(H2)}, for each $\epsi>0$ there are a $w_\epsi\in W$ and
  a $t_\epsi\in[0,\infty)$:
    \begin{equation}\label{eq:dir_dens}
      \|v-t_\epsi w_\epsi\|\le\epsi;
    \end{equation}
let $\beta_v$ a unit norm functional on $V$ assuming the norm at $v$. Then
\eqref{eq:dir_dens} implies:
\begin{equation}
  |1-t_\epsi\beta_v(w_\epsi)|\le\epsi;
\end{equation} as $|\beta_v(w_\epsi)|\le1$, the previous equation
implies $t_\epsi\ge1-\epsi$. On the other hand, evaluating with
$\alpha_{w_\epsi}$, \eqref{eq:dir_dens} gives
\begin{equation}
  \left|\alpha_{w_\epsi}(v)-t_\epsi\right|\le\epsi;
\end{equation} as the functional $\alpha_{w_\epsi}$ has unit norm,
$t_\epsi\le 1+\epsi$. We thus conclude that
\begin{equation}
  \|v-w_\epsi\|\le\|v-t_\epsi w_\epsi\|+\|(1-t_\epsi)w_\epsi\|\le2\epsi,
\end{equation} implying (2). Note that, as the functionals $\alpha_w$
have unit norm,
\begin{equation}
  \sup_{w\in W}\left\|\alpha_w(\cdot)\right\|\le\|\cdot\|.
\end{equation} On the other hand, for each $v\in V\setminus\{0\}$ and 
each $\epsi>0$, choose $w_\epsi\in W$ with
\begin{equation}
  \left\|\frac{v}{\|v\|}-w_\epsi\right\|\le\epsi; 
\end{equation}
then
\begin{equation}
  \alpha_{w_\epsi}\left(\frac{v}{\|v\|}\right)\ge1-\epsi,
\end{equation}
implying that:
\begin{equation}
  \sup_{w\in W}\left\|\alpha_w(\cdot)\right\|\ge(1-\epsi)\|\cdot\|,
\end{equation} from which (3) follows.
\end{proof}
\begin{proof}[Proof of Theorem \ref{thm:metdiff=TX}]
  Let $V$ be a differentiability set for the countable collection of
  Lipschitz functions $\left\{\xdst\cdot,x.: x\in\dset\right\}$. We
  let $\gfun$ contain the components of the coordinate functions,
  $\gbor$ contain the characteristic functions of the charts, and
  $\gsem=\{\xdname\}$. Let $Y$ be the set provided by Theorem
  \ref{thm:Sus_dir_dens}; we will show that for each $p\in Y\cap V$
  the norm $\tnrmname$ and the seminorm  $\dxdxnrmname$ coincide on
  the fibre $T_pX$. Let $\gamma'(t)\in G_p\gpairset$ with
  $\gamma'(t)\ne0$; then $\md\gamma(t)\ne0$. Without loss of
  generality we assume that $p$ belongs to the chart $U_\alpha$ and we
  consider a functional
  $\sum_{i=1}^{N_\alpha}a_i\,d\phi^i_\alpha\mid_p\in T_p^*X$; then:
  \begin{equation}\label{eq:metdiff=TX_p1}
  \left|\left\langle\sum_{i=1}^{N_\alpha}a_i\,d\phi^i_\alpha\mid_p,\frac{\gamma'(t)}{\md\gamma(t)}\right\rangle\right|=
  \frac{\left|\sum_{i=1}^{N_\alpha}a_i\,\left(\phi^i_\alpha\circ\gamma\right)'(t)\right|}{\md\gamma(t)};
\end{equation}
choose $s_n\searrow0$ such that $t+s_n\in\dom\gamma$ and note that
\begin{equation}
  \label{eq:metdiff=TX_p2}
  \begin{split}
    \left|\sum_{i=1}^{N_\alpha}a_i\,\left(\phi^i_\alpha\circ\gamma\right)'(t)\right|&
    =\lim_{n\to\infty}\frac{\left|\sum_{i=1}^{N_\alpha}a_i\,\left(\left(\phi^i_\alpha\circ\gamma\right)(t+s_n)-\left(\phi^i_\alpha\circ\gamma\right)(t)\right)\right|}{s_n}\\
    &\le\limsup_{n\to\infty}\frac{\left|\sum_{i=1}^{N_\alpha}a_i\,\left(\left(\phi^i_\alpha\circ\gamma\right)(t+s_n)-\left(\phi^i_\alpha\circ\gamma\right)(t)\right)\right|}{\xdst\gamma(t+s_n),\gamma(t).}\\
    &\quad\times\limsup_{n\to\infty}\frac{\xdst\gamma(t+s_n),\gamma(t).}{s_n}\\
    &\le\ctnrm\sum_{i=1}^{N_\alpha}a_i\,d\phi^i_\alpha\mid_p.\,\md\gamma(t);
  \end{split}
\end{equation}
we thus conclude that:
\begin{equation}
  \label{eq:metdiff=TX_p3}
  \left|\left\langle\sum_{i=1}^{N_\alpha}a_i\,d\phi^i_\alpha\mid_p,\frac{\gamma'(t)}{\md\gamma(t)}\right\rangle\right|
  \le\ctnrm\sum_{i=1}^{N_\alpha}a_i\,d\phi^i_\alpha\mid_p.,
\end{equation}
which implies
$\frac{\gamma'(t)}{\md\gamma(t)}\in\ol{B(\tnrmname(x))}$.
\par Let
\begin{equation}
    W_p=\left\{\frac{\gamma'(t)}{\md\gamma(t)}:\text{$\gamma'(t)\ne0$
        and $\gamma'(t)\in G_p\gpairset$}\right\};
  \end{equation}
by Theorem \ref{thm:Sus_dir_dens} the set $W_p$ contains a dense set of
directions in $T_pX$. Theorem \ref{thm:met_diff_norm} implies then
\begin{equation}
  \label{eq:metdiff=TX_p3bis}
  \dxdxnrm\gamma'(t).=\md\gamma(t),
\end{equation}
and so we can find a sequence
$\left\{\sum_{i=1}^{N_\alpha}a_{i,k}\,d\phi^i_\alpha\mid_p\right\}\subset
\ol{B\left(\ctnrmname\right)}$ such that:
\begin{enumerate}
\item We have:
  \begin{equation}
    \label{eq:metdiff=TX_p4}
  \lim_{k\to\infty}\left\langle\sum_{i=1}^{N_\alpha}a_{i,k}\,d\phi^i_\alpha\mid_p,\gamma'(t)\right\rangle=\md\gamma(t);  
  \end{equation}
\item For each $k$ there is an $x_k\in\dset$ with:
  \begin{equation}
    \label{eq:metdiff=TX_p5}
    d\left(\xdst\cdot,x_k.\right)\mid_p=\sum_{i=1}^{N_\alpha}a_{i,k}\,d\phi^i_\alpha\mid_p.
  \end{equation}
\end{enumerate}
By compactness we can find a subsequence of
$\left\{\sum_{i=1}^{N_\alpha}a_{i,k}\,d\phi^i_\alpha\mid_p\right\}$
converging to
$\omega_{\gamma'(t)}\in\ol{B\left(\ctnrmname\right)}$. Now,
(\ref{eq:metdiff=TX_p4}) implies that
\begin{equation}\label{eq:metdiff=TX_p5bis}\omega_{\gamma'(t)}\left(\frac{\gamma'(t)}{\md\gamma(t)}\right)=1;
\end{equation}
 so for $w\in
W_p$ of the form $\frac{\gamma'(t)}{\md\gamma(t)}$ let
$\alpha_w=\omega_{\gamma'(t)}$; applying Lemma
\ref{lem:finitedim_norms} we conclude that:
\begin{equation}
  \label{eq:metdiff=TX_p6}
  \tnrmname(p)=\sup_{w\in W_p}\left|\alpha_w(\cdot)\right|;
\end{equation}
but Lemma \ref{lem:finitedim_norms} implies also that $W_p$ is dense
in $S\left(\tnrmname.\right)$ and by (\ref{eq:metdiff=TX_p5bis}) we
conclude that for $w'\in W_p$:
\begin{equation}
  \label{eq:metdiff=TX_p7}
  \sup_{w\in W_p}\left|\alpha_w(w')\right| = 1 =\dxdxnrm w'.,
\end{equation}
from which we have $\tnrmname(p)=\dxdxnrmname(p).$
\end{proof}

\bigskip\bigskip
\begin{proof}[Proof of Theorem \ref{thm_intro_norms_agree}]
Let $\|\cdot\|_1$-$\|\cdot\|_3$ be the seminorms
as in Theorem \ref{thm_intro_norms_agree}, 
constructed using
Lemma \ref{lem:sup_of_seminorms}, and let $\|\cdot\|_4$ be the 
dual $\Lip$ norm $\|\cdot\|_{\Lip}^*$.   Clearly we have 
$\|\cdot\|_1\leq \|\cdot\|_2\leq
\|\cdot\|_3$.  

We claim that 
\begin{equation}
\label{eqn_3_leq_lip}
\|\cdot\|_3\leq \|\cdot\|_{\Lip}^*\quad \mu-a.e.
\end{equation}
To see this, recall that by Lemma \ref{lem:sup_of_seminorms} there is a
countable collection $\{f_i\}$ of $1$-Lipschitz functions such that for
$\mu$-a.e. $p\in X$, the differentials $df_i(p)\in T_p^*X$ are well-defined, and
$$
\|\cdot\|_3(p)=\sup_i |df_i(p)|\,.
$$
Recalling that for $\mu$ a.e. $p\in X$ we have $\|df_i(p)\|_{\Lip}=\Lip f_i(p)$,
we get that for $\mu$ a.e. $p\in X$, every $i$, and every $v\in T_pX$, 
$$
|df_i(v)|\leq \|df_i\|_{\Lip}\cdot \|v\|_{\Lip}^*
=\Lip f_i(p)\cdot \|v\|_{\Lip}^*\leq \|v\|_{\Lip}^*
$$
since $f_i$ is $1$-Lipschitz.  Taking supremum gives 
(\ref{eqn_3_leq_lip}).

By Theorem \ref{thm:metdiff=TX} we have $\|\cdot\|_1=\|\cdot\|_{\Lip}^*$ $\mu$-a.e.,
so Theorem \ref{thm_intro_norms_agree} follows. 
\end{proof}

\subsection{A new proof of $\smllip f = \biglip f$ in
  differentiability spaces}
\label{subsec:Liplip}
In this subsection we provide a proof, independent of the one given in
\cite{deralb},  of the following result:
\begin{thm}\label{thm:alb_speed_one}
  Let $(X,\mu)$ be a differentiability space and $f:X\to\R$
  Lipschitz. The for all $(\epsi,\sigma)\in(0,1)^2$ there is a $(1,1+\epsi)$-biLipschitz Alberti
  representation of $\mu$ with $f$-speed $\ge\sigma\biglip f$. In
  particular,
  \begin{equation}\label{thm:alb_speed_one_s1}
    \Lip f(x)=\lip f(x)\quad\text{for $\mu$-a.e.~$x$.}
  \end{equation}
\end{thm}
The equality (\ref{thm:alb_speed_one_s1}) generalizes one of the
main results in \cite[Thm.~6.1]{cheeger}, which is a consequence of
the fact that, in a PI-space $(X,\mu)$,
the function $\Lip f$ is a representative of the minimal 
generalized upper
gradient $g_f$ of $f$. This last statement does not make sense in a
general differentiability space as one might have
$g_f<\Lip f$ on a positive measure set, e.g.~because
$X$ might not contain enough curves and one might then have
$g_f=0$. However, in a differentiability space the concept of the \emph{maximal slope of $f$ along fragments
passing at time $t=0$ through $x$} and having $0$ as a density point of
their domain, remains useful and can be interpreted as \emph{the size of the
gradient of $f$}. The result \eqref{thm:alb_speed_one_s1} is also
proven in \cite{deralb} in a conceptually different way, and there it
is also shown that in a differentiability space one has $\Lip f = |df|$, where
$|df|$ is the \emph{local norm} of the form $df$, which is the
Weaver differential form associated to the function $f$. The proof
of Theorem \ref{thm:alb_speed_one} relies
on the following lemma. 
\begin{lemma}\label{cone_diameter}
  Let $\ltwonrmname$ denote the standard $l^2$-norm on $\real^N$, and
  let $\emptynrmname'$ denote another norm on $\real^N$ satisfying:
  \begin{equation}\label{eq:cone_diameter_s1}
    \frac{1}{C}\ltwonrmname\le\emptynrmname'\le C\ltwonrmname.
  \end{equation}
Then the diameter of the set $\cone(v,\theta)\cap
S\left(\emptynrmname'\right)$, with respect to the norm
$\emptynrmname'$, is at most
\begin{equation}\label{eq:cone_diameter_s2}
4C^2(1-\cos\theta+\sin\theta).
\end{equation}
\end{lemma}
\begin{proof}
  Let $v_1,v_2\in\cone(v,\theta)\cap S\left(\emptynrmname'\right)$; then
  we can find
  $u_1,u_2\in S\left(\ltwonrmname\right)$ such that:
 $v_i=\frac{u_i}{\emptynrm u_i.}$;
now
  \begin{equation}
    \ltwonrm u_1-u_2.\le 2(1-\cos\theta+\sin\theta)
  \end{equation}
  by using the definition of Euclidean cone. Observe also that
  (\ref{eq:cone_diameter_s1}) implies:
\begin{equation}
  \left|\emptynrm u_1.'-\emptynrm u_2.'\right|\le\emptynrm
  u_1-u_2.'\le 2C(1-\cos\theta+\sin\theta);
\end{equation}
thus
\begin{equation}
  \begin{split}
    \emptynrm \frac{u_1}{\emptynrm u_1.'} - \frac{u_2}{\emptynrm
      u_2.'}.'&=
    \emptynrm \frac{u_1}{\emptynrm u_1.'}-\frac{u_2}{\emptynrm u_1.'}
    + \frac{u_2}{\emptynrm u_1.'} - \frac{u_2}{\emptynrm u_2.'}.'\\
    &\le\frac{\emptynrm u_1-u_2.'}{\emptynrm u_1.'}+\frac{\emptynrm
      u_2.'}{\emptynrm u_1.'\,\emptynrm u_2.'}\left|\emptynrm
      u_1.'-\emptynrm u_2.'\right|\\
    &\le C\emptynrm u_1-u_2.'+C\left|\emptynrm
      u_1.'-\emptynrm u_2.'\right|\\
    &\le4C^2(1-\cos\theta+\sin\theta).
  \end{split}
\end{equation}
\end{proof}
\begin{proof}[Proof of Theorem \ref{thm:alb_speed_one}]
  We fix an $N$-dimensional chart $(U,\phi)$ and a countable dense set
  $\dset\subset X$. We will show that, for each $(\varepsilon,\sigma)\in(0,1)^2$, the measure $\mu\mrest U$ admits
  a $(1,1+\varepsilon)$-biLipschitz Alberti representation with
  $f$-speed $\geq\sigma\biglip f$; the result about $\mu$ will then
  follow by applying the \emph{gluing principle} Theorem \ref{thm:alberti_glue}.
\par We first consider the special case in which $f$ is of the form
$\langle v_0^*,\phi\rangle$ for some
$v_0^*\in\real^N\setminus\{0\}$. For each $\eta\in(0,1)$ we can use
Egorov and Lusin Theorems to find disjoint compact sets $C_\alpha\in
U$, and dual norms $\alnrmname$ and $\alnrmname^*$ on $\real^N$ such
that:
\begin{equation}
  \label{eq:alb_speed_one_p1}
  \mu\left(U\setminus\bigcup_\alpha C_\alpha\right)=0;
\end{equation}
\begin{equation}
  \label{eq:alb_speed_one_p2}
  \begin{aligned}
    \frac{1}{1+\eta}\tnrmname&\le
    \alnrmname\le(1+\eta)\tnrmname\quad\text{(on the
      fibres of $TX|C_\alpha$);}\\
\frac{1}{1+\eta}\ctnrmname&\le
\alnrmname^*\le(1+\eta)\ctnrmname\quad\text{(on the
      fibres of $T^*X|C_\alpha$).}
  \end{aligned}
\end{equation}
By Theorem \ref{thm:metdiff=TX} we can also assume that on the fibres
of each $TX|C_\alpha$ one has:
\begin{equation}
  \label{eq:alb_speed_one_p3}
  \tnrmname=\dxdxnrmname.
\end{equation}
Having fixed $\alpha$, we will show that $\mu\mrest C_\alpha$ admits
a $(1,1+\varepsilon)$-biLipschitz Alberti representation with $\langle v_0^*,\phi\rangle$-speed
$\geq\sigma\ctnrm v_0^*.$ for each $\sigma\in(0,1)$; the result about
$\mu\mrest U$ will follow again by using Theorem \ref{thm:alberti_glue}. As we can rescale $v_0^*$, we can assume that $\alnrm
v_0^*.^*=1$; we will denote by $v_0\in S\left(\alnrmname\right)$ a
vector where $v_0^*$ assumes the norm. We let $M$ denote a constant
such that:
\begin{equation}
  \label{eq:alb_speed_one_p4}
  \frac{1}{M}\ltwonrmname\le\alnrmname\le M\ltwonrmname.
\end{equation}
We fix $\varepsilon_0\in(0,1)$ and $\theta\in(0,\pi/2)$ and, using
\ref{thm:bate_alb_arb_dir} and Theorem \ref{thm:alberti_rep_prod}.
we find a $(1,1+\varepsilon_0)$-biLipschitz Alberti representation $\albrep.$ of
$\mu\mrest C_\alpha$ in the $\phi$-direction of
$\cone\left(\frac{v_0}{\ltwonrm v_0.},\theta\right)$. Let $\gfun$
contain the components of $\phi$ and $\{f\}$, and let $\gbor$ contain
$\chi_U$. Using the Alberti representation $\albrep.$ and
(\ref{eq:alb_speed_one_p3}) we conclude that for $\mu\mrest
C_\alpha$-a.e.~$p$ there is an $\gpair$-generic velocity vector
$\gamma'(t)\in T_pX$ such that:
\begin{equation}\label{eq:alb_speed_one_p5}
\begin{aligned}
         \md\gamma(t)=\tnrm\gamma'(t).&\in[1,1+\epsi_0];\\
         (\phi\circ\gamma)'(t)&\in\cone\left(\frac{v_0}{\ltwonrm v_0.},\theta\right).
\end{aligned}
\end{equation}
In particular, (\ref{eq:alb_speed_one_p5}) and
(\ref{eq:alb_speed_one_p2}) imply that:
\begin{equation}
  \label{eq:alb_speed_one_p6}
  \alnrm\gamma'(t).\in\left[\frac{1}{1+\eta},(1+\eta)(1+\epsi_0)\right].
\end{equation}
We now use Lemma
\ref{cone_diameter} to get
\begin{equation}\label{eq:alb_speed_one_p6bis}
\alnrm\frac{(\phi\circ\gamma)'(t)}{\alnrm(\phi\circ\gamma)'(t).}-v_0.\le 4M^2(1-\cos\theta+\sin\theta);
\end{equation}
as
\begin{equation}\label{eq:alb_speed_one_p7}
  \left|1-\alnrm(\phi\circ\gamma)'(t).\right|\le\max\left(1-\frac{1}{1+\eta},(1+\eta)(1+\epsi_0)-1\right),
\end{equation}
we obtain
\begin{equation}\label{eq:alb_speed_one_p8}
  \begin{split}
    \alnrm(\phi\circ\gamma)'(t)-v_0.&\le4M^2(1-\cos\theta+\sin\theta)\\
    &\quad+\max\left(1-\frac{1}{1+\eta},(1+\eta)(1+\epsi_0)-1\right)\\&=a(\eta,\epsi_0,\theta),
  \end{split}
\end{equation}
where $\lim_{\eta,\epsi_0,\theta\to0}a(\eta,\epsi_0,\theta)=0$. Recall that
$t\in\dom\gamma$ is  a Lebesgue density point, and assume that $\langle
v_0^*,\phi\rangle\circ\gamma$, which is $M\glip
\phi.(1+\varepsilon_0)$-Lipschitz because of~(\ref{eq:alb_speed_one_p4}), has been extended to a neighbourhood of $t$ by
using MacShane's Lemma:
\begin{equation}\label{eq:alb_speed_one_p9}
  \begin{split}
    \langle v_0^*,\phi\rangle\circ\gamma(t+h)&-\langle
    v_0^*,\phi\rangle\circ\gamma(t)=\int_t^{t+h}\left(\langle
    v_0^*,\phi\rangle\circ\gamma\right)'(s)\,ds\\
&\ge\int_{[t,t+h]\cap\dom\gamma}\left(\langle
    v_0^*,\phi\rangle\circ\gamma\right)'(s)\\&\quad-\underbrace{M\glip \phi.\,
(1+\epsi_0)\lebmeas([t,t+h]\cap\dom\gamma)}_{o(h)}\\
&=\int_{[t,t+h]\cap\dom\gamma}\langle
    v_0^*,v_0\rangle\,ds\\
&\quad+\int_{[t,t+h]\cap\dom\gamma}\langle
    v_0^*,(\phi\circ\gamma)'(s)-v_0\rangle\,ds+o(h)\\
&\ge \lebmeas([t,t+h]\cap\dom\gamma)-h\,a(\eta,\epsi_0,\theta)+o(h),
\end{split}
\end{equation}
where in the last step we used the approximate continuity of
$(\phi\circ\gamma)'(s)$ at $t$.
Now (\ref{eq:alb_speed_one_p9}) implies that
\begin{equation}\label{eq:alb_speed_one_p10}
  (\langle
  v_0^*,\phi\rangle\circ\gamma)'(t)\ge\frac{1-a(\eta,\epsi_0,\theta)}{(1+\eta)^2(1+\epsi_0)}\md\gamma(t)\biglip\langle v_0^*,\phi\rangle(\gamma(t)),
\end{equation}
and it suffices to choose $\eta,\epsi_0,\theta$ small enough to guarantee
\begin{equation}\label{eq:alb_speed_one_p11}
\begin{aligned}
  \frac{1-a(\eta,\epsi_0,\theta)}{(1+\eta)^2(1+\epsi_0)}&\ge\sigma;\\
  \epsi&\ge\epsi_0.
\end{aligned}
\end{equation}
\par We now consider the general case in which $df$ is not
constant. We let $V\subset U$ be a full-measure Borel subset where $f$
is differentiable with respect to the chart functions $\phi$. On the
set where $df=0$ we have $\biglip f=0$, so we can assume that $df\ne0$
on $V$. We fix $\eta>0$ and use Lusin and Egorov Theorems to find
disjoint compact sets $C_\alpha\subset V$ and
$v^*_\alpha\in\real^N\setminus\{0\}$ such that
$\mu\left(V\setminus\bigcup_\alpha C_\alpha\right)=0$ and:
\begin{equation}
  \label{eq:alb_speed_one_p12}
  \ctnrm df(x) - v^*_\alpha.\le\ctnrm df(x).\quad(\forall x\in C_\alpha).
\end{equation}
We fix $\sigma'\in(0,1)$ and, using the special case $f=\langle
v^*_\alpha,\phi\rangle$, we obtain a $(1,1+\varepsilon)$-biLipschitz
Alberti representation $\albrep\alpha.=(P_\alpha,\nu_\alpha)$ of
$\mu\mrest C_\alpha$ with $\langle
v^*_\alpha,\phi\rangle$-speed $\geq\sigma'\ctnrm v^*_\alpha.$; then for
$P_\alpha$-a.e.~$\gamma$ we have:
\begin{equation}
  \begin{split}
    (f\circ\gamma)'(t)&\ge(\langle
    v_\alpha^*,\phi\rangle\circ\gamma)'(t)-\eta\metdiff\gamma(t)\ctnrm
    df.\\ &\ge\left(\sigma'\ctnrm v^*_\alpha.-\eta\ctnrm
    df.\right)\metdiff\gamma(t)\\
&\ge(\sigma'-(1+\sigma')\eta)\biglip f(\gamma(t))\metdiff\gamma(t),
  \end{split}
\end{equation}
and it suffices to choose $\eta$ small enough and $\sigma'$ close to
$1$  to guaratee that $\sigma'-(1+\sigma')\eta\ge\sigma$.
\par The proof of (\ref{thm:alb_speed_one_s1}) is now immediate. Let
$\gfun$ contain the components of the chart functions and $f$, and let
$\gbor$ contain the characteristic functions of the charts. Now,
for each $\sigma\in(0,1)$, we conclude that for $\mu$-a.e.~$x\in X$
there is an $\gpair$-generic velocity vector $\gamma'(t)\in T_xX$ with
\begin{equation}
  \left(f\circ\gamma\right)'(t)\ge\sigma\Lip f(\gamma(t))\md\gamma(t);
\end{equation}
observing that
\begin{equation}
  \left|\left(f\circ\gamma\right)'(t)\right|\le\lip f(\gamma(t))\md\gamma(t),
\end{equation} we conclude that the Borel set
\begin{equation}
  \left\{x\in X: \lip f(x)\ge\sigma\Lip f(x)\right\}
\end{equation} has full $\mu$-measure, and then let $\sigma\nearrow1$.
\end{proof}
\subsection{A new proof of $g_f=\biglip f$ in PI-spaces}
\label{subsec:gf_new}
In this subsection we give a new proof of the characterization of the
\textbf{minimal generalized upper gradient} $g_f$ of a Lipschitz function $f$ in a
PI-space. We will assume that the reader is familiar with the material
in \cite{cheeger}; in particular, we will denote by
$H^{1,p}(X,\mu)$  the Sobolev space introduced by Cheeger in
\cite[Sec.~2]{cheeger}. One of the main results in \cite{cheeger}
states that, if $p\in(1,\infty)$ and if $(X,\mu)$ is a PI-space, then
$H^{1,p}(X,\mu)$ is reflexive. Since we will use the reflexivity of
$H^{1,p}(X,\mu)$, throughout this subsection the power $p$ will be taken
to lie in $(1,\infty)$. Our goal is to give a new proof of \cite[Thm.~6.1]{cheeger}:
\begin{thm}
  \label{thm:min_upp_char}
  If $(X,\mu)$ is a PI-space and if $f\in\lipfun X.\cap H^{1,p}(X,\mu)$, then
  $\biglip f$ is a representative of the minimal generalized upper gradient of $f$
  (which is then independent of the power $p>1$).
\end{thm}
\par We first give some remarks on how the new proof differs from the
original one. The original proof contained two steps:
\begin{description}
\item[(S1)] Proof of Theorem \ref{thm:min_upp_char} under the
  additional assumption that $(X,\mu)$ is a length space.
\item[(S2)] Removing the assumption that $(X,\mu)$ is a length space.
\end{description}
The argument for \textbf{(S1)} was motivated by the observation that,
whenever $(X,\mu)$ is a length space and $g$ is a \emph{continuous}
upper gradient of $f$, then $g\ge\biglip f$ holds at each
point. Therefore the strategy in \cite{cheeger} was to prove an approximation
result \cite[Thm.~5.3]{cheeger} which states that for any $f\in
\lipfun X.\cap H^{1,p}(X,\mu)$ there is a sequence $(f_k,h_k)\subset
H^{1,p}(X,\mu)\times L^p(\mu)$ such that $f_k\to f$ in
$H^{1,p}(X,\mu)$, the function $v_k$ is a continuous upper gradient of $f_k$,
and $v_k\to g_f$ in $L^p(\mu)$. This approximation result is
probably the most technical part of Cheeger's original proof.
\par The first simplification of the new argument is that one does not
need to handle first the case in which $(X,\mu)$ is a length
space. The strategy of the proof is motivated by the
observation (Lemma \ref{lem:bound_upp_grad}) that if $g$ is a \emph{bounded}
upper gradient of $f$, then $g\ge\biglip f$ holds $\mu$-a.e.: this is
where Alberti representations are used. Had the minimal
generalized upper gradient been defined by minimizing the $p$-energy on
\emph{bounded} upper gradients, then Theorem \ref{thm:min_upp_char}
would have followed directly from Lemma
\ref{lem:bound_upp_grad}. However, as an
upper gradient in $L^p(\mu)$ can be infinite on a null set, one needs,
roughly speaking,
to approximate $f$ in $\lipfun X.\cap H^{1,p}(X,\mu)$ by functions
which have bounded upper gradients. Here we use an instance of the
argument ``modulus equals capacity'' \cite{ziemer_exlen_cap} which
appears also in the proof of \cite[Thm.~5.3]{cheeger}: however, as we
do not need to build approximations which use continuous upper
gradients, there are fewer technical details to handle.
\par The following lemma relates bounded upper gradients and Alberti
representations.
\begin{lemma}
  \label{lem:bound_upp_grad}
  If $(X,\mu)$ is a PI-space, $u:X\to\real$ is Lipschitz, and $g$ is
  a bounded upper gradient of $u$, then
  \begin{equation}
    \label{eq:bound_upp_grad_s1}
    g\ge\biglip u\quad\text{$\mu$-a.e.}
  \end{equation}
\end{lemma}
\begin{proof}
  For each $\varepsilon>0$ we can find countably many disjoint compact
  sets $\{K_\alpha\}$ and nonnegative real numbers
  $\{\lambda_\alpha\}$ such that:
  \begin{enumerate}
  \item For each $x\in K_\alpha$ one has $g(x)\in[\lambda_\alpha,\lambda_\alpha+\varepsilon)$;\vskip2mm
  \item The $\{K_\alpha\}$ cover $X$ in measure:
    $\mu\left(X\setminus\bigcup_\alpha K_\alpha\right)=0$.
  \end{enumerate}
  By Theorem \ref{thm:alb_speed_one} for $\mu\mrest K_\alpha$-a.e.~$x$
  there is a $(1,1+\varepsilon)$-biLipschitz fragment $\gamma$:
  \begin{enumerate}
  \item The domain $\dom\gamma$ is a subset of $[-1,\infty)$;\vskip2mm
  \item One has $\gamma(0)=x$ and:
    \begin{equation}
      \label{eq:bound_upp_grad_p1}
      \lim_{r\searrow0}\frac{\lebmeas\left(\dom\gamma\cap(-r,r)\right)}{2r}=1;
    \end{equation}
  \item The point $0$ is an approximate continuity point of
    $(u\circ\gamma)'$ and
    \begin{equation}
      \label{eq:bound_upp_grad_p2}
      (u\circ\gamma)'(0)\ge\frac{1}{1+\varepsilon}\biglip u(x).
    \end{equation}
  \end{enumerate}
  Let $[c,d]$ be the minimal interval containing $\dom\gamma$ and let
  $\{(a_i,b_i)\}$ denote the set of component of
  $[c,d]\setminus\dom\gamma$; we extend $\gamma$ on each interval
  $(a_i,b_i)$ by choosing a $C$-quasigeodesic joining $\gamma(a_i)$ to
  $\gamma(b_i)$: note that this is possible because a PI-space is
  $C$-quasiconvex for some $C$ \cite[Sec.~17]{cheeger}\footnote{This
    result is due to Semmes.}. Then:
  \begin{equation}
    \label{eq:bound_upp_grad_p3}
    \begin{split}
      \left|\int_0^r(u\circ\gamma)'(s)\,ds\right|=\left|u\left(\gamma(r)\right)-u(x)\right|
      &\le\int_0^rg\circ\gamma\,\metdiff\gamma(t)\,dt\\
      &\le(\lambda_\alpha+\varepsilon)(1+\varepsilon)r+o(r);
    \end{split}
  \end{equation}
  dividing by $r$ and letting $r\searrow0$ we get:
  \begin{equation}
    \label{eq:bound_upp_grad_p4}
    \biglip u(x)\le(1+\varepsilon)^2\left(g(x)+\varepsilon\right),
  \end{equation}
  and the result follows letting $\varepsilon\searrow0$.
\end{proof}
\begin{remark}
 Note that in Lemma \ref{lem:bound_upp_grad} we had to work with
 bounded upper gradients to establish (\ref{eq:bound_upp_grad_p3}); in
 fact, to apply the Fundamental Theorem of Calculus, one needs curves,
 and the $K_\alpha$ might only contain fragments, and thus, filling-in the
 fragments in $K_\alpha$ using that a PI-space is quasiconvex might produce curves
 where $g$ is unbounded or infinite on a set of positive length. Note
 also that in a PI-space one can use curves instead of fragments in
 building Alberti representations; this follows from a general
 observation in \cite{curr_alb} that if $\mu$ is a Radon measure on a
 quasiconvex metric space $X$, a Lipschitz Alberti representation of
 $\mu$ can be replaced by one which gives the same derivation and
 whose probability measure has support contained in the set of curves
 in $X$.
\end{remark}
\par To prove Theorem \ref{thm:min_upp_char} we can just consider, as
in \cite{cheeger}, upper gradients which are lower semicontinuous. In
fact, the Vitali-Carath\'eodory Theorem \cite[Thm.~2.25]{rudin-real}
states that for any $h\in L^1(\mu)$ and any $\varepsilon>0$ there are
functions $u$ and $v$ such that $u\le h\le v$, $u$ is upper
semicontinuous and bounded from above, $v$ is lower
semicontinuous and bounded from below, and
$\|u-v\|_{L^1(\mu)}<\varepsilon$. Thus, any upper gradient of $f$ can
be replaced, up to slighly increasing the $L^p(\mu)$-norm, by one
which is lower semicontinous and bounded below by a small positive
constant. We thus only need to prove:
\begin{thm}
  \label{thm:mod_eq_cap}
  Suppose  $(X,\mu)$ is a PI-space,  $u$ is a real-valued Lipschitz function
  defined on $X$ and  $g$ is a lower-semicontinuous upper gradient of $u$. Then:
    \begin{equation}
    \label{eq:mod_eq_cap_s1}
    g\ge\biglip u\quad\text{$\mu$-a.e.}
  \end{equation}
\end{thm}
\par To prove Theorem \ref{thm:mod_eq_cap} we recall a consequence
of the Poincar\'e inequality, which follows from the characterization of
the Poincar\'e inequality in terms of the maximal function associated
to an upper gradient \cite[Lem.~5.15]{heinonen98}. Suppose that $g$ is an upper gradient
for the function $u$ and that $g\in L^p(\mu)$; consider for
$N\in(0,\infty)$ the set:
\begin{equation}
  \label{eq:max_fun_set}
  A(g,N)=\left\{x\in X:\sup_{r>0}\av_{\ball x,r.}g^p\,d\mu\le N^p\right\};
\end{equation}
then if
$x,y\in A(g,N)$ are Lebesgue points of $u$, one has
\begin{equation}
  \label{eq:max_lip_set}
  \left|u(x)-u(y)\right|\le CN\dst x,y.,
\end{equation}
where $C$ is a universal constant that depends only on the PI-space $(X,\mu)$.
\begin{proof}
  Let $N,M$ be natural numbers and $S=A(g,N)\cap\ball x,M.$; it
  suffices to show that \eqref{eq:mod_eq_cap_s1} holds $\mu\mrest
  S$-a.e. Fix $\varepsilon>0$ and let:
  \begin{equation}
    \label{eq:mod_eq_cap_p1}
    u_n(x)=\inf\left\{\int_\gamma(g\wedge n+\varepsilon)\,d\hmeas
      ._\gamma+u(y):\;\text{$\gamma$ is a Lipschitz curve joining
        $x$ to $y\in S$}\right\}.
  \end{equation}
  As $(X,\mu)$ is $C$-quasiconvex for some $C$, the function $u_n$ is
  $C(n+\varepsilon)$-Lipschitz. Note also that $h_n=g\wedge
  n+\varepsilon$ is an upper gradient of $u_n$. We let
  $\{x_{j,m}\}_{j\in J_m}$ be
  a finite $\frac{1}{m}$-dense set in $S$, which exists because $X$ is
  proper. Using the fact that the $h_n$ are lower-semicontinuous and
  uniformly bounded away from zero
  (compare \cite[3.3,3.4]{ziemer_exlen_cap} and
  \cite[Lem.~5.18]{cheeger}), 
  it follows that for each
  $m\in\natural$ there is an $N_m\in\natural$ such that, for $n\ge N_m$,
  one has:
  \begin{equation}
    \label{eq:mod_eq_cap_p2}
    u_n(x_{j,m})=u(x_{j,m})\quad(\forall j\in J_m).
  \end{equation}
  Let $v_n$ be obtained by truncating $u_n$ so that
  \begin{equation}
    \label{eq:mod_eq_cap_p3}
    \left|v_n(x)\right|\le\sup_{y\in\ball x,M.}\left|u_n(y)\right|;
  \end{equation}
  thus, for $n\ge N_m$ one has:
  \begin{equation}
    \label{eq:mod_eq_cap_p4}
    v_n(x_{j,m})=u(x_{j,m})\quad(\forall j\in J_m).
  \end{equation}
  Note that $h_n$ is an upper gradient of $v_n$ and that
  \eqref{eq:max_lip_set} implies that the functions $v_{N_m}$, when
  restricted to $S$, are uniformly $C(N+\varepsilon)$-Lipschitz;
  therefore, \eqref{eq:mod_eq_cap_p3} implies that $v_{N_m}\to u$
  uniformly on $S$. As the Banach space $H^{1,p}(\mu\mrest\ball x,M.)$ is
  reflexive, by
  applying Mazur's Lemma
  we can find Lipschitz functions $w_n$ and integers $Q_n$ such that:
  \begin{enumerate}
  \item The sequence $\{w_n\}$ converges to the function $w$ in
    $H^{1,p}(\mu\mrest\ball x,M.)$ and $w=u$ on $S$;\vskip2mm
  \item Each function $w_n$ is a convex combination of finitely many
    of the functions $v_{N_m}$;\vskip2mm
  \item The function $g\wedge Q_n+\varepsilon$ is an upper gradient
    for $w_n$.
  \end{enumerate}
  We then recall that in a
    PI-space there is a constant $C$ such that, for each Lipschitz
    function $f$, one has $Cg_f\ge\biglip f$ $\mu$-a.e. As $w_n\to w$ in
    $H^{1,p}(\mu\mrest\ball x,M.)$, one has that the generalized
    minimal upper gradients $\{g_{w_n-w}\}$ converge to $0$ in
    $L^p(\mu\mrest \ball x,M.)$; by the locality property of
    generalized minimal upper gradients \cite[Cor.~2.25]{cheeger}, as
    $u=w$ on $S$, we have that $\{g_{w_n-u}\}$ converges to $0$ in
    $L^p(\mu\mrest S)$; we thus conclude that $\biglip(w_n-u)\to 0$ in
    $L^p(\mu\mrest S)$. As $|\biglip w_n-\biglip
    u|\le\biglip(w_n-u)$, we can then pass to a subquence such that
    $\biglip w_n\to\biglip u$ $\mu\mrest S$-a.e. Now,
  by Lemma \ref{lem:bound_upp_grad} we have that $g\wedge Q_n+\varepsilon\ge\biglip
  w_n$ holds $\mu\mrest S$-a.e., and thus $g\ge\biglip u$ holds $\mu\mrest S$-a.e.
\end{proof}

\section{The geometry of blow-ups/tangent cones}
\label{sec:geom_blowups}
\newcount\maskgeom
\maskgeom=0
\ifnum\maskgeom>0{
\begin{enumerate}
\item A \textcolor{blue}{??} denotes a tentative statement;
\item Discuss Alberti representations without the normalization for
  probability measures and rectifiable measures do not need to be finite;
\item We need a preamble on convergence of metric measure spaces and
  measured blow-ups of Lipschitz functions: $\tang(X,\mu,p)$ denotes
  the set of blow-ups $(Y,\nu,q)$ of $(X,\mu)$ at $p$;
  $\tang(X,\mu,f,p)$ denotes the set of blow-ups $(Y,\nu,g,q)$ of
  $(X,\mu,f)$ at $p$: here $f:X\to Z$ is Lipschitz and $Z$ proper (we
  take $Z=\real^N$);
\item In the approach I'd prefer to discuss convergence using
  embeddings in a common metric space; this is to avoid issues with
  the measurability of the GH-approximations and of their choice,
  because approximations might push-forward rect-measures to sums of
  dirac measures;
\item Goal: proof of Theorem \ref{thm:measured_blow_up}
  \item Standing assumption: We have just one chart $(X,\psi)$.
\item Distinguished geodesics pass through every point: a corollary of
  Theorem \ref{thm:measured_blow_up}.
\item The canonical map $BX\ra T_xX$ is a metric submersion, with respect to which
distinguished geodesics
are ``horizontal'': a corollary of
  Theorem \ref{thm:measured_blow_up}.
\item State a version of Theorem \ref{thm:measured_blow_up} with
  ultrafilters to blow-up a pull-back metric of a map $F:X\to Z$?
\end{enumerate}
}\fi
In this section we show that, if $(X,\mu)$ is a differentiability
space, blowing-up the measure $\mu$ at a generic point yields measures
which possess Alberti representations concentrated on
distinguished geodesic lines on which the blow-ups of the chart
functions have constant derivatives, and are harmonic. This generalizes the fact that in
PI-spaces the blow-ups are generalized linear functions. Weaker
versions of the results presented here, where the blow-up of the
measure is not discussed,
have been obtained in \cite{deralb}, and 
\cite{david_difftang_ahlfors}. The  result in \cite{deralb} is more general than
\cite{david_difftang_ahlfors} because it
applies also in the context of Weaver derivations: we point out that
the results in this section, under the assumption that $\mu$ is
asymptotically doubling, have natural counterparts in that
context. 
We first recall some notions of
blow-ups of metric measure spaces and Lipschitz functions. Note that
we use the terminology \emph{blow-up} to avoid a conflict with the
word \emph{tangent} which is used for different objects in this paper;
often, instances of what we call \emph{blow-ups} are called
\emph{tangent cones / tangent spaces} in the literature.

\subsection{Blow-ups of metric measure spaces and Lipschitz maps}
\begin{defn}\label{def:space_blow-up}
A \textbf{blow-up of a metric space $X$ at a point
  $p$} is a (complete) pointed metric space $(Y,q)$ which is a pointed
Gromov-Hausdorff limit of a sequence $(\frac{1}{r_n}X,p)$ where $r_n\searrow0$:
the notation $\frac{1}{r_n}X$ means that the metric on $X$ is rescaled by
$1/r_n$; the class of blow-ups of $X$ at $p$ is denoted by
$\tang(X,p)$.
\end{defn}

\begin{remark}\label{rem:space_blow_up}
In Subsection \ref{subsec:metdiff_blowups} we discuss blow-ups
of metric spaces in a more general context which requires the notion
of ultralimits: under suitable assumptions on $X$, a sequence
$(\frac{1}{r_n}X,p)$ will always be precompact and the two notions
will agree. This is the case, for example, if $X$ is a doubling
metric space. However, in the context of differentiability spaces we
merely know (Theorem~\ref{thm:as_doubling}) that $\mu$ is
asymptotically doubling, and that porous sets are $\mu$-null. This implies that,
for $\mu$-a.e.~$p\in X$, there is a compact set $S_p$ such that:
$S_p$ is metrically doubling, and for each $\varepsilon>0$, there is an
$r_0>0$ such that, for each $r\le r_0$, the set $S_p\cap\ball p,r.$ is
$\varepsilon r$-dense in $\ball p,r.$. This allows essentially to
reduce the existence of blow-ups to the case in which $X$ is doubling.
\end{remark}

Recall that if the sequence $(\frac{1}{r_n}X,p)$ converges to
$(Y,q)$ in the pointed Gromov-Hausdorff sense, there is a pointed
metric space $(Z,z)$ such that there are isometric embeddings
$\iota_n:(\frac{1}{r_n}X,p)\to (Z,z)$ and $\iota:(Y,q)\to (Z,z)$, and,
for each $R>0$, one has:
\begin{equation}
  \label{eq:pointed_embedds}
  \begin{aligned}
    &\lim_{n\to\infty}\sup_{y\in\ball z,R.\cap
      \iota(Y)}\dist\left(\iota_n\left(\frac{1}{r_n}
        X\right),\{y\}\right)=0,\\
    &\lim_{n\to\infty}\sup_{y\in\ball z,R.\cap
      \iota_n(\frac{1}{r_n}X)}\dist\left(\iota\left(Y\right),\{y\}\right)=0.
  \end{aligned}
\end{equation}
In particular, each $q'\in Y$ can be \emph{approximated} by a sequence
$p'_n\in\frac{1}{r_n} X$ such that $\iota_n(p'_n)\to\iota(q')$ in
$Z$. This notion can be made independent of the embedding in $Z$ and
one can \textbf{represent} each point $q'\in Y$ by some sequence
$(p'_n)\subset X$ of points converging to $p$ (compare the treatment
with ultralimits in subsection
\ref{subsec:metdiff_blowups}). Moreover, if $(p'_n)$ represents $q'$,
and if $(\tilde p'_n)$ represents $\tilde q'$, we have:
\begin{equation}
  \label{eq:represents_dist_comp}
  \ydst q',\tilde q'.=\lim_{n\to\infty}\frac{\xdst p'_n,\tilde p'_n.}{r_n}.
\end{equation}
\begin{defn}
  \label{defn:measure_blow_up}
  Let $(X,\mu)$ be a metric measure space; \textbf{a blow-up of
    $(X,\mu)$ at $p$} is a triple $(Y,\nu,q)$ such that one has
  $(\frac{1}{r_n}X,p)\to(Y,q)\in\tang(X,p)$, and, having chosen a
  pointed metric space $(Z,z)$ and isometric embeddings
  $\iota_n:(\frac{1}{r_n}X,p)\to(Z,z)$ and $\iota:(Y,q)\to(Z,z)$ such
  that (\ref{eq:pointed_embedds}) holds, one has:
  \begin{equation}
    \label{eq:measure_blow_up1}
    \mpush(\iota_n).\frac{\mu}{\mu\left(\ball
        p,r_n.\right)}\xrightarrow{\text{w*}}\mpush\iota.\nu.\quad\text{(convergence
      in the weak* topology).}
  \end{equation}
  The set of blow-ups of $(X,\mu)$ at $p$ will be denoted by $\tang(X,\mu,p)$.
\end{defn}

\begin{remark}
  \label{rem:measure_blow_up}
  Note that if $\mu$ is asymptotically doubling and if porous sets are
  $\mu$-null, then for
  $\mu$-a.e.~$p\in X$ one has $\tang(X,\mu,p)\ne\emptyset$. In fact,
  at a generic point $p$, for each sequence of scaling factors
  $r_n\searrow0$, there is a subsequence $r_{n_k}$ such that
  $(\frac{1}{r_{n_k}}X,p)\to(Y,q)\in\tang(X,p)$, and there is a
  doubling measure $\nu$ such that (\ref{eq:measure_blow_up1}) holds.
\end{remark}

\par We finally discuss blow-ups of Lipschitz mappings which take
values into Euclidean spaces.
\begin{defn}
  \label{defn:maps_blow_up}
  Let $(X,\mu)$ be a metric measure space and $\psi:X\to\real^N$ a
  Lipschitz map; then \textbf{a blow-up of $(X,\mu,\psi)$ at $p$} is a
  tuple $(Y,\nu,\varphi,q)$ such that one has that $(Y,\nu,q)\in\tang(X,\mu,p)$,
  where the blow-up is realized by considering scaling factors
  $r_n\searrow0$, and where $\varphi:Y\to\real^N$ is a Lipschitz function
  such that, whenever $(p'_n)\subset X$ represents $q'$, one has:
  \begin{equation}
    \label{eq:maps_blow_up1}
    \varphi(q')=\lim_{n\to\infty}\frac{\psi(p'_n)-\psi(p)}{r_n}.
  \end{equation}
  The set of blow-ups of $(X,\mu,\psi)$ at $p$ will be denoted by $\tang(X,\mu,\psi,p)$.
\end{defn}

\begin{remark}
  \label{rem:maps_blow_up}
  If $\mu$ is asymptotically doubling and porous sets are $\mu$-null, then for
  $\mu$-a.e.~$p\in X$ one has that $\tang(X,\mu,\psi,p)\ne\emptyset$
  by an application of Ascoli-Arzel\'a.
\end{remark}

\subsection{Blowing up Alberti representations}

\par We can now state the main result of this Section.
\begin{thm}
    \label{thm:measured_blow_up}
    Let $(U,\psi)$ be an $N$-dimensional differentiability chart for
    the differentiability space $(X,\mu)$; then for $\mu\mrest
    U$-a.e.~$p$, for each blow-up
    $(Y,\nu,\varphi,q)\in\tang(X,\mu,\psi,p)$ and for each unit vector
    $v_0\in T_pX$, the measure $\nu$ admits an
    Alberti representation $\albrep .=(Q,\Phi)$ where:
    \begin{enumerate}
    \item $Q$ is concentrated on the set $\lines(\varphi,v_0)$ of unit
      speed geodesic lines in $Y$ with $(\varphi\circ\gamma)'=v_0$;\vskip2mm
    \item For each $\gamma\in\lines(\varphi,v_0)$ the measure
      $\Phi_\gamma$ is given by:
      \begin{equation}
        \label{eq:measured_blow_up_s1}
        \Phi_\gamma=\hmeas ._\gamma.
      \end{equation}
    \end{enumerate}
  \end{thm}
Suppose that $X'\subset X$ and that the measures $\mu'$ and $\mu\mrest
X'$ are in the same measure class. Then an application of measure
differentiation shows that for $\mu\mrest X'$-a.e.~$p$ the sets
$\tang(X',\mu',p)$ and $\tang(X,\mu,p)$ coincide. Given
$(Y,\nu,\varphi,q)\in\tang(X,\mu,\psi,p)$ we will then obtain the
Alberti representations of $\nu$ by \emph{blowing-up} Alberti
representations of measures $\mu'\ll\mu$ which admit Alberti
representations of a special form.
\begin{defn}[Simplified Alberti representations]
We say that the Alberti representation $\albrep .=(P,\Psi)$ of the
measuure $\mu'$ is \textbf{simplified} if there are
$(C_0,D_0,\tau_0)\in(0,\infty)^3$ such that:
\begin{enumerate}
\item The measure $P$ is finite and is supported on the set of
  $(C_0,D_0)$-biLipschitz fragments whose domain is a subset of $[0,\tau_0]$;\vskip2mm
\item Denoting by $M(X)$ the set of finite Radon measures on $X$,
  $\Psi$ is the Borel map:
  \begin{equation}
    \label{eq:frags_measures}
    \begin{aligned}
      \Psi:\frags(X)&\to M(X)\\
      \gamma&\mapsto\mpush\gamma.\left(\lebmeas\mrest\dom\gamma\right).
    \end{aligned}
  \end{equation}
\end{enumerate}
\end{defn}
To prove Theorem \ref{thm:measured_blow_up} we will use the
following technical result about blow-ups of a simplified Alberti representation
$\albrep.$.\def\cone{{\mathcal C}} 
\def\stdtang{(Y,\nu,\varphi,q)}
    \def\stdpoint{(X,\mu,\psi,p)}
\def\conset#1.{\mathcal{S}_{\setbox0=\hbox{$#1\unskip$}\ifdim\wd0=0pt R_0
    \else #1\fi}}
\begin{thm}
  \label{thm:alberti_blow_up}
  Suppose that the simplified Alberti representation $\albrep.$ of the finite
  measure
  $\mu'\ll\mu$ is in the $\psi$-direction of a cone $\cone$ and that it has
   $\langle v_0,\psi\rangle$-speed $\ge\sigma_0\ctnrm v_0.$. Then there
  is a Borel set $U$ with full  
  $\mu'$-measure such that for each $p\in U$, for each  $\stdtang\in\tang\stdpoint$ and each
    $R_0>0$ the measure $\nu\mrest\ball q,R_0.$ admits an Alberti
    representation $\albrep R_0.=(Q_{R_0},\Phi)$ such that:
    \begin{enumerate}
    \item The finite Radon measure $Q_{R_0}$ has support contatined in
      a compact
      set $\conset .\subset\frags(Y)$ of geodesic segments;\vskip2mm
    \item The total mass of $Q_{R_0}$ is bounded by
      $\frac{D_0}{2R_0}\left(\adoubling(\mu,p)\right)^{\log_2R_0 +1}$,
      where $\adoubling(\mu,p)$ denotes the asymptotic doubling
      constant of $\mu$ at $p$, i.e.:
      \begin{equation}
        \label{eq:alberti_blow_up_s_add1}
        \adoubling(\mu,p)=\limsup_{r\searrow0}\frac{\mu\left(\ball
            p,2r.\right)}{\mu\left(\ball p,r.\right)};
      \end{equation}
    \item The set $\conset .$ consists of those geodesic segments
      $\gamma$ which have
      domain contained in $\left[0,\frac{4R_0}{C_0}\right]$, image
      contained in $\clball q,2R_0.$, which have both endpoints lying
      outside of $\ball q,\frac{3}{2}R_0.$, which intersect $\clball
      q,R_0.$,  which have constant 
      speed $\theta_\gamma\in[C_0,D_0]$, which satisfy:
      \begin{multline}
        \label{eq:alberti_blow_up_s2}
        \sgn(s_2-s_1)\,\left\langle
          v_0,\varphi\circ\gamma(s_2)-\varphi\circ\gamma(s_1)\right\rangle
        \\\ge\sigma_0\theta_\gamma(s_2-s_1)\biglip\left(\langle
          v_0,\psi\rangle\right)(p)\quad(\forall s_1,s_2\in\dom\gamma),
      \end{multline}
      and such that there is a $w_\gamma\in\bar\cone$ for which the
      following holds:
      \begin{equation}
        \label{eq:alberti_blow_up_s1}
        \varphi\circ\gamma(s_2)-\varphi\circ\gamma(s_1)=(s_2-s_1)w_\gamma\quad(\forall
        s_1,s_2\in\dom\gamma);
      \end{equation}
    \item For each $\gamma\in\conset .$ the measure
      $\Phi_\gamma$ is given by:
      \begin{equation}
        \label{eq:alberti_blow_up_s3}
        \Phi_\gamma=\frac{1}{\theta_\gamma}\hmeas ._\gamma\mrest\ball q,r_0..
      \end{equation}
    \end{enumerate}
\end{thm}
\par We now introduce a bit of terminology to split the measure on
fragments in a \emph{good} and a \emph{bad} part.
\def\stdpar{(\psi{},\cone,v_0,\sigma_0,[C_0,D_0],\tau_0,\varepsilon{},S)}
\def\pars{{\rm PAR}(\varepsilon{},S)}%
\begin{defn}
  \label{defn:reg_set}
  Let $\varepsilon>0$ and $S>0$; we denote by
  $\widetilde\reg\stdpar$ the set of pairs $(\gamma,p)\in\frags(X)\times X$ such
  that:
  \begin{description}
  \item[(Reg1)] The fragment $\gamma$ is $[C_0,D_0]$-biLipschitz with
    domain contained in $[0,\tau_0]$;\vskip2mm
  \item[(Reg2)] There is a $t\in\dom\gamma$ with $p=\gamma(t)$ and for each
    $r_1,r_2\le S$ one has:
    \begin{equation}
      \label{eq:reg_set_s1}
      \lebmeas\left(\dom\gamma\cap[t-r_1,t+r_2]\right)\ge(1-\varepsilon)(r_1+r_2);
    \end{equation}
  \item[(Reg3)] There are a $\theta\in[C_0,D_0]$ and $w\in\cone$ such
    that if $r\le S$ and $s_1,s_2\in[t-r,t+r]\cap\dom\gamma$ one has:
    \begin{equation}
      \label{eq:reg_set_s2}
      \begin{aligned}
        \left|\dst
          \gamma(s_1),\gamma(s_2).-\theta|s_1-s_2|\right|&\le\varepsilon|s_1-s_2|\\
        \left|\psi\circ\gamma(s_1)-\psi\circ\gamma(s_2)-w(s_1-s_2)\right|&\le\varepsilon|s_1-s_2|;
      \end{aligned}
    \end{equation}
  \item[(Reg4)] If $r\le S$ and $s_ 1,s_2\in[t-r,t+r]$ with $s_1\le
    s_2$ then:
    \begin{equation}
      \label{eq:reg_set_s3}
      \left\langle
        v_0,\psi\circ\gamma(s_2)-\psi\circ\gamma(s_1)\right\rangle\ge(\sigma_0-\varepsilon)\theta
      \biglip\left(\langle v_0,\varphi\rangle\right)(p)(s_2-s_1).
    \end{equation}
  \end{description}
  In the following we will usually fix a choice of
$(\psi{},\cone,v_0,\sigma_0,[C_0,D_0],\tau_0)$ and vary
$(\varepsilon,S)\in(0,\infty)^2$; we thus introduce the shorter notation
$\pars$ for $\stdpar$. We denote by $\reg(\pars)$ the subset of those
$(\gamma,p)\in\widetilde\reg(\pars)$ such that:
\begin{description}
\item[(Reg5)] For all $r_1,r_2\le S$ one has:
  \begin{equation}
    \label{eq:reg_set_s4}
    \lebmeas\left(\gamma^{-1}\left(\widetilde\reg(\pars)\right)_\gamma\cap[t-r_1,t+r_2]\right)\ge(1-\varepsilon)(r_1+r_2),
  \end{equation}
\end{description} where then notation $\left(\widetilde\reg(\pars)\right)_\gamma$ denotes the
$\gamma$-section of the set $\widetilde\reg(\pars)$.
\end{defn}
\begin{lemma}
  \label{lem:reg_borel}
  The set $\reg(\pars)$ is a Borel subset of $\frags(X)\times X$. 
\end{lemma}
\begin{proof}
  \def\Imgset{\text{\normalfont IMG}}
  \def\Invmap{\text{\normalfont Inv}}
  \def\Triset{\text{\normalfont TRIP}}
  We first show that the set $\widetilde\reg(\pars)$ is Borel. The set
  of fragments satisfying \textbf{(Reg1)} is closed in
  $\frags(X)$. Now consider the set
  \begin{equation}
    \label{eq:reg_borel_p1}
    \Imgset=\left\{(\gamma,p)\in\frags(X)\times X: p\in\gamma(\dom\gamma)\right\},
  \end{equation}
  which is closed in $\frags(X)\times X$; let
  $\Imgset(C_0,D_0,\tau_0)$ denote the closed subset of those
  $(\gamma,p)\in\Imgset$ such that $\gamma$ satisfies \textbf{(Reg1)};
  then the map:
  \begin{equation}
    \label{eq:reg_borel_p2}
    \begin{aligned}
      \Invmap:\Imgset(C_0,D_0,\tau_0)&\to\real\\
      (\gamma,p)&\mapsto\gamma^{-1}(p)
    \end{aligned}
  \end{equation}
  is continuous. Using an argument similar to that used to prove that
  the map defined at (\ref{eq:generic_borel_p9}) is Borel, we see
  that, for fixed $r_1,r_2>0$, the map:
  \begin{equation}
    \label{eq:reg_borel_p3}
    \begin{aligned}
      \psi_{r_1,r_2}:\frags(X)\times\real&\to\real\\
      (\gamma,t)&\mapsto\lebmeas\left(\dom\gamma\cap[t-r_1,t+r_2]\right)
    \end{aligned}
  \end{equation}
  is Borel; then the set of pairs $(\gamma,p)$ satisfying
  \textbf{(Reg1)--(Reg2)} is Borel since it can be written as:
  \begin{equation}
    \label{eq:reg_borel_p4}
    \bigcap_{r_1,r_2\in[0,S]\cap\rational}\left\{(\gamma,p)\in\Imgset(C_0,D_0,\tau_0):
      \psi_{r_1,r_2}\left(\gamma,\gamma^{-1}(p)
        \right)\ge(1-\varepsilon)(r_1+r_2)\right\}.
    \end{equation}
    That the set of pairs satisfying \textbf{(Reg3)--(Reg--4)} is
    Borel follows by arguments similar to those used in the proof of
    Lemma \ref{lem:generic_borel}, compare
    (\ref{eq:generic_borel_p3}), (\ref{eq:generic_borel_p4}).
    \par Consider the set:
    \begin{equation}
      \label{eq:reg_borel_p5}
      \Triset=\left\{(\gamma,p,t)\in\frags(X)\times X\times\real:
        t\in\dom\gamma, \gamma(t)=p\right\},
    \end{equation}
    which is closed in $\frags(X)\times X\times\real$. We now fix
    $r_1,r_2\ge0$ and define the 
    Borel set:
    \begin{equation}
      \label{eq:reg_borel_p5_1}
      A_{r_1,r_2}=\left\{(\gamma,p,t)\in\widetilde\reg(\pars)\times\real\cap\Triset: \gamma^{-1}(p)\in[t-r_1,t+r_2]\right\};
    \end{equation}
    using
    \cite[Thm.~7.25]{kechris_desc} we get that the map:
    \begin{equation}
      \label{eq:reg_borel_p6}
      \begin{aligned}
      \Omega_{r_1,r_2}:\frags(X)\times\real\times M(X)&\to\real\\
      (\gamma,t,\mu)&\mapsto\mu\left(\left(A_{r_1,r_2}\right)_{(\gamma,t)}\right)
    \end{aligned}
  \end{equation}
    is Borel. The proof that $\reg(\pars)$ is Borel is completed by
    observing that
    \textbf{(Reg5)} can be expressed as:
    \begin{equation}
      \label{eq:reg_borel_p7}
      \Omega_{r_1,r_2}\left(\gamma,\gamma^{-1}(p),\Psi(\gamma)\right)\ge(1-\varepsilon)(r_1+r_2)\quad(\forall r_1,r_2\in[0,S]\cap\rational).
    \end{equation}
\end{proof}
\par Consider the map $\Psi$ defined in (\ref{eq:frags_measures}); we
can decompose the measures $\Psi(\gamma)$ as follows:
\begin{equation}
  \label{eq:reg_dec}
  \begin{aligned}
    \Psi_{\pars}(\gamma)&=\Psi(\gamma)\mrest\left(\reg(\pars)\right)_\gamma;\\
    \Psi^c_{\pars}(\gamma)&=\Psi(\gamma)\mrest\left(\reg(\pars)\right)^c_\gamma.
  \end{aligned}
\end{equation}
\begin{lemma}
  \label{lem:reg_dec_borel}
  The maps $\Psi_{\pars}$ and $\Psi^c_{\pars}$ are Borel. Thus, given
  an Alberti representation $\albrep.$ of the finite measure $\mu'$
   satisfying the assumptions of Theorem
  \ref{thm:alberti_blow_up}, we can define the finite Radon measures:
  \begin{equation}
    \label{eq:reg_dec_borel_s1}
    \begin{aligned}
      \mu'_{\pars}&=\int_{\frags(X)}\Psi_{\pars}(\gamma)\,dP(\gamma)\\
      \mu'^c_{\pars}&=\int_{\frags(X)}\Psi^c_{\pars}(\gamma)\,dP(\gamma),
    \end{aligned}
  \end{equation}
  which satisfy:
  \begin{equation}
    \label{eq:reg_dec_borel_s2}
    \mu'_{\pars}+\mu'^c_{\pars}=\mu',
  \end{equation}and:
  \begin{equation}
    \label{eq:reg_dec_borel_s3}
    \lim_{S\searrow0}\left\|\mu'^c_{\pars}\right\|_{M(X)}=0.
  \end{equation}
\end{lemma}
\begin{proof}
  By \cite[Thm.~7.25]{kechris_desc} the map:
  \begin{equation}
    \label{eq:reg_dec_borel_p1}
    \begin{aligned}
    \Omega:\frags(X)\times M(X)&\to\real\\
    (\gamma,\mu)&\mapsto\mu\mrest\left(\reg(\pars)\right)_\gamma
  \end{aligned}
\end{equation}
  is Borel, and thus
  $\Psi_{\pars}$ is Borel as $\Psi_{\pars}(\gamma)$ can be written as $\Omega\left(\gamma,\Psi(\gamma)\right)$; the proof for $\Psi^c_{\pars}$ is
  similar.
  \par Now note that $\left\|\Psi^c_{\pars}(\gamma)\right\|\le\tau_0$
  and that, for each $\gamma$, one has
  \begin{equation}
    \label{eq:reg_dec_borel_p2}
    \lim_{S\searrow0}\left\|\Psi^c_{\pars}(\gamma)\right\|=0,
  \end{equation}
  as for $\lebmeas$-a.e.~$t\in\dom\gamma$ there is an $S(t)$ such that,
  for $s\le S(t)$, one has $(\gamma,\gamma(t))\in\reg(\pars)$. Then
  (\ref{eq:reg_dec_borel_s3}) follows by the Dominated Convergence Theorem.
\end{proof}
\par The following lemma follows from a standard argument in measure
differentiation. \def\pars{{\rm PAR}(\varepsilon_m{},S_m)}%
\begin{lemma}
  \label{lem:meas_diff}
  Let $\{\varepsilon_m\}\subset(0,\infty)$ be a sequence with
  $\sum_m\varepsilon_m<\infty$; then there are a Borel $U\subset X$
  and a sequence of pairs $\{(s_m,S_m)\}_m\subset(0,\infty)^2$ such
  that:
  \begin{enumerate}
  \item One has $\mu(X\setminus U)\le\sum_m\varepsilon_m$ and, for each
    $m$, one also has $s_m\le S_m$;\vskip2mm
  \item For each $x\in U$ and for each $r\le s_m$, one has:
    \begin{equation}
      \label{eq:meas_diff}
      \mu'^c_{\pars}\left(\ball
        x,r.\right)\le\varepsilon_m\mu'\left(\ball x,r.\right).
    \end{equation}
  \end{enumerate}
\end{lemma}
\def\sball#1,#2,#3.{B_{#1}(#2,#3)} 
\def\scball#1,#2,#3.{\bar B_{#1}(#2,#3)} 
\begin{proof}[Proof of Theorem \ref{thm:alberti_blow_up}] We fix a
  sequence $\varepsilon_m$ such that $\sum_m\varepsilon_m<\infty$: the set $U$
  is the intersection of the set provided by Lemma \ref{lem:meas_diff}
  and the set of points $p$ where the limit:\def\mball #1,#2,#3.{\mu'\left(\sball #1,#2,#3.\right)}%
  \begin{equation}
    \label{eq:alberti_blow_up_p0}
    \lim_{r\searrow0}\frac{\mu'\left(\ball
        x,r_n.\right)}{\mu\left(\ball x,r_n.\right)}
  \end{equation}
  exists and is finite. Having fixed a
  point $p\in U$, we
  let $r_n$ be a sequence converging to $0$ such that the rescalings
  \begin{equation}
    \label{eq:alberti_blow_up_p1}
    \left(\frac{1}{r_n}X,\frac{\mu'}{\mu'\left(\ball p,r_n.\right)},\frac{\psi-\psi(p)}{r_n},p\right)
  \end{equation}
  converge to $\stdtang$ in the measured Gromov-Hausdorff sense. We
  let $X_n=\frac{1}{r_n}X$. As in the following we consider
  simultaneously different metric spaces, we will use subscripts to 
  denote objects which ``live'' in a given metric space, e.g.~$\sball
  X_n,p,R_0.$ denotes the ball of radius $R_0$ and center $p$ in the
  metric space $X_n$.
  \par By the theory of measured Gromov-Hausdorff convergence we can
  find a compact metric space $Z$, which is a convex compact subset of
  some Banach space 
  (e.g. $\ell^\infty$),
  which satisfies the following properties:
  \begin{description}
  \item[($Z$1)] There are isometric embeddings:
    \begin{equation}
      \label{eq:alberti_blow_up_z1}
      \begin{aligned}
        J_n:&\left(\scball X_n,p,4R_0.,p\right)\to(Z,q_Z)\\
        J_\infty:&\left(\scball Y,q,4R_0.,q\right)\to(Z,q_Z);
      \end{aligned}
    \end{equation}
    in the following we will often implicitly identify balls like
    $\sball X_n,p,r.$ and $\sball Y,q,r.$ with their images in $Z$;\vskip2mm
  \item[($Z$2)] There are compact sets $K_n,\tilde K_n\subset Z$ and a
    sequence $\eta_n\searrow0$ such that:
    \begin{equation}
      \label{eq:alberti_blow_up_z2}
      \begin{aligned}
        \scball X_n,p,R_0.&\subset K_n\subset\scball
        X_n,p,R_0+\eta_n.\\
        \scball X_n,p,2R_0.&\subset \tilde K_n\subset\scball
        X_n,p,2R_0+\eta_n.\\
        &\hzdst K_n,{\scball Y,q,R_0.}.\le\eta_n\\
        &\hzdst\tilde K_n,{\scball Y,q,2R_0.}.\le\eta_n,
      \end{aligned}
    \end{equation}
    where $\hzdst\cdot,\cdot.$ denotes the Hausdorff distance between
    subsets of $Z$;\vskip2mm
  \item[($Z$3)] There is an $\glip\psi.$-Lipschitz function
    $\psi_Z:Z\to\real^N$ such that, denoting by $\psi_{X_n}$ the
    restriction $\psi_Z|\scball X_n,p,2R_0.$ and by $\psi_Y$ the
    restriction $\psi_Z|\scball Y,q,2R_0.$, one has:
    \begin{equation}
      \label{eq:alberti_blow_up_z3}
      \begin{aligned}
        \psi_{X_n}\circ J_n&=\frac{\psi-\psi(p)}{r_n}\\
        \psi_Y\circ J_\infty&=\varphi;
      \end{aligned}
    \end{equation}
  \item[($Z$4)] Letting $\mu_n$ and $\mu_\infty$ denote, respectively,
    the measures
    \begin{equation}
      \label{eq:alberti_blow_up_z4}
\begin{aligned}
      \mpush J_n.&\frac{\mu'\mrest\sball X_n,p,R_0.}{\mu'\left(\sball
          X,p,r_n.\right)}\\
      \mpush J_\infty.&\nu\mrest\sball Y,q,R_0.,
    \end{aligned}
  \end{equation}
    one has $\mu_n\xrightarrow{\text{w*}}\mu_\infty$.
  \end{description}
\par We chose $Z$ convex to ``fill-in'' fragments to Lipschitz curves;
specifically, let $\curves(Z)$ denote the set of Lipschitz maps
$\gamma:K\to Z$, where $K\subset\real$ is a (possibly degenerate)
compact interval; we topologize $\curves(Z)$ with the Vietoris
topology. Let
\begin{equation}
  \label{eq:alberti_blow_up_p2}
  \fillfrag:\frags(Z)\to\curves(Z)
\end{equation}
be the map which extends a fragment $\gamma$ to a Lispchitz curve,
with domain the minimal compact interval $I(\gamma)$ containing
$\dom\gamma$, by extending $\gamma$ linearly on each component of
$I(\gamma)\setminus\dom\gamma$. The map $\fillfrag$ is continuous.
\def\gamman{\Gamma_{X_n}}%
\par Let $\gamman\subset\frags(X)$ denote the set of those
$[C_0,D_0]$-biLipschitz fragments which
intersect $\scball X,p,2R_0r_n.$; note that $\gamman$ is closed. We
define maps:
\begin{equation}
  \label{eq:alberti_blow_up_p3}
  \rprm_n:\gamman\to\frags(Z)
\end{equation}
by composing $J_n\circ\left(\gamma|\gamma^{-1}(\scball
  X,p,2R_0.)\right)$, where we naturally identify $\gamma$ with a
fragment in $X_n$, with the unique affine map $A_\gamma:\real\to\real$
which has dilating factor $\frac{1}{r_n}$ and which maps the point:
\begin{equation}
  \label{eq:alberti_blow_up_p4}
  \min\left\{t:t\in\gamma^{-1}(\scball X,p,2R_0.)\right\}
\end{equation} to $0$. Note that $\rprm_n$ is continuous.
\par We will now refer back to the map $\Psi$ defined in
(\ref{eq:frags_measures}), adding subscripts regarding the metric
space. From the definition of $\rprm_n$ we see that:
\begin{equation}
  \label{eq:alberti_blow_up_meas_rescaling}
  r_n\Psi_Z\left(\rprm_n(\gamma)\right)=\mpush
  J_n.\Psi_{X_n}(\gamma)\mrest\scball X_n,p,2R_0.
\end{equation}
\par Let $g\in C_c(Z)$ so that:
\begin{equation}
  \label{eq:alberti_blow_up_weak*conv}
  \lim_{n\to\infty}\int g\,d\mu_n=\int g\,d\mu_\infty;
\end{equation}
then\def\stddeno{\mu'\left(\sball X,p,r_n.\right)}
\begin{equation}
  \label{eq:alberti_blow_up_p5}
  \begin{split}
  \int g\,d\mu_n&=\frac{1}{\stddeno}\int_{\sball X_n,p,R_0.}g\circ
  J_n\,d\mu'\\
  &=\frac{1}{\stddeno}\int_{\sball X_n,p,R_0.}g\circ
  J_n\,d(\mu'_{\pars}+\mu'^c_{\pars});
\end{split}
\end{equation}
Note that
\begin{equation}
  \label{eq:alberti_blow_up_p6}
  \left|\frac{1}{\stddeno}\int_{\sball X_n,p,R_0.}g\circ
  J_n\,d\mu'^c_{\pars}\right|\le\|g\|_\infty\,\frac{\mu'^c_{\pars}\left(\sball
    X,p,r_nR_0.\right)}{\mu'\left(\sball X,p,r_n.\right)};
\end{equation}
for $n$ sufficiently large $r_nR_0\le s_m$ so that by
(\ref{eq:meas_diff}) and using that $\mu'$ is doubling we conclude
that:
\begin{equation}
  \label{eq:alberti_blow_up_est1}
  \lim_{n\to\infty}\frac{1}{\stddeno}\int_{\sball X_n,p,R_0.}g\circ
  J_n\,d\mu'^c_{\pars}=0.
\end{equation}
We also introduce some notation to deal with regularity in $Z$ and
$X$; so we let:
\begin{equation}
  \label{eq:alberti_blow_up_p7}
  \begin{aligned}
  {\rm
    PAR}_X(\varepsilon,S)&=\left(\psi,\cone,v_0,\sigma_0,[C_0,D_0],\tau_0,\varepsilon,S\right)\\
  {\rm
    PAR}_Z(\varepsilon,S)&=\left(\psi,\cone,v_0,\sigma_0,[C_0,D_0],\frac{4R_0}{C_0},\varepsilon,S\right);
\end{aligned}
\end{equation}\def\parsx{{\rm
    PAR}_X(\varepsilon_m,S_m)}\def\parsz{{\rm
    PAR}_Z(\varepsilon_m,S_m/r_n)}%
in particular, inspection of conditions
\textbf{(Reg1)}--\textbf{(Reg5)} shows that:
\begin{multline}
  \label{eq:alberti_blow_up_p8}
  \frac{1}{\stddeno}\int_{\sball X_n,p,R_0.}g\circ
  J_n\,d\mu'_{\parsx}\\=\frac{1}{\stddeno}\int_{\gamman}dP(\gamma)\int_Zg\,
   r_n\chi_{\sball X_n,p,R_0.}\,d\Psi_{\parsz}\left(\rprm_n(\gamma)\right).
\end{multline}\def\gammatn{\tilde\Gamma_{X_n}}%
\par Let $\gammatn$ be the Borel subset of those $\gamma\in\gamman$
such that:
\begin{equation}
  \label{eq:alberti_blow_up_p9}
  \chi_{\sball X_n,p,R_0.}\Psi_{\parsz}\left(\rprm_n(\gamma)\right)\ne0;
\end{equation} then (\ref{eq:alberti_blow_up_p9}) implies that there
is a $p_\gamma=\gamma(t)\in\left(\reg(\parsx)\right)_\gamma\cap\sball
X,p,r_nR_0.$. Note that the set ${\mathcal
  B}_{\gamma,n}=\gamma^{-1}\left(\scball X,p,2r_nR_0.\right)$ has
diameter at most $\frac{4R_0r_n}{C_0}$; let $a_\gamma,b_\gamma$ be
minimal such that the interval $[t_\gamma-a_\gamma,t_\gamma+b_\gamma]$
contains $\gamma^{-1}({\mathcal B}_{\gamma,n})$. For $n$-sufficiently
large one has $a_\gamma,b_\gamma\le S_m$ so that by \textbf{(Reg2)} the
$\varepsilon_m(a_\gamma+b_\gamma)$-neighbhourhood of ${\mathcal
  B}_{\gamma,n}$ contains $[t_\gamma-a_\gamma,t_\gamma+b_\gamma]$. A similar conclusion holds for the smallest interval
containing $\gamma^{-1}\left(\sball X_n,p,R_0.\right)$ from which we
get:
\begin{multline}
  \label{eq:alberti_blow_up_est2}
  r_n\left\|\Psi_{\parsz}\left(\rprm_n(\gamma)\right)\mrest\sball
  X_n,p,R_0. -
  \Psi_Z\left(\fillfrag\circ\rprm_n(\gamma)\right)\mrest\sball
  Z,q_z,R_0.\right\|\\
\le\varepsilon_mr_n\frac{2R_0}{C_0}.
\end{multline}
Note also that:
\begin{equation}
  \label{eq:alberti_blow_up_p10}
  \gamma\left(\dom\gamma\cap[t_\gamma-\frac{r_nR_0}{D_0},t_\gamma+\frac{r_nR_0}{D_0}]\right)
  \subset\sball X,p,2r_nR_0.;
\end{equation}
as for $n$ sufficiently large one has $\frac{r_nR_0}{D_0}\le S_m$, we
have:
\begin{equation}
  \label{eq:alberti_blow_up_lower_growth}
  r_n\Psi_{\parsz}\left(\fillfrag\circ\rprm_n(\gamma)\right)\left(\sball X_n,p,2r_nR_0.\right)\ge2(1-\varepsilon_m)\frac{r_nR_0}{D_0}.
\end{equation}
For $n$ sufficiently large we also have $\frac{3r_nR_0}{C_0}\le S_m$
which implies:
\begin{equation}
  \label{eq:alberti_blow_up_p10+1}
  \begin{aligned}
    \lebmeas\left(\dom\gamma\cap\left[t_\gamma,t_\gamma+\frac{3r_nR_0}{C_0}\right]\right)&\ge(1-\varepsilon_m)\frac{3R_0}{C_0}r_n\\
    \lebmeas\left(\dom\gamma\cap\left[t_\gamma-\frac{3r_nR_0}{C_0},t_\gamma\right]\right)&\ge(1-\varepsilon_m)\frac{3R_0}{C_0}r_n;
  \end{aligned}
\end{equation}
so we can find $s_{1,\gamma}\le t_\gamma\le s_{2,\gamma}$ with:
\begin{equation}
  \label{eq:alberti_blow_up_p10+2}
  \begin{aligned}
    \left|t_\gamma-s_{i,\gamma}\right|&\ge(1-\varepsilon_m)\frac{3r_nR_0}{C_0}\quad(\text{for
      $i=1,2$})\\
    d_X(p,\gamma(s_{i,\gamma}))&\ge\left((1-\varepsilon_m)\frac{3R_0}{C_0}-R_0\right)r_n\quad(\text{for
      $i=1,2$});
  \end{aligned}
\end{equation}
in particular, for $m$ sufficiently large
(\ref{eq:alberti_blow_up_p10+2}) implies that the maximum and minimum
point in ${\mathcal
  B}_{\gamma,n}$ are mapped by $\gamma$ outside of $\sball
X,p,\frac{3}{2}R_0.$. Thus, the endpoints of
$\fillfrag\circ\rprm_n(\gamma)$ lie out of $\sball Z,q_Z,\frac{3}{2}R_0.$.
\par We now obtain an upper estimate for $P(\gammatn)$ (note that we
assume that $n$ is sufficiently large depending on $m$):
\begin{equation}
  \label{eq:alberti_blow_up_probest}
  \begin{split}
  2(1-\varepsilon_m)\frac{r_nR_0}{D_0}P(\gammatn)&\le
  r_n\int_{\frags(X)}\Psi_{\parsz}\left(\rprm_n(\gamma)\right)\left(\sball X,p,2r_nR_0.\right)\,dP(\gamma)\\
  &\le\mball X,p,2R_0r_n.;
\end{split}
\end{equation}
in particular, using (\ref{eq:alberti_blow_up_est2}),
\begin{equation}
  \label{eq:alberti_blow_up_est3}
  \begin{split}
    \lim_{n\to\infty}\frac{1}{\mball
      X,p,r_n.}&\left|\int_{\gammatn}dP(\gamma)\int g\,r_n\chi_{\sball
        X_n,p,R_0.}d\Psi_{\parsz}\left(\rprm_n(\gamma)\right)\right.\\
    &\quad-\left.\int_{\gammatn}dP(\gamma)\int g\,r_n\chi_{\sball
        Z,q_Z,R_0.}\,d\Psi_{\parsz}\left(\fillfrag\circ\rprm_n(\gamma)\right)\right|\\
    &\le\limsup_{n\to\infty}\|g\|_\infty\,P(\gammatn)\,r_n\varepsilon_m\frac{2R_0}{C_0}\\
    &\le\limsup_{n\to\infty}\|g\|_\infty\frac{\varepsilon_m}{1-\varepsilon_m}\frac{D_0}{C_0}\frac{\mball
      X,p,2r_nR_0.}{\mball X,p,r_n.}\\
    &=O(\varepsilon_m),
  \end{split}
\end{equation}
where in the last step we used that $\mu'$ is doubling. As
$n\to\infty$ we can send $m\to\infty$ so that the left hand side of
(\ref{eq:alberti_blow_up_est3}) converges to $0$.
\par Let $\Omega_m\subset\curves(Z)$ denote the set of $D_0$-Lipschitz
curves so that there are a $\theta_\gamma\in[C_0,D_0]$ and a
$w\in\bar\cone$ such that (note the constant $C_2$ will be specified later):
\begin{description}
\item[($\Omega$1)] For all $s_1,s_2\in\dom\gamma$ one has:
  \begin{equation}
    \label{eq:alberti_blow_up_omega1}
    \left|\zdst
      \gamma(s_1),\gamma(s_2).-\theta_\gamma|s_1-s_2|\right|\le C_2\varepsilon_m;
  \end{equation}
\item[($\Omega$2)] The domain of $\gamma$ is a subset $\left[0,\frac{4R_0}{C_0}\right]$;\vskip2mm
\item[($\Omega$3)] The image of $\gamma$ is contained in the
  $C_2\varepsilon_m$-neighbourhood of $\scball X_n,p,2R_0.$;\vskip2mm
\item[($\Omega$4)] For all $s_1,s_2\in\dom\gamma$ one has:
  \begin{equation}
    \label{eq:alberti_blow_up_omega2}
    \left|\psi_Z\circ\gamma(s_1)-\psi_Z\circ\gamma(s_2)-w|s_1-s_2|\right|\le
    C_2\varepsilon_m;
  \end{equation}
\item[($\Omega$5)] For all $s_1,s_2\in\dom\gamma$ with $s_2\ge s_1$
  one has:
  \begin{equation}
    \label{eq:alberti_blow_up_omega3}
    \left\langle v_0,
      \psi_Z\circ\gamma(s_2)-\psi_Z\circ\gamma(s_1)\right\rangle\ge(\sigma_0-\varepsilon_m)\theta_\gamma
    \biglip\left(\langle v_0,\varphi\rangle\right)(p)(s_2-s_1)-C_2\varepsilon_m.
  \end{equation}
\end{description}
\par Note that the set $\Omega_m$ is compact. We also define
$\Omega_\infty$ by requiring in \textbf{($\Omega$3)} that $\gamma$
lies in $\scball Y,p,2R_0.$ and that the error term $\varepsilon_m$ is
replaced by $0$. In view of \textbf{(Reg1)}--\textbf{(Reg5)}, for an
appropriate choice of $C_2$ one has
$\fillfrag\circ\rprm_n(\gammatn)\subset\Omega_m$ for $n\ge N(m)$. If
we let $P_n$ denote the Radon measure on $\curves(Z)$:
\begin{equation}
  \label{eq:alberti_blow_up_11}
  P_n=\frac{1}{\mball X,p,r_n.}r_n\,\mpush\fillfrag\circ\rprm_n(\gamma).P\mrest\gammatn;
\end{equation}
we have that $P_n$ has support contained in $\Omega_m$ for $n\ge N(m)$. By (\ref{eq:alberti_blow_up_probest})
the total mass of $P_n$ is bounded by:
\begin{equation}
  \label{eq:alberti_blow_up_p11+1}
  \frac{D_0}{2(1-\varepsilon_m)R_0}\frac{\mball X,p,2R_0r_n.}{\mball X,p,r_n.}.
\end{equation}
Moreover, an
application of Ascoli-Arzel\'a shows that the set
$\Omega=\bigcup_m\Omega_m\cup\Omega_\infty$ is compact; we can thus
find a subsequence $n_m\ge N(m)$ such that
$P_{n_m}\xrightarrow{\text{w*}} Q_{R_0}$. The previous discussion on
the properties of the fragments in $\gammatn$ implies that the support $\spt
Q_{R_0}$ of $Q_{R_0}$ is a subset of $\conset .\subset\Omega_\infty$ and that point (2) in the statement of
this Theorem follows from (\ref{eq:alberti_blow_up_p11+1}).
We now observe that:\def\gammatn{\tilde\Gamma_{X_{n_m}}}%
\begin{equation}
  \label{eq:alberti_blow_up_12}
  \begin{split}
  \frac{1}{\mball X,p,r_n.}\int_{\gammatn}dP(\gamma)\int
  gr_n&\chi_{\sball
    Z,q_Z,R_0.}\,d\Psi_Z\left(\fillfrag\circ\rprm_{n_m}(\gamma)\right)\\
  &=\int_\Omega\,dP_{n_m}(\gamma)\int g\chi_{\sball Z,q_Z,R_0.}\,d\Psi_Z(\gamma);
\end{split}
\end{equation}
fix a $\xi\in(0,1)$, and let $\psi_\xi$ be a continuous function, which
takes values in $[0,1]$ and which equals $1$ on $\scball
Z,q_Z,R_0+\xi.$ and which vanishes out of $\sball Z,q_Z,R_0+2\xi.$;
then:
\begin{equation}
  \label{eq:alberti_blow_up_13}
  \left|\int_\Omega dP_{n_m}(\gamma)\int g\left(\chi_{\sball
       Z,q_Z,R_0.}-\psi_\xi\right)\,d\Psi_Z(\gamma)\right|\le P_{n_m}(\Omega)\|g\|_\infty\,\frac{2\xi}{C_0};
\end{equation}
as the map $\gamma\mapsto\int g\psi_\xi\,d\Psi_Z(\gamma)$ is continuous:
\begin{equation}
  \label{eq:alberti_blow_up_14}
  \lim_{m\to\infty}\int_\Omega dP_{n_m}(\gamma)\int
  g\psi_\xi\,d\Psi_Z(\gamma)=\int_\Omega dQ_{R_0}(\gamma)\int g\psi_\xi\,d\Psi_Z(\gamma).
\end{equation}
Also,
\begin{equation}
  \label{eq:alberti_blow_up_15}
  \left|\int_\Omega dQ_{R_0}(\gamma)\int
    g\psi_\xi\,d\Psi_Z(\gamma)-\int_\Omega dQ_{R_0}(\gamma)\int
    g\,d\Psi_Z(\gamma)\mrest\sball Y,q,R_0.\right|\le Q_{R_0}(\Omega)\|g\|_\infty\frac{\xi}{C_0};
  \end{equation}
  so we conclude that:
  \begin{equation}
    \label{eq:alberti_blow_up_est4}
    \lim_{m\to\infty}\int g\,d\mu_{n_m}=\int_{\conset.}dQ_{R_0}\int
    g\,d\Psi_Z(\gamma)\mrest\sball Y,q,R_0.;
  \end{equation}
  in particular, if we let:
  \begin{equation}
    \label{eq:alberti_blow_up_16}
    \Phi_\gamma=\Psi_Z(\gamma)\mrest\sball
    Y,q,R_0.=\frac{1}{\theta_\gamma}\hmeas ._\gamma,
  \end{equation} we get that $(Q_{R_0},\Phi)$ gives an Alberti representation of
  $\nu\mrest\sball Y,q,R_0.$. 
    It might be worth noting that in
  (\ref{eq:alberti_blow_up_16}) we used that  $\gamma$ is a geodesic
  with constant speed $\theta_\gamma$ and that the function
  $\gamma\mapsto\theta_\gamma$ is continuous.
\end{proof}
\begin{lemma}
  \label{lem:local_alberti_blow_up}
  There is a Borel $U\subset X$ with full $\mu$-measure such that, for
  each $p\in U$, for each $\stdtang\in\tang\stdpoint$, for each $R_0>0$ and
  each $v_0\in S(\ptnorm\,\cdot\,,p.)$, the measure $\nu\mrest\ball
  q,R_0.$ admits an Alberti representation $(Q_{R_0},\Phi)$ which
  satisfies the following conditions:
  \begin{enumerate}
  \item The measure $Q_{R_0}$ is a finite Radon measure with total
    mass at most 
    $$\frac{1}{2R_0}\left(\adoubling(\mu,p)\right)^{\log_2R_0+1}\,;$$\vskip2mm
  \item The support of $Q_{R_0}$ is contatined in a compact set
    $\conset .\subset\frags(Y)$ which consists of the unit-speed geodesic
    segments $\gamma$ whose domain lies in $[0,4R_0]$, whose image lies in
    $\clball q,2R_0.$, which have both endpoints lying outside of
    $\ball q,\frac{3}{2}R_0.$, which intersect $\clball q,R_0.$, and which satisfy:
    \begin{equation}
      \label{eq:local_alberti_blow_up_s1}
      \varphi\circ\gamma(s_2)-\varphi\circ\gamma(s_1)=(s_2-s_1)v_0\quad(\forall s_1,s_2\in\dom\gamma);
    \end{equation}
    \item For each $\gamma\in\conset .$ the measure $\Phi_\gamma$ is
      given by:
    \begin{equation}
      \label{eq:local_alberti_blow_up_s2}
      \Phi_\gamma=\hmeas ._\gamma\mrest\ball q,R_0..
    \end{equation}
  \end{enumerate}
\end{lemma}
\begin{proof}
  By Theorem \ref{thm:alb_speed_one} we can choose Borel maps
  $v_n:X\to TX$ \footnote{A choice of the representative of $TX$ is
    implied.} with $1\le\|v_n\|_{TX}\le1+\frac{1}{n}$ and such that:
  \begin{enumerate}
  \item For each $x\in X$ the closure of the set $\{v_n(x)\}_n$
    contains $S(\ptnorm\,\cdot\,,p.)$;\vskip2mm
  \item For each $n$ there is a measure $\mu'_n$ in the same measure
    class of $\mu$ and there are countably many disjoint compact sets
    $\{K_{n,\alpha}\}$ whose union has $\mu$-negligible complement and
    such that the function $v_n$ is constant on
    each $K_{n,\alpha}$;\vskip2mm
    \item  The measure $\mu'_n\mrest K_{n,\alpha}$ admits a simplified
      and 
    $(1,1+\frac{1}{n})$-biLipschitz Alberti representation $\albrep n,\alpha.$  in the
    $\psi$-direction of the cone $\cone\left(v_n\mrest
      K_{n,\alpha}/\|v_n\mrest K_{n,\alpha}\|_2,\pi/2n\right)$ with
    $\langle v_n\mrest K_{n,\alpha},\psi\rangle$-speed $\ge(1-1/n)$.
  \end{enumerate}
  Let $U_n$ be a Borel subset of $\bigcup_\alpha K_{n,\alpha}$ with
  full $\mu$-measure and such that, for each $p\in U_n\cap
  K_{n,\alpha}$, the conclusion of Theorem
  \ref{thm:alberti_blow_up} holds taking $\albrep.=\albrep n,\alpha.$. Let $U=\bigcap U_n$ 
 and fix $p\in U$ and $\stdtang\in\tang\stdpoint$. Choose a sequence $n_m$ such that
  $v_{n_m}(p)\to v_0$ and let $Q_{R_0,n_m}$, $\Phi_{n_m}$ and $\conset R_0,n_m.$ be the
  measures and sets of geodesics provided by Theorem
  \ref{thm:alberti_blow_up}. By Ascoli-Arzel\'a
  the set $\Omega=\conset.\bigcup_m \conset R_0,n_m.$ is a compact
  subset of $\frags(Y)$; as the measures $Q_{R_0,n_m}$ are uniformly
  bounded and supported in $\Omega$, we can pass to a subsequence such
  that $Q_{R_0,n_m}\xrightarrow{\text{w*}} Q_{R_0}$. Note also that
  $Q_{R_0}$ is supported in $\conset .$. For $g\in C_b(Y)$ one proves
  that:
  \begin{equation}
    \label{eq:local_alberti_blow_up_p2}
    \lim_{m\to\infty}\int_\Omega dQ_{R_0,n_m}(\gamma)\int
    g\,d(\Phi_{n_m})_\gamma=\int_{\conset .} dQ_{R_0}(\gamma)\int g\,d\Phi_\gamma
  \end{equation}
by using an argument similar to the one used to derive  the estimates
  (\ref{eq:alberti_blow_up_13}) and (\ref{eq:alberti_blow_up_15}). Thus the pair $(Q_{R_0},\Phi)$ provides the
  desired Alberti representation.
\end{proof}
\par To prove Theorem \ref{thm:measured_blow_up} we need to introduce
a bit more of terminology.  We can regard
parametrized Lipschitz curves in $Y$, whose domain is a possibly infinite interval of
$\real$, as elements of $\fellc(\real\times X)$ by identifying them with their
graph. \def\geo{\text{\normalfont Geo}(Y)}%
We denote by $\geo$ the set of unit speed geodesic segments,
half-lines or lines in $Y$; note that $\geo$ is a
$K_\sigma$. Moreover, if we let:
\begin{equation}
  \label{eq:map_geo}
  \begin{aligned}
    \Phi:\geo&\to\radon\\
    \gamma&\mapsto\hmeas ._\gamma,
  \end{aligned}
\end{equation}
then, for each $g\in C_c(Y)$, the map:
\begin{equation}
  \label{eq:eval_geo}
  \begin{aligned}
    \Phi_g:\geo&\to\real\\
    \gamma&\to\int gd\Phi_\gamma=\int_\real g\circ\gamma(t)\,dt
  \end{aligned}
\end{equation}
is continuous.
\begin{proof}[Proof of Theorem \ref{thm:measured_blow_up}]
  Let $U$ be the $\mu$-full measure subset provided by Lemma
  \ref{lem:local_alberti_blow_up} and consider $p\in U$ and
  $\stdtang\in\tang\stdpoint$. Fix a diverging sequence of radii $\{R_n\}$ with
  $R_n>2R_{n-1}$ and let $Q_{R_n}$ and $\conset R_n.$ be the
  corresponding measures and sets provided by Lemma
  \ref{lem:local_alberti_blow_up}. Note that $\conset R_n.$ can also
  be regarded as a compact subset of $\fellc$; in particular,
  for $i\le n$ we define the sets:
  \begin{equation}
    \label{eq:measured_blow_up_p1}
    \conset n,i.=\left\{\gamma\in\conset R_n.:\dist(\gamma,q)\in(R_{i-1},R_i]\right\},
  \end{equation} where we take $R_0=0$, and observe that the sets
  $\conset n,i.$ are Borel. We also consider
  the following Borel subsets of
  $\fellc\times\real$:\def\tconset#1.{\mathcal{\tilde S}_{\setbox0=\hbox{$#1\unskip$}\ifdim\wd0=0pt i,n
    \else #1\fi}}%
  \begin{equation}
    \label{eq:measured_blow_up_p2}
    \tconset .=\left\{(\gamma,t): t\in\dom\gamma, \gamma\in\conset i,n., d\left(\gamma(t),q\right)\in(R_{i-1},R_i]\right\};
  \end{equation}
  note that the sets $\tconset .$ have compact sections, i.e.~for
  $\gamma\in\fellc$, each section $(\tconset .)_\gamma$ is compact. By the
  Lusin-Novikov Uniformization Theorem \cite[Thm.~18.10]{kechris_desc}
  we can find Borel maps $\tau_{i,n}:\conset i,n.\to\real$ such that
  $(\gamma,\tau_{i,n}(\gamma))\in\tconset
  .$. \def\tran#1.{\text{\normalfont Tran}_{\setbox0=\hbox{$#1\unskip$}\ifdim\wd0=0pt n
      \else #1\fi}}%
  In particular, we can define a Borel map $\tran .:\geo\to\geo$ by
  requiring that for $\gamma\in\tconset .$ the geodesic segment
  $\tran .(\gamma)$ is the composition of $\gamma$ with the
  translation by $\tau_{n,i}(\gamma)$. Note that if $\gamma\in\tconset
  .$ the extremes of $\dom\gamma$ are at distance at least
  $\frac{3}{2}R_n-R_i$ from $\tau_{n,i}(\gamma)$ and so
  \begin{equation}
    \label{eq:measured_blow_up_p3}
    \left[-\frac{R_n}{2},\frac{R_n}{2}\right]\subset\dom\tran .(\gamma)\subset[-3R_n,3R_n].
  \end{equation}
  Let $Q_n=\mpush{\tran .}.Q_{R_n}$ and denote by $K(m,i)$ the set of
  geodesic segments $\gamma$ whose domain is contained in
  $[-3R_m,3R_m]$, and which intersect $\clball q,R_i.$ in a point
  $p_\gamma=\gamma(t_\gamma)$ where $t_\gamma$ is at distance at most
  $2R_i$ from $0$. The set $K(\infty,i)$ is defined similarly by
  requiring $\gamma$ to be a geodesic line. Note that the sets:
  \begin{equation}
    \label{eq:measured_blow_up_p4}
    K(i)=K(\infty,i)\cup\bigcup_m K(m,i)
  \end{equation} are compact and that $Q_n$ is concentrated on the
   set $\bigcup_{i\le n} K(n,i)$. We now obtain an upper bound on
  $Q_n\left( K(i)\right)$:\def\mball#1:#2,#3.{#1\left(\ball #2,#3.\right)}%
  \begin{equation}
    \label{eq:measured_blow_up_p5}
    \nu\left(\ball q,2R_i.\right)=\int_{\geo}\hmeas
    ._\gamma\left(\ball q,2R_i.\right)\,dQ_n(\gamma)\ge\int_{K(i)}\mball\hmeas ._\gamma:q,2R_i.\,dQ_n(\gamma);
  \end{equation}
  if $\gamma\in\tran .\left(\conset n,l.\right)$ and if $l\le i$ and
  $n\ge i$, one has $\mball\hmeas ._\gamma:q,2R_i.\ge\frac{R_i}{2}$ so
  from (\ref{eq:measured_blow_up_p5}) we obtain:
  \begin{equation}
    \label{eq:measured_blow_up_p6}
    Q_n\left(K(i)\right)\le2\frac{\mball\nu:q,2R_i.}{R_i}.
  \end{equation}
  In particular, we can pass to a subsequence and find a Radon measure
  $Q$ on $\geo$ such that for each $i$ one has $Q_n\mrest
  K(i)\xrightarrow{\text{w*}} Q\mrest K(i)$; in particular
  $Q_n\xrightarrow{\text{w*}} Q$. Moreover, as $Q_n$ is concentrated
  on $\bigcup_iK(n,i)$, the measure $Q$ has support contained in
  $\lines(\varphi,v_0)$.  To show that $(Q,\Phi)$ gives an
  Alberti representation of $\nu$ we take $g\in C_c(Y)$ and choose $i$
  sufficiently large so that $\spt g\subset\ball q,R_i.$:
  \begin{equation}
    \label{eq:measured_blow_up_p7}
    \int_Y g\,d\nu=\int_{K(i)}dQ_n(\gamma)\int g\,d\hmeas ._\gamma=\int_{K(i)}\Phi_g(\gamma)\,dQ_n(\gamma),
  \end{equation} and
  \begin{equation}
    \label{eq:measured_blow_up_p8}
    \lim_{n\to\infty}\int_{K(i)} \Phi_g(\gamma)\,dQ_n(\gamma)=\int_{K(i)}\Phi_g(\gamma)\,dQ(\gamma)=\int_{\geo}dQ(\gamma)\int
     g\,d\hmeas ._\gamma.
   \end{equation}
 \end{proof}
We now state an immediate consequence of Theorem
\ref{thm:measured_blow_up} in terms of the canonical maps from
blow-ups of $X$ to the fibres of $TX$.
\begin{defn}
  \label{defn:can_blow_map}
  Let $\stdtang\in\stdpoint$ be realized by choosing scales
  $r_n\searrow0$; suppose that the Lipschitz function $f:X\to\real$ is
  differentiable at $p$ with respect to the $\{\psi^i\}_{i=1}^N$. Then
  the maps:
  \begin{equation}
    \label{eq:can_blow_map1}
    \frac{f-f(p)}{r_n}:\frac{1}{r_n}X\to\real
  \end{equation}
  converge to the map $g:Y\to\real$ given by:
  \begin{equation}
    \label{eq:can_blow_map2}
    g(y)=\sum_{i=1}^N\frac{\partial f}{\partial \psi^i}(p)\varphi(y).
  \end{equation}
  In particular, we obtain a canonical map $E:Y\to T_pX$ by letting:
  \begin{equation}
    \label{eq:can_blow_map3}
    \left\langle\sum_{i=1}^Na_i\,d\psi^i\mid_p, E(y)\right\rangle=\sum_{i=1}^Na_i\varphi(y).
  \end{equation}
\end{defn}
\begin{corollary}
  \label{cor:can_surj}
  Let $p\in U$ be a point where the conclusion of Theorem
  \ref{thm:measured_blow_up} holds; then the canonical map $E:Y\to T_pX$ is
  surjective. Moreover, for each $\tilde q\in Y$ there is a line
  $\gamma\in\lines(\varphi,v_0)$ passing through it, and there is a
  $c_\gamma\in\real$ such that:
  \begin{equation}
    \label{eq:can_surj_s1}
    E\left(\gamma(t)\right)=v_0(t-c_\gamma)\quad(\forall t\in\real).
  \end{equation}
\end{corollary}
Corollary \ref{cor:can_surj} generalizes \cite[Sec.~13]{cheeger} where
the surjectivity of the map $E$ was proven for the case in which
$(X,\mu)$ is a PI-space. The surjectivity of the map $E$ in the case
in which $(X,\mu)$ is a differentiability space has already been
proven in \cite{deralb,david_difftang_ahlfors}.
\subsection{Harmonicity of blow-up functions}
In this subsection we prove Corollary \ref{cor_blow_ups_harmonic}.
\begin{proof}[Proof of Corollary \ref{cor_blow_ups_harmonic}]
Let $u:X\ra\R$ be a Lipschitz function, and suppose that
$x\in X$ a point of differentiability of $u$
where $x$ is as in the statement of Theorem \ref{tnm_intro_blow_up}.
Choose a unit vector $\xi\in (T_xX,\tnrmname(x))$ supporting $du\in T_x^*X$,
i.e. 
\begin{equation}
du(\xi)=\ctnrm du(x).=\ctnrm du(x).\cdot\tnrm\xi..
\end{equation}
Since $x$ is  a point of differentiability, the blow-up $\hat u$ of $u$ will be of the 
form $\hat u=\al\circ\hat \phi_i$ for some $\al\in T_x^*X$.  Now consider an Alberti
representation for $\hat \mu$ as in Theorem \ref{tnm_intro_blow_up}\,(2), which is supported
on unit speed geodesics $\ga$ with $(\hat u_i\circ \ga)'\equiv \xi$.  Fix such a
unit speed geodesic $\ga:\R\ra \hat X$.  Note that for all $t\in \R$
\begin{equation}
(\hat u\circ\ga)'(t)=\al\left((\hat \phi_i)'(t)\right)=\al(\xi)=\ctnrm\hat u(x).
\quad
\text{and}
\quad\Lip(\hat u)(\ga(t))=(\Lip(\hat u\circ\ga))(t)\,.
\end{equation}
If $v:\hat X\ra \R$ is Lipschitz
and agrees with $\hat u$ outside a compact subset $K\subset \hat X$, then 
for all $t\in \R$ we have $\Lip(v\circ\ga)(t)\leq \Lip(v)(\ga(t))$, and for $\L$-a.e. 
$t\in \R\setminus K$ we have
\begin{equation}
\Lip(v\circ\ga)=|(v\circ\ga)'(t)|=|(\hat u\circ\ga)'(t)|=\Lip(\hat u\circ\ga)(t)
=\Lip(\hat u)(\ga(t))\,.  
\end{equation}
Therefore if $\ga^{-1}(K)\subset [a,b]$, then 
\begin{equation}
  \begin{split}
\int_{\ga^{-1}(K)}\left[\Lip(v)(\ga(t))  \right]^p\,dt
&-\int_{\ga^{-1}(K)}\left[\Lip(\hat u)(\ga(t))  \right]^p\,dt\\
&\geq\int_{\ga^{-1}(K)}\left[\Lip(v\circ\ga)(t))  \right]^p\,dt\\
&\quad-\int_{\ga^{-1}(K)}\left[\Lip(\hat u\circ\ga)(t))  \right]^p\,dt\\
&=\int_{[a,b]}\left[\Lip(v\circ\ga)(t))  \right]^p\,dt\\
&\quad-\int_{[a,b]}\left[\Lip(\hat u\circ\ga)(t))  \right]^p\,dt\\
&\ge\int_{[a,b]}|(v\circ\ga)'(t)|^p\,dt\\
&\quad-\int_{[a,b]}|(\hat u\circ\ga)'(t)|^p\,dt\\
&\geq 0
\end{split}
\end{equation}
by Jensen's inequality.  Integrating this with respect to the measure on curves coming
from the Alberti representation, we get that
\begin{equation}
\int_K\left[\Lip(v)\right]^p\,d\hat \mu\geq \int_K\left[\Lip(\hat u)\right]^p\,d\hat \mu\,.
\end{equation}
\end{proof}

\section{Lipschitz mappings $f:X\ra Z$ and metric differentiation}
 \label{sec:lipmaps_metricdiff}

 \newcount\maskmetdiff
 \maskmetdiff=0
 \ifnum\maskmetdiff>0{
\begin{enumerate}
\item Motivation: brief discussion of what happens when $(Z,\nu)$ is a PI space
and $f_*\mu\ll \nu$.
\item The canonical measurable subbundle  $\W\subset T^*X$
determined by $f^*(\Lip(Z))$, and the annihilator $\W^\perp\subset TX$.
\item Seminorms on $TX$, $TX/\W^\perp$, $\W$.
\item State results for the pseudodistance $\rho=f^*d_Z$.
\end{enumerate}
\vskip0.5cm
\vskip0.5cm
\begin{enumerate}
\item The canonical measurable subbundle $\W_{\varrho}$ determined by a Lipschitz
  compatible seminorm $\varrho$; canonical norm on $TX$ induced by
  $\W_\varrho$;
\item Given a dense set $\dset$, the subbundle $\W_{D_X,\varrho}$ and
  equality with $\W_\varrho$ and of the induced norms: proof(1):
  Hahn-Banach; proof(2): weak* continuity; remark that this gives a
  much stronger statement than keith's result in IUMJ;
\item Specialize the discussion to a map $f:X\to Z$; what happens if
  $(X,\mu)$ and $(Z,\nu)$ are differentiability spaces and
  $f_*\mu\ll\nu$; one might also discuss some specific illustrative example;
\item State and sketch proofs of Sec.~7 for a map $f:X\to Z$; might
  require a discussion of ultrafilters;
\item One might move the new proof of $g_f=\biglip f$ in PI-spaces in
  the next section; one might discuss renorming as another application;
\end{enumerate}
}\fi
\def\sbs{\W}
\def\sbr{\W_{\rdname}}
\def\scoll{\text{\normalfont Sub}(\rdname)}
 \subsection{The canonical subbundle determined by a pseudodistance}
\label{subsec:can_sub_sem}

In this subsection we associate a canonical subbundle $\sbr$ of $T^*X$
to a Lipschitz compatible pseudometric $\rdname$; we denote by
$C_\rdname$ the Lipschitz constant of $\rdname$, that is, $\rdname\le
C_\rdname\xdname$.
\begin{defn}
  \label{defn:can_sub_countable_set}
  Let $\Phi$ be a countable set of $\rdname$-Lipschitz functions and
  let $V$ be a $\Phi$-differentiability set; we define a subbundle
  $\sbs_\Phi$ of $T^*X$ by letting, for $x\in V$, the fibre
  $\sbs_\Phi(x)$ equal the linear span of $\left\{df(x):f\in\Phi\right\}$.
\end{defn}
The collection $\scoll$ of subbundles associated to countable sets of
$\rdname$-Lipschitz functions has a partial order $\preceq$: we say
that $\sbs_\Phi\preceq\sbs_{\Phi'}$ if for $\mu$-a.e.~$x\in X$ one has
$\sbs_\Phi(x)\subseteq\sbs_{\Phi'}(x)$.
\begin{lemma}
  \label{lem:can_sub_sem}
  The poset $(\scoll,\preceq)$ contains a
  maximal element $\sbr$ which we call\/ {\normalfont\bf the canonical subbundle
  associated to $\rdname$}.
\end{lemma}
\begin{proof}
  As the constructions depend only on the measure class of $\mu$, we
  can assume that $\mu$ is a probability measure. We basically follow
  the argument used in the proof of Lemma \ref{lem:sup_of_seminorms}: to each $\sbs_\Phi\in\scoll$
  we associate a ``size'', which is the expectation of the random
  variable $\dim\sbs_\Phi$:
  \begin{equation}
    \label{eq:can_sub_sem_p1}
    \|\sbs_\Phi\|=\int\dim\sbs_\Phi(x)\,d\mu(x);
  \end{equation}
  note that the finite dimensionality of $T^*X$ implies that
  \begin{equation}
    \label{eq:can_sub_sem_p2}
    S=\sup_{\sbs_\Phi\in\scoll}\|\sbs_\Phi\|<\infty.
  \end{equation}
  Let $\sbs_{\Phi_n}$ be a maximizing sequence and let
  $\Phi_\infty=\bigcup_n\Phi_n$; then
  $\|\sbs_{\Phi_\infty}\|=S$. Suppose, by contradiction, that for some
  $\sbs_\Phi\in\scoll$ one has
  $\sbs_\Phi\not\preceq\sbs_{\Phi_\infty}$; then there is a positive
  measure set $V$ such that, if $x\in V$, one has
  \begin{equation}
    \label{eq:can_sub_sem_p3}
    \sbs_{\Phi_\infty}(x)\subsetneq\span\left(\sbs_\Phi(x)\cup\sbs_{\Phi_\infty}(x)\right);
  \end{equation}
  but then we obtain the contradiction $\|\sbs_{\Phi\cup\Phi_{\infty}}\|>S$.
\end{proof}
Let $\dset\subset X$ be a countable dense set and
$\Phi_{\dset,\rdname}=\left\{\rdstp x.:x\in\dset\right\}$; we let
\def\sbd{\sbs_{\dset,\rdname}}
$\sbd=\sbs_{\Phi_{\dset,\rdname}}$. We now show that $\sbd$ equals
$\sbr$: this is a stronger result than
\cite[Thm.~2.7]{keith04bis} because it applies to 
subbundles associated to Lipschitz compatible pseudometrics.
\begin{thm}
  \label{thm:distance_functions_span}
  For any countable dense set $\dset\subset X$ we have $\sbd=\sbr$.
\end{thm}
We offer two conceptually different proofs of Theorem
\ref{thm:distance_functions_span}.
\begin{proof}[Proof of Theorem \ref{thm:distance_functions_span} via a
  measurable Hahn-Banach]
  As $\sbd\preceq\sbr$, assume by contradiction that there is a
  positive measure Borel set $U$ such that, for each $x\in U$ one has:
  \begin{equation}
    \label{eq:distance_functions_span_p1}
    \sbd(x)\subsetneq\sbr(x).
  \end{equation}
  Without loss of generality we can assume that there are
  $1$-Lipschitz functions $\{\phi_i\}_{i=1}^N$ such that
  $(U,\{\phi_i\}_{i=1}^N)$ is a differentiability chart. Let
  \begin{multline}
    \label{eq:distance_functions_span_p2}
    \tilde U=\biggl\{(x,a)\in U\times\real^N:
    \sum_{i=1}^Na_id\phi_i(x)\in\sbr(x)\cap
    S(\ctnrmname(x)),\quad\text{and}\\
    \dist_{\ctnrmname(x)}\left(\sum_{i=1}^Na_id\phi_i(x),\sbd(x)\right)\ge\frac{1}{2}\biggr\}.
  \end{multline}
  Note that the distances in the fibre $T_x^*X$ are computed with
  respect to the norm $\ctnrmname(x)$. The set $\tilde U$ is Borel and
  by (\ref{eq:distance_functions_span_p1}) for each $x\in U$ the
  section $\tilde U_x$ is nonempty (compare \cite[Lem.~4.22]{rudin-functional}) and compact. By the Lusin-Novikov
  Uniformization Theorem \cite[Thm.~18.10]{kechris_desc} we obtain a
  unit-norm Borel section $\omega$ of $\sbr\mid U$ satisfying:
  \begin{equation}
    \label{eq:distance_functions_span_p3}
    \dist_{\ctnrmname(x)}\left(\omega(x),\sbd(x)\right)\ge\frac{1}{2}\quad(\forall
    x\in U).
  \end{equation}
  Using Hahn-Banach in each fibre $T_x^*X$ and an argument similar to
  the one above, we obtain a Borel section $\xi$ of $TX\mid U$ such
  that:
  \begin{equation}
    \begin{aligned}
    \label{eq:distance_functions_span_p4}
    \tnrm\xi.&\le2;\\
    \langle\omega(x),\xi(x)\rangle&=1\quad(\forall x\in U),
  \end{aligned}
\end{equation}
and such that $\xi(x)$ is annihilated by the functionals in
$\sbd(x)$. Up to shrinking $U$ we can assume that there are $\tilde
N\le N$, $(1,\rdname)$-Lispchitz functions $\{\psi_i\}_{i=1}^{\tilde
  N}$ and bounded Borel maps $s_i:U\to\real$ satisfying:
\begin{equation}
  \label{eq:distance_functions_span_p5}
  \begin{aligned}
    \|s_i\|&\le C;\\
    \omega=&\sum_{i=1}^{\tilde N}s_id\psi_i.
  \end{aligned}
\end{equation}
Let $\gfun$ contain the $\phi_i$, the $\psi_i$ and the components of
the chart functions; let $\gbor$ contain $\chi_U$, the $s_i$ and the
characteristic functions of the charts; let $\gsem$ contain $\xdname$
and $\rdname$; by Theorem \ref{thm:Sus_dir_dens} we obtain an 
$\mu$-measurable
subset $V\subset U$ of full $\mu$-measure with $G_x\gquad$ containing
a dense set of directions in $T_xX$ for each $x\in V$. In particular,
fix $\varepsilon>0$ and
let $\gamma'(t)\in T_xX$ be an $\gquad$-generic velocity vector such
that:
\begin{equation}
  \label{eq:distance_functions_span_p6}
  \tnrm\xi(x)-\gamma'(t).\le\varepsilon;
\end{equation}
then
\begin{equation}
  \label{eq:distance_functions_span_p7}
  \left|\left\langle d\rdstp x.,\gamma'(t)\right\rangle\right|\le
  C_\rdname \varepsilon+\left|\left\langle d\rdstp
      x.,\xi(x)\right\rangle\right| = C_\rdname \varepsilon;
\end{equation}
by Theorem \ref{thm:met_diff_norm} we conclude that:
\begin{equation}
  \label{eq:distance_functions_span_p8}
  \rmd\gamma(t)\le C_\rdname \varepsilon.
\end{equation}
However, $\tnrm\gamma'(t).\le 2+\varepsilon$ and so
\begin{equation}
  \label{eq:distance_functions_span_p9}
  \left\langle\omega,\gamma'(t)\right\rangle\ge\frac{1}{2}-\varepsilon(2+\varepsilon);
\end{equation}
note also that
\begin{equation}
  \label{eq:distance_functions_span_p10}
  \left|\left\langle\omega,\gamma'(t)\right\rangle\right|=\left|\sum_{i=1}^{\tilde
      N}s_i\left(\gamma(t)\right)\,(\psi_i\circ\gamma)'(t)\right|
  \le\tilde N C\max_i\left|(\psi_i\circ\gamma)'(t)\right|;
\end{equation}
now choose $s_n\searrow0$ with $t+s_n\in\dom\gamma$; we have:
\begin{equation}
  \label{eq:distance_functions_span_p11}
  \left|\psi_i\circ\gamma(t+s_n)-\psi_i\circ\gamma(t)\right|\le\rdst\gamma(t+s_n),\gamma(t).\le
  o(s_n)+\int_{[t,t+s_n]\cap \dom\gamma}\rmd\gamma(\tau)\,d\tau;
\end{equation}
dividing by $s_n$ and letting $n\nearrow\infty$ we get:
\begin{equation}
  \label{eq:distance_functions_span_p12}
  \left|(\psi_i\circ\gamma)'(t)\right|\le\rmd\gamma(t).
\end{equation}
Combining (\ref{eq:distance_functions_span_p8}),
(\ref{eq:distance_functions_span_p9}) and
(\ref{eq:distance_functions_span_p12}) we conclude that:
\begin{equation}
  \label{eq:distance_functions_span_p13}
  \frac{1}{2}-\varepsilon(2+\varepsilon)\le\tilde N\,C\,C_\rdname \varepsilon
\end{equation} which yields a contradiction if $\varepsilon$ is
sufficiently small.
\end{proof}
\begin{proof}[Proof of Theorem \ref{thm:distance_functions_span} via
  Weaver derivations]
  We show that if $K\subset X$ is compact and if $f$ is
  $\rdname$-Lipschitz, for $\mu$-a.e.~$x\in K$ one has
  $df(x)\in\sbd(x)$. Fix $n\in\natural$ and choose a finite susbset $\{x_k\}_{k\in
    I_n}\subset\dset$ such that each $x\in K$ lies within
  $\xdname$-distance at most $\frac{1}{n}$ from some $x_k$. To fix the
  ideas, suppose that $f$ is $(C,\rdname)$-Lipschitz and define
  $f_n:K\to\real$ by:
  \begin{equation}
    \label{eq:distance_functions_span_dp1}
    f_n(x)=\inf\left\{f(x_k)+C\rdst x,x_k.:k\in I_n\right\}.
    \end{equation}
    The functions $\{f_n\}_n$ are uniformly $(C,\rdname)$-Lipschitz
    and hence uniformly $(C\,C_\rdname,\xdname)$-Lipschitz. By
    \cite[Thm.~4.1]{deralb} the exterior derivative operator
    $d$ associated to the diffentiable structure is weak*
    continuous. In particular, let $L^2(\mu\mrest K,T^*X)$ denote the
    $L^2$-space of sections of $T^*X\mid K$. Note that the dual of
    $L^2(\mu\mrest K,T^*X)$ is $L^2(\mu\mrest K,TX)$ and that these
    spaces are both reflexive by finite dimensionality of $T^*X$. Then
    as the $f_n\to f$ pointwise in $K$, we have that
    $df_n\to df$ weakly in $L^2(\mu\mrest K,T^*X)$, and Mazur's Lemma
    and a standard argument give tail convex combinations $g_n$ of the
    functions $f_n$ with $dg_n\to df$ $\mu\mrest K$-a.e. So the proof
    is completed if we show that each $dg_n$ is a section of $\sbd$,
    which happens if each $df_n$ is a section of $\sbd$. But for each
    $n$ there are closed subsets $\{C_i\}_{i\in
      I_n}$ of $K$, such that $f_n\mid C_i=f(x_i)+C\rdstp x_i.$, which
    gives $df_n\mid C_i=Cd\rdstp x_i.$.
  \end{proof}
We now associate to $\sbr$ two a priori different norms on
$TX$. Roughly speaking, we maximize the seminorms induced by sections
of $\sbr$. Recall that if $f$ is $\rdname$-Lipschitz we can define the
``big Lip'' with respect to $\rdname$:
\begin{equation}
  \label{eq:sem_biglip}
  \biglipr f(x) =
  \limsup_{r\searrow0}\sup\left\{\frac{\left|f(x)-f(y)\right|}{r}:
    \rdst x,y.\le r\right\},
\end{equation}
and that the map $x\mapsto\biglipr f(x)$ is
Borel. \def\secr{\text{\normalfont Sec}_1(\rdname)}%
\par Let $\secr$ denote the set of those sections $\omega$ of $\sbr$ which
are locally the differential of a $(1,\rdname)$-Lipschitz function;
i.e.~$\omega\in\secr$ if and only if there are countably many disjoint
Borel sets $\{V_\beta\}_\beta$ and countably many
$(1,\rdname)$-Lipschitz functions $\{f_\beta\}_\beta$ such that
$\mu\left(X\setminus\bigcup_\beta V_\beta\right)=0$ and $\omega\mid
V_\beta = df_\beta\mid V_\beta$. To each $\omega\in\secr$ we associate
a seminorm $p_\omega$ on $TX$ by letting:
\begin{equation}
  \label{eq:seminorm_1_sec}
  p_\omega(v)=\left|\langle \omega,v\rangle\right|.
\end{equation}
We observe that $p_\omega\le C_\rdname\tnrmname$ and denote by
$\rLIPnrmname$ the essential supremum (Lemma \ref{lem:sup_of_seminorms}) of the collection
$\{\omnrmname\}_{\omega\in\secr}$.
\par Another way of obtaining seminorms on $TX$ is to use arbitrary
sections of $\sbr$ and rescale them by the local $\rdname$-Lipschitz
constant; note, however, that if $u,v$ are both $\rdname$-Lipschitz,
one can have  $du=dv$ and $\biglipr u\ne\biglipr v$ on a set of
positive measure. We are thus led to use a slightly more complicated
framework. \def\secs{\text{\normalfont Sec}_*(\rdname{})}
Let $\secs$ denote the set of countable pairs
$\tilde\omega=\left\{(V_\beta,f_\beta)\right\}$ where the $V_\beta$
are disjoint Borel sets with $\mu\left(X\setminus\bigcup_\beta
  V_\beta\right)=0$, and the $f_\beta$ are $\rdname$-Lipschitz
functions. To each $\tilde\omega\in\secs$ we associate a seminorm
$p_{\tilde\omega}$ on $TX$ by letting, for $x\in V_\beta$ and $v\in
T_xX$:
\begin{equation}
  \label{eq:seminorm_*_sec}
  p_{\tilde\omega}(v)=
  \begin{cases}
    0 &\text{if $\biglipr f_\beta(x)=0$}\\
    \frac{\left|\langle df_\beta(x),v\rangle\right|}{\biglipr
      f_\beta(x)} & \text{otherwise};
  \end{cases}
\end{equation}
we denote by $\rLipnrmname$ the essential supremum (Lemma \ref{lem:sup_of_seminorms})
of the collection  $\{\omnrmname\}_{\omega\in\secs}$.
\begin{thm}
  \label{thm:can_seminorm}
  Let $\dset\subset X$ be a countable dense set. Then one has:
  \begin{equation}
    \label{eq:can_seminorm_s1}
    \dxnrmname=\rLIPnrmname=\rLipnrmname;
  \end{equation}
  in particular, if $\dset'\subset X$ is another countable dense set:
  \begin{equation}
    \label{eq:can_seminorm_s2}
    \dxnrmname=\localnrmname;
  \end{equation}
  in the sequel, we will denote the {\normalfont\bf canonical norm}
  (\ref{eq:can_seminorm_s1}) by $\rnrmname$.
\end{thm}
\begin{proof}
  Each pseudodistance function $\rdstp x.$ gives rise to an element of
  $\secr$ and so $\dxnrmname\le\rLipnrmname$; to each $\omega\in\secr$
  one can associate $\tilde\omega\in\secs$ with $p_\omega\le
  p_{\tilde\omega}$ and so $\rLIPnrmname\le\rLipnrmname$. We thus just
  prove that:
  \begin{equation}
    \label{eq:can_seminorm_p1}
    \rLipnrmname\le\dxnrmname.
  \end{equation}
  It suffices to show that for any
  $\tilde\omega=\left\{(V_\beta,f_\beta)\right\}\in\secs$ one has
  \begin{equation}
    \label{eq:can_seminorm_p1bis}
    p_{\tilde\omega}\le\dxnrmname.
  \end{equation}
  Let $\gfun$ contain the components
  of the chart functions and the functions $\{f_\beta\}_\beta$; let $\gbor$
  contain the characteristic functions of the charts and the characteristic functions
  $\{\chi_{V_\beta}\}_\beta$; let $\gsem$ contain $\xdname$ and $\rdname$. Let
  $V$ be an $\{f_\beta\}_\beta$-differentiability set and fix $\beta$; let
  $V'_\beta=V\cap V_\beta$; by Theorem \ref{thm:Sus_dir_dens} there is
  a full $\mu$-measure $\mu$-measurable subset $W_\beta\subset V'_\beta$ such
  that, for each $x\in W_\beta$ the set of $\gquad$-generic velocity
  vectors contains a dense set of directions. In particular, for each
  $v\in T_xX$ and $\varepsilon>0$ we can find an $\gquad$-generic
  velocity vector $\gamma'(t)\in T_xX$ with $\tnrm
  v-\gamma'(t).\le\varepsilon$. Assume that $\biglipr f_\beta(x)>0$;
  note that the derivative $(f_\beta\circ\gamma)'(t)$ exists and is
  approximately continuous at $t$. Without loss of generality assume
  that $(f_\beta\circ\gamma)'(t)\ne0$; then we can find
  $s_n\searrow0$ such that $t+s_n\in\dom\gamma$ and
  $\rdst\gamma(t+s_n),\gamma(t).=r_n>0$. We now obtain the estimate:
  \begin{equation}
    \label{eq:can_seminorm_p2}
    \begin{split}
      \left|(f_\beta\circ\gamma)(t+s_n)-(f_\beta\circ\gamma)(t)\right|&\le\\&\le\sup\left\{\frac{\left|f_\beta(y)-f_\beta(x)\right|}{r_n}:
        \rdst y,x.\le
        r_n\right\}\\&\quad\times\rdst\gamma(t+s_n),\gamma(t).\\
      &\le\left(\biglipr
        f_\beta(x)+O(1/n)\right)\,\biggl(\int_{[t,t+s_n]\cap\dom\gamma}\rmd\gamma(\tau)\,d\tau\\
      &+o(s_n)\biggr);
    \end{split}
  \end{equation}
  dividing by $s_n$ and letting $n\nearrow\infty$ we get, by
  approximate continuity of $\rmd\gamma$ at $t$:
  \begin{equation}
    \label{eq:can_seminorm_p3}
    \left|(f_\beta\circ\gamma)'(t)\right|\le\biglipr f_\beta(x)\,\rmd\gamma(t).
  \end{equation}
  Now Theorem \ref{thm:met_diff_norm} implies that
  $\rmd\gamma(t)=\dxnrm\gamma'(t).$ and so:
  \begin{equation}
    \label{eq:can_seminorm_p4}
    \begin{split}
      \left|\left\langle
          df_\beta,\gamma'(t)\right\rangle\right|&\le\biglipr
      f_\beta(x)\dxnrm\gamma'(t).\\
      &\le\biglipr f_\beta(x)\,\dxnrm v.+\varepsilon C_\rdname\,\biglipr f_\beta(x);
    \end{split}
  \end{equation}
  let $L$ denote the global Lipschitz constant of $f_\beta$; then:
  \begin{equation}
    \label{eq:can_seminorm_p5}
    \left|\langle df_\beta,v\rangle\right|\le\biglipr
    f_\beta(x)\,\dxnrm v.+\varepsilon C_\rdname\,\biglipr
    f_\beta(x)+\varepsilon L\tnrm v.;
  \end{equation}
  so (\ref{eq:can_seminorm_p1bis}) follows by letting $\varepsilon\searrow0$.
\end{proof}

\subsection{Metric Differentiation for Lipschitz maps}
\label{subsec:lip_map_met_diff}

We now reformulate the results of the previous subsection for a
Lipschitz map $F:X\to Z$; throughout this subsection $\rdname$ will
denote the pull-back pseudometric $F^*\zdname$. Putting together
Theorems \ref{thm:distance_functions_span} and \ref{thm:can_seminorm}
we obtain: \def\sbf{\W_{F}}%
\begin{thm}
  \label{thm:lip_map_subbundle}
  Associated to the map $F$ there is a canonical subbundle $\sbf$ of
  $T^*X$ such that:
  \begin{enumerate}
  \item For each $g\in F^*\left(\lipfun Z.\right)$ (i.e.~$g=h\circ F$
    for some $h\in\lipfun Z.$) the section $dg$ lies in $\sbf$;\vskip2mm
  \item For each countable dense set $\dset\subset X$ the subbundle
    $\sbf$ coincides with the subbundle spanned by the sections
    $\left\{d\rdstp x.: x\in\dset\right\}$.\vskip2mm
  \end{enumerate}
  Suppose now that $\gfun$ contains the components of the chart
  functions of $(X,\mu)$, that $\gbor$ contains the characteristic
  functions of the charts, and suppose also that $\gsem$ contains the
  pseudometric $\rdname$. The subbundle $\sbf$ induces a canonical
  seminorm $\fnrmname=\rnrmname$ on $TX$ such that, for each
  $\gtrip$-generic velocity vector $\gamma'(t)$ one has:
  \begin{equation}
    \label{eq:lip_map_subbundle_s1}
    \fnrm\gamma'(t).=\lim_{s\to0}\frac{\zdst F\circ\gamma(t+s),F\circ\gamma(t).}{|s|}.
  \end{equation}
\end{thm}
\begin{remark}
  \label{rem:equivalent_formulations}
  In practice, it does not matter whether metric differentiation is
  formulated in terms of pseudometrics or Lipschitz maps. In fact,
  consider a Lipschitz compatible pseudometric $\rdname$ on $X$ and
  associate to it the Lipschitz map:
  \begin{equation}
    \label{eq:equivalent_formulations}
    \begin{aligned}
      F:X&\to l^\infty(\dset)\\
      y&\mapsto\left\{\rdstp x.(y)\right\}_{x\in\dset};
    \end{aligned}
  \end{equation}
  then we get $\rnrmname.=\fnrmname$ and $\sbr=\sbf$.
\end{remark}
We now specialize the discussion to the case in which $(Z,\nu)$ is a
differentiability space; throughout the remainder of this subsection
we will fix choices of countable dense sets $\dset\subset X$ and
$D_Z\subset Z$. The case of interest is when the measure $\mpush
F.\mu$ is absolutely continuous with respect to $\nu$. Using the
Radon-Nikodym Theorem we can find a Borel subset $V_0\subset Z$ such
that $\mpush F.\mu\mrest V_0$ and $\nu\mrest V_0$ are in the same
measure class. The case of interest is when $\nu(V_0)>0$, which we
will assume throughout the remainder of this subsection.
\par Let $U_0=F^{-1}(V_0)$ and suppose that $g\in\lipfun Z.$ is
differentiable at $z_0$ with respect to the Lipschitz functions
$\{\psi^i\}_{i=1}^M$; suppose now that $z_0=F(x_0)$ and that the
functions $\{\psi^i\circ F\}_{i=1}^M$ are differentiable at $x_0$ with
respect to the functions $\{\phi^j\}_{j=1}^N$. We then obtain the chain
rule:
\begin{equation}
  \label{eq:chain_rule}
  g\circ F(x)-g\circ F(x_0)=\sum_{i=1}^M\sum_{j=1}^N\frac{\partial
    g}{\partial \psi^i}(z_0)\,\frac{\partial(\psi^i\circ F)}{\partial
    \phi^j}\,\left(\phi^j(x)-\phi^j(x_0)\right)+o\left(\xdst x,x_0.\right).
\end{equation}
The following Corollary is a consequence of the chain rule
(\ref{eq:chain_rule}):
\begin{corollary}
  \label{cor:pull_back_spans}
  Let $\left\{(U_\alpha,\phi_\alpha)\right\}_\alpha$ be an atlas for
  $(X,\mu)$ and $\left\{(V_\beta,\psi_\beta)\right\}_\beta$ an atlas
  for $(Z,\nu)$. Then the subbundle $\sbf\mid U_0$ is spanned by the
  sections $\left\{d(\psi^i_\beta\circ F)\right\}_{\beta,i}$.
\end{corollary}
\begin{defn}
  \label{defn:push_fw_pull_bk}
  As the measures $\mpush F.\mu\mrest V_0$ and $\nu\mrest V_0$ are in the same
  measure class, we obtain a pull-back map:
  \begin{equation}
    \label{eq:push_fw_pull_bk1}
    F^*:T^*Z\mid V_0\to T^*X\mid U_0,
  \end{equation}
  which maps each section $dg$ of $T^*Z\mid V_0$ to the section
  $F^*dg=d(g\circ F)$ of $T^*X\mid U_0$. We define the push-forward
  map:
  \begin{equation}
    \label{eq:push_fw_pull_bk2}
    F_*:TX\mid U_0\to TZ\mid V_0
  \end{equation}
  by duality; that is, for $x\in U_0$, $v\in T_xX$ and
  $g\in\lipfun Z.$ we let:
  \begin{equation}
    \label{eq:push_fw_pull_bk3}
    \left\langle F_*(v),dg\mid_{F(x)}\right\rangle = \left\langle v,\left(F^*dg\right)_x\right\rangle.
  \end{equation}
\end{defn}
We conclude this subsection by proving:
\begin{thm}
  \label{thm:isometry_push_fw}
  Let $\sbf^\perp$ denote the annihilator of $\sbf$: i.e.~the fibre
  $\sbf^\perp(x)$ consists of those vectors in $T_xX$ which are
  annihilated by the functionals in $\sbf(x)$. The seminorm
  $\fnrmname$ induces a norm on the quotient bundle $TX/\sbf^\perp$
  which we will still denote by $\fnrmname$. Then $F_*$ induces an
  injective isometry:
  \begin{equation}
    \label{eq:isometry_push_fw}
    F_*:\left(TX/\sbf^\perp\mid U_0,\fnrmname\right)\to(TZ\mid V_0,\tznrmname).
  \end{equation}
\end{thm}
The proof of Theorem \ref{thm:isometry_push_fw} uses the following
generalization of Theorem \ref{thm:Sus_dir_dens}, whose proof is
omitted.
\begin{lemma}
  \label{lem:gen_Sus_dir_dens}
  Suppose that $\gfun$ contains the components of the
  $\{\phi_\alpha\}_\alpha$ and of the $\{\psi_\beta\circ F\}_\beta$;
  suppose that $\gbor$ contains the $\{\chi_{U_\alpha}\}_\alpha$, the
  $\{\chi_{V_\beta}\}_\beta$ and $\chi_{U_0}$; suppose that $\gsem$
    contains $\rdname$. Suppose also that $\gfun'$ contains the
    components of the $\{\psi_\beta\}_\beta$ and that $\gbor'$ contains
    the $\{\chi_{V_\beta}\}$ and $\chi_{V_0}$. Let
    \begin{multline}
      \label{eq:gen_Sus_dir_dens}
      G_x(\gfun,\gbor,\gsem;\gfun',\gsem')=\biggl\{v\in T_xX:
      \text{$v=\gamma'(t)$, where $\gamma'(t)$ is $\gtrip$-generic}\\\quad\text{ and
        $(F\circ \gamma)'(t)$ is $(\gfun',\gbor')$-generic}\biggr\};
    \end{multline}
    then there is a full $\mu$-measure $\mu$-measurable subset $U_1\subset
    U_0$ such that, for each $x\in U_1$,
    $G_x(\gfun,\gbor,\gsem;\gfun',\gsem')$ contains a dense set of directions.
  \end{lemma}
  \begin{proof}
    [Proof of Theorem \ref{thm:isometry_push_fw}]
    We apply Lemma \ref{lem:gen_Sus_dir_dens} and show that for each
    $x\in U_1$ and each $v\in T_xX$ one has:
    \begin{equation}
      \label{eq:isometry_push_fw_p1}
      \fnrm v.=\tznrm F_*(v).;
    \end{equation}
    by density of directions, we just need to show
    (\ref{eq:isometry_push_fw_p1}) for $v=\gamma'(t)$ where
    $\gamma'(t)$ is $\gtrip$-generic and $(F\circ\gamma)'(t)$ is
    $(\gfun',\gbor')$-generic. By Theorem \ref{thm:met_diff_norm}
    applied in $X$ to the pseudometric $\rdname$ we get:
    \begin{equation}
      \label{eq:isometry_push_fw_p2}
      \fnrm\gamma'(t).=\rmd\gamma(t);
    \end{equation}
    note that by the definition of the $\rdname$-metric differential
    we have:
    \begin{equation}
      \label{eq:isometry_push_fw_p3}
      \rmd\gamma(t)=\md F\circ\gamma(t);
    \end{equation}
    finally, applying again Theorem \ref{thm:met_diff_norm}
    in $Z$ to the metric $\zdname$, we get:
    \begin{equation}
      \label{eq:isometry_push_fw_p4}
      \tznrm F_*\gamma'(t).=\tznrm (F\circ\gamma)'(t).=\md
      F\circ\gamma (t).
    \end{equation}
  \end{proof}

  \subsection{Metric differentiation and blow-ups}
  \label{subsec:metdiff_blowups}

  In this subsection we generalize the results of Section
  \ref{sec:geom_blowups} in the case in which one considers either a
  Lipschitz compatible pseudometric $\rdname$ on $X$ or a Lipschitz map
  $F:X\to Z$. \pcreatedst{tr}{\tilde\varrho}
  \def\stdtangs{(Y,\nu,\varphi,\trdname,q)}
  \def\stdpoints{(X,\mu,\psi,\rdname,p)}%
  \begin{defn}
    \label{defn:sem_blowup}
    Let $\rdname$ be a Lipschitz compatible pseudometric on $X$ and
    $(U,\psi)$ be an $N$-dimensional differentiability chart. A
    \textbf{blow-up of $(X,\mu,\psi,\rdname)$ at $p$} along the scales
    $r_n\searrow0$ is a tuple $\stdtangs$ such that:
    \begin{enumerate}
    \item The tuple $\stdtang$ is a blow-up of $(X,\mu,\psi)$ at $p$,
      i.e.~the tuples:
      \begin{equation}
        \label{eq:sem_blowup1}
        \left(\frac{1}{r_n}X,\frac{\mu}{\mu\left(\ball p,r_n.\right)},\frac{\psi-\psi(p)}{r_n},p\right)
      \end{equation}
      converge to $\stdtang$ in the measured Gromov-Hausdorff sense;\vskip2mm
    \item $\trdname$ is a Lipschitz compatible pseudometric on $Y$ and
      if the points $y,y'\in Y$ are represented, respectively, by the
      sequences $[x_n],[x_n']\subset X$, then:
      \begin{equation}
        \label{eq:sem_blowup2}
        \trdst y,y'.=\lim_{n\to\infty}\frac{\rdst x_n,x'_n.}{r_n}.
      \end{equation}
    \end{enumerate}
    We denote by $\tang\stdpoints$ the set of blow-ups of
    $(X,\mu,\psi,\rdname)$ at $p$.
  \end{defn}
  \begin{thm}
    \label{thm:sem_measured_blow_up}
    Let $(U,\psi)$ be an $N$-dimensional differentiability chart for
    the differentiability space $(X,\mu)$, and let $\rdname$ be a
    Lipschitz compatible pseudometric. Then for $\mu\mrest
    U$-a.e.~$p$, for each blow-up
    $\stdtangs\in\tang\stdpoints$, and for each unit vector
    $v_0\in T_pX$, the measure $\nu$ admits an
    Alberti representation $\albrep .=(Q,\Phi)$ where:
    \begin{enumerate}
    \item $Q$ is concentrated on the set $\lines(\varphi,v_0,\trdname)$ of unit
      speed geodesic lines in $Y$ with:
      \begin{equation}
        \label{eq:sem_measured_blow_up_s1}
        \begin{aligned}
          (\varphi\circ\gamma)'&=v_0;\\
          \trdst\gamma(t),\gamma(s).&=\rnrm v_0.|t-s|;
        \end{aligned}
      \end{equation}
    \item For each $\gamma\in\lines(\varphi,v_0,\trdname)$ the measure
      $\Phi_\gamma$ is given by:
      \begin{equation}
        \label{eq:sem_measured_blow_up_s2}
        \Phi_\gamma=\hmeas ._\gamma.
      \end{equation}
    \end{enumerate}
  \end{thm}
  \begin{proof}
    The proof follows the method used to prove Theorem
    \ref{thm:measured_blow_up}; we just:
    \begin{enumerate}
    \item add in condition \textbf{(Reg3)} that:
      \begin{equation}
\label{eq:sem_measured_blow_up_p1}
\left|\rdst\gamma(s_1),\gamma(s_2).-\rnrm\gamma'(t).\,|s_1-s_2|\right|\le\varepsilon|s_1-s_2|;
      \end{equation}
    \item require in Lemma \ref{lem:meas_diff} that $U$ consists of
      points at which the map $x\mapsto\rnrmname(x)$ is approximately continuous.
     \end{enumerate}
   \end{proof}
   We now discuss what happens in the case of a Lipschitz map $F:X\to
   Z$. When we defined blow-ups of the chart functions there was no
   issue with the target space because $\real^N$ possesses a group of
   dilations. For a general map $F:X\to Z$ we first need to use
   ultramits \cite[Sec.~2.4]{kleiner_leeb_rigid_quasi} to blow-up $Z$; we recall
   here the relevant constructions.  \def\stdtangf{(Y,\nu,\varphi,q;G,W,w_0)}
     \def\stdpointf{(X,\mu,\psi,p;F,Z)}%
   \begin{defn}
     \label{defn:ultra_blowups}
     Let $(Z,z_0)$ denote a pointed metric space and let
     $r_n\searrow0$; we define a \textbf{blow-up $(W,w_0)$ of $(Z,z_0)$} along
     the scales $r_n\searrow0$ as an ultralimit of the sequence of
     pointed metric spaces
     $\left(\frac{1}{r_n}Z,z_0\right)$. Specifically, we choose a
     nonprincipal ultrafilter $\omega$ and consider the set $\tilde W$
     of those sequences $[z_n]\subset Z$ such that: \pcreatedst{w}{d_W}
     \pcreatedst{tw}{d_{\tilde W}}
     \begin{equation}
       \label{eq:ultra_blowups1}
       \limsup_{n\to\infty}\frac{\zdst z_n,z_0.}{r_n}<\infty.
     \end{equation}
     We define a pseudometric $\twdname$ on $\tilde W$ by:
     \begin{equation}
       \label{eq:ultra_blowups2}
       \twdst [z_n],[z_n'].=\lim_\omega\frac{\zdst z_n,z_n'.}{r_n}.
     \end{equation}
     On $\tilde W$ we consider the equivalence relation:
     \begin{equation}
       \label{eq:ultra_blowups3}
       [z_n]\sim[z_n']\Longleftrightarrow\twdst [z_n],[z_n'].=0;
     \end{equation}
     then $\twdname$ induces a metric $\wdname$ on the quotient space
     $W=\tilde W/\sim$, and the base point $w_0$ is the equivalence
     class of the constant sequence $[z_0]$. We denote the set of
     blow-ups of $Z$ at $z_0$ by $\tang(Z,z_0)$.
     \par Consider now the case of a Lipschitz map $F:X\to Z$; having
     fixed scales $r_n\searrow0$, we construct blow-ups
     $(Y,q)\in\tang(X,p)$ and $(W,w_0)\in\tang(Z,F(p))$. We then
     obtain a Lipschitz map $G:(Y,q)\to(W,w_0)$ by blowing up the
     graph of $F$ at $(p,F(p))$. Specifically, if $[x_n]\subset X$
     represents the point $y\in Y$, we let $G(y)$ be the equivalence
     class of the sequence $[F(x_n)]$. In general, we say \textbf{that a tuple
     $\stdtangf$ is a blow-up of $(X,\mu,\psi;F,Z)$ at $p$} if:
     $\stdtang\in\tang\stdpoint$, $(W_0,w_0)\in\tang(Z,F(p))$, and $G$
     is obtained by blowing up $F:X\to Z$ at $p$. We denote the set of
     blow-ups of $(X,\mu,\psi;F,Z)$ at $p$ by $\tang\stdpointf$.
   \end{defn}
   Applying Theorem \ref{thm:sem_measured_blow_up} to the pseudometric
   $F^*\zdname$ we get:
   \begin{thm}
    \label{thm:lip_measured_blow_up}
    Let $(U,\psi)$ be an $N$-dimensional differentiability chart for
    the differentiability space $(X,\mu)$, and let $F:X\to Z$ be a
    Lipschitz map. Then for $\mu\mrest
    U$-a.e.~$p$, for each blow-up
    $\stdtangf\in\tang\stdpointf$, and for each unit vector
    $v_0\in T_pX$, the measure $\nu$ admits an
    Alberti representation $\albrep .=(Q,\Phi)$ where:
    \begin{enumerate}
    \item $Q$ is concentrated on the set $\lines(\varphi,v_0,G)$ of unit
      speed geodesic lines in $Y$ with:
      \begin{equation}
        \label{eq:lip_measured_blow_up_s1}
        \begin{aligned}
          (\varphi\circ\gamma)'&=v_0;\\
          \wdst G\circ\gamma(t),G\circ\gamma(s).&=\fnrm v_0.|t-s|;
        \end{aligned}
      \end{equation}
    \item For each $\gamma\in\lines(\varphi,v_0,G)$ the measure
      $\Phi_\gamma$ is given by:
      \begin{equation}
        \label{eq:lip_measured_blow_up_s2}
        \Phi_\gamma=\hmeas ._\gamma.
      \end{equation}
    \end{enumerate}
  \end{thm}
  \begin{remark}
    In \cite[Sec.~10]{cheeger} it was shown that if $(X,\mu)$ is a
    PI-space and if $f$ is a real-valued Lipschitz map defined on $X$,
    at $\mu$-a.e.~$p$, blowing-up $f$ at $p$ always produces a \emph{generalized
      linear function} $g$; in particular, the corresponding space $Y$
    contains through each point a geodesic line $\gamma$ on which the
    blow-up $F$ is affine,
    and such that $\gamma$ behaves as an integral curve of the
    gradient of $F$. Applying Theorem \ref{thm:lip_measured_blow_up}
    to the case in which $F=f$, one gets, through each point of $Y$,
    many geodesic lines on which the blow-up $G$ is affine, and these geodesic
    lines can be used to obtain a Fubini-like decomposition of the
    measure $\nu$. Among these geodesic lines, those where the slope
    of $G$ is maximal correspond to the vector $v_0$ which is the
    derivative of $f$ at $p$ with respect to the coordinate functions $\psi$.
  \end{remark}

  \section{Examples}
  \label{sec:examples}
In this section we provide examples that illustrate how metric differentiation
can be used to constrain the infinitesimal geometry of a Lipschitz map
$F:X\to Y$, where $X$ is a differentiability space. 
We will use a
family of examples of differentiability spaces introduced in
\cite{cheeger_inverse_poinc}: inverse limits of admissible inverse systems 
of metric measure graphs.  
For these spaces  
we find that for some natural
classes of target spaces, blow-ups of arbitrary Lipschitz maps are quite degenerate.

In this section we will say that a map $\eta:\R\ra Z$ into a metric space $Z$ is a {\bf geodesic with speed $\si$} if $d(\eta(s),\eta(t))=\si |s-t|$ for all $s,t\in\R$; here we allow geodesics with speed $0$, i.e. constant maps.

\begin{theorem}
\label{thm_examples_main_theorem}
Let $(X_\infty,d_\infty,\mu_\infty)$ be the inverse limit of an admissible inverse system (see below).  Let $A\subset X_\infty$ be a measurable subset, and $F:A\ra Z$ be a Lipschitz map.  

Consider the following conditions:
\begin{enumerate}
\renewcommand{\labelenumi}{(\alph{enumi})}
\item $Z$ is $CBB$ space, i.e. an Alexandrov space with curvature bounded below.
\item For  every $z\in Z$ and every blow-up $W$ of $Z$ at $z$ 
(in the sense of ultralimits, see Definition \ref{defn:ultra_blowups}), 
two constant speed geodesics $\ga,\ga':\R\ra W$ which coincide on a nonempty open 
interval $(a,b)\subset \R$ coincide everywhere.
\item $Z$ is an equiregular sub-Riemannian manifold. 
\item $Z$ is an Alexandrov space with curvature bounded above, and $X_\infty$
satisfies the monotone bigon condition (Definition \ref{def_monotone_bigon}).
\end{enumerate}

If $Z$ satisfies one of the conditions (a)--(c), then for $\mu_\infty$-a.e. $p\in X_\infty$, 
if $G:Y_\infty\ra W$ is a blow-up of $F$ at $p$, then $G$ factors as
$G=\bar G\circ \varphi$ where $\varphi:Y_\infty\ra\R$ is $1$-Lipschitz and 
$\bar G:\R\ra W$ is a constant speed geodesic.

If $Z$ satisfies (d), then for $\mu_\infty$-a.e. $p\in X_\infty$, 
every blow-up $G:Y_\infty\ra W$ of $F$ at $p$ factors as $\bar G\circ \phi$  where $\phi:Y_\infty\ra Z'$
is $1$-Lipschitz, $\bar G:Z'\ra W$ an isometric embedding, and  $Z'$ is a metric cone over a finite set, 
i.e. the union of finitely many geodesics rays leaving a basepoint.   
\end{theorem}
In fact the argument gives slightly more precise control, see (\ref{eqn_g_factors_through_geodesic}) and (\ref{eqn_bar_g_displayed}) below.  Also, the argument can be generalized somewhat further, see Remark~\ref{rem:generality}.

\begin{remark}
  \label{rem:RNP}
  Note that when the target $Z$ is a Banach space with the
  Radon-Nikodym Property the same conclusion as in the cases (a)--(c)
  follows by differentiating $F$ along the Alberti representations,
  compare the discussion on RNP-differentiability in \cite{cheeger_kleiner_radon}.
\end{remark}

Theorem~\ref{thm_examples_main_theorem} has the following consequence:
\begin{corollary}
  \label{cor:main_exa_thm}
  Under the assumptions of Theorem~\ref{thm_examples_main_theorem}, if
  $F:A\to Z$ is a bi-Lipschitz embedding, then $A$ is $1$-rectifiable.
\end{corollary}
The proof of the corollary is given after the proof of Theorem~\ref{thm_examples_main_theorem}.

\begin{defn}[Admissible inverse systems, \cite{cheeger_inverse_poinc}]
  \label{defn:adm_inv_sys}
  We consider an inverse system of metric measure graphs:
  \begin{equation}
    \label{eq:adm_inv_sys1}
    \cdots\xleftarrow{\pi_{i-1}} X_i\xleftarrow{\pi_i}X_{i+1}\xleftarrow{\pi_{i+1}}\cdots,
  \end{equation}
  where the index $i$ can range either over $\zahlen$ or over
  $\natural\cup\{0\}$: in the former case we will say that the inverse
  system is \textbf{signed}, and in the latter case that it is
  \textbf{unsigned}.  We denote the metric and measure on $X_i$ by $d_i$ and $\mu_i$ respectively.
  Having fixed an integer $m\geq2$ and parameters
  $\Delta,C,\theta\in(0,\infty)$, we say that the inverse system
  $\left\{X_i,\pi_i\right\}$ is \textbf{admissible} if it satisfies
  the following axioms:
  \begin{description}
  \item[(Ad1)] Each metric space $(X_i,d_i)$ is a nonempty connected
    graph with vertices of valence $\leq\Delta$ and such that each
    edge of $X_i$ is isometric to an interval of length $m^{-i}$ with
    respect to the path metric $d_i$;    \vskip2mm
  \item[(Ad2)] Let $X_i'$ denote the graph obtained by subdividing
    each edge of $X_i$ into $m$ edges of length $m^{-(i+1)}$. Then
    $\pi_i$ induces a map $\pi_i:(X_{i+1},d_{i+1})\to(X_i',d_i)$ which
    is open, simplicial and an isometry on every edge;\vskip2mm
  \item[(Ad3)] For each $x_i\in X_i'$ the inverse image
    $\pi_i^{-1}(x_i)\subset X_{i+1}$ has $d_{i+1}$-diameter at most $\theta\,m^{-(i+1)}$;\vskip2mm
  \item[(Ad4)] Each graph $X_i$ is equipped with a measure $\mu_i$
    which restricts to a multiple of arclength on each edge; if
    $e_1,e_2$ are two adjacent edges of $X_i$ we have:
    \begin{equation}
      \label{eq:adm_inv_sys2}
      \frac{\mu_i(e_1)}{\mu_i(e_2)}\in[C^{-1},C];
    \end{equation}
  \item[(Ad5)] The measures $\{\mu_i\}$ are compatible with the
    projections $\{\pi_i\}$: $\mpush\pi_i.\mu_{i+1}=\mu_i$;\vskip2mm
  \item[(Ad6)] Let $\graphstar(x,G)$ denote the star of a vertex $x$ in a
    graph $G$, i.e.~the union of all the edges containing $x$. Then,
    for each vertex $v_i'\in X_i'$ and each
    $v_{i+1}\in\pi_i^{-1}(v_i')$, the quantity:
    \begin{equation}
      \label{eq:adm_inv_sys3}
      \frac{\mu_{i+1}\left(\pi_i^{-1}(e_i')\cap\graphstar(v_{i+1},X_{i+1})\right)}{\mu_i(e'_i)}
    \end{equation}
    is the same for all edges $e'_i\in\graphstar(v_i',X_i')$;\vskip2mm
  \item[(Ad7)] If the inverse system $\left\{X_i,\pi_i\right\}$ is
    unsigned we will assume that $X_0\simeq[0,1]$,
    $\mu_0=\lebmeas\mrest[0,1]$ and we will denote by $\varphi_i$ the map:
    \begin{equation}
      \label{eq:adm_inv_sys4}
      \varphi_i=\pi_1\circ\cdots\circ\pi_{i-1}.
    \end{equation}
    If the inverse system $\left\{X_i,\pi_i\right\}$ is signed we
    require the existence of open surjective maps
    $\varphi_i:X_i\to\real$ which are, regarding $\real$ as a graph of
    edges $\left\{[km^{-i},(k+1)m^{-i}]\right\}_{k\in\zahlen}$,
    simplicical and restrict to isometries on every edge. Moreover, we
    require that the $\{\varphi_i\}$ are compatible with the
    $\{\pi_i\}$:
    \begin{equation}
      \label{eq:adm_inv_sys5}
      \varphi_i\circ\pi_i=\varphi_{i+1}\quad(\forall i).
    \end{equation}
  \end{description}
\end{defn}
An immediate consequence of the axioms \textbf{(Ad1)}--\textbf{(Ad7)} is that the metric measure
spaces $(X_i,d_i,\mu_i)$ converge in the measured Gromov-Hausdorff
sense\footnote{If $\left\{X_i,\pi_i\right\}$ is signed we consider the
  convergence in the pointed sense by choosing basepoints
  $\{q_i\}_{i\in\zahlen}$ satisfying $\pi_i(q_{i+1})=q_i$.} to a metric
measure space $(X_\infty,d_\infty,\mu_\infty)$ which is called
\textbf{the inverse limit} of the admissible inverse system. If
$\left\{X_i,\pi_i\right\}$ is unsigned, then $(X_\infty,d_\infty)$ is
compact geodesic and $\mu_\infty$ is a doubling probability measure;
if $\left\{X_i,\pi_i\right\}$ is signed, $(X_\infty,d_\infty)$ is
proper geodesic and $\mu_\infty$ is a doubling measure. In both cases
there are $1$-Lipschitz maps $\pi_{\infty,k}:X_\infty\to X_k$
satisfying:
\begin{equation}
  \label{eq:invlimmaps}
  \begin{aligned}
    \pi_{k-1}\circ\pi_{\infty,k}&=\pi_{\infty,k-1}\\
    \mpush\pi_{\infty,k}.\mu_\infty&=\mu_k.
  \end{aligned}
\end{equation}
For $j>k$ we will use the short-hand notation $\pi_{j,k}$ to denote
the map $\pi_k\circ\cdots\circ\pi_{j-1}$. Moreover, the maps
$\varphi_i:X_i\to\text{$\real$ or $[0,1]$}$ pass to the limit giving a
$1$-Lipschitz map $\varphi_\infty:X_\infty\to\text{$\real$ or
  $[0,1]$}$ satisfying:
\begin{equation}
  \label{eq:invlimcoord}
  \varphi_\infty(q)=\varphi_i(\pi_{\infty,i}(q))\quad\text{($\forall
    q\in X_\infty$, $\forall i\in\text{$\zahlen$ or $\natural\cup\{0\}$}$)}
\end{equation}

We now define a special class of paths in $X_i$ or $X_\infty$.
\begin{defn}
  \label{defn:monpaths}
  Let $I\subseteq\real$ be connected and $\gamma:I\to X_i$ continuous,
  where we allow $i=\infty$. We say that $\gamma$ is a \textbf{monotone
  geodesic} if $\varphi_i\circ\gamma:I\to\text{$\real$ or
  $[0,1]$}$ is either a strictly increasing or decreasing affine
map. In particular, the axioms \textbf{(Ad1)}--\textbf{(Ad7)} imply
that $\gamma$ is a constant speed geodesic in $(X_i,d_i)$. Moreover,
by axioms \textbf{(Ad2)} and \textbf{(Ad7)}, if $j>i$ and if
$\gamma_i:I\to X_i$ is a monotone geodesic, then for each
$q_j\in\pi_{j,i}^{-1}\left(\gamma_i(I)\right)$, one can lift $\gamma_i$
to obtain a monotone geodesic $\gamma_j:I\to X_j$ passing through
$q_j$ and satisfying $\pi_{j,i}\circ\gamma_j=\gamma_i$.
\end{defn}
We now summarize some important consequences of the axioms
\textbf{(Ad1)}--\textbf{(Ad7)}.
\begin{thm}
  \label{thm:inverselimit_poinc}
  Let $\left\{X_i,\pi_i\right\}$ be an admissible inverse system and
  let $X_\infty$ denote the inverse limit; then:
  \begin{enumerate}
  \item The metric measure space $(X_\infty,d_\infty,\mu_\infty)$
    admits a $(1,1)$-Poincar\'e inequality; in particular, it is a
    differentiability space with a single differentiability chart $(X_\infty,\varphi_\infty)$;\vskip2mm
  \item  The $(X_i,\mu_i)$'s and $(X_\infty,\mu_\infty)$ admit distinguished Alberti representations whose support is precisely the set of all monotone geodesics
$\ga:I\ra X_\infty$ where $\varphi_\infty\circ\ga=\id_I$, and $I=[0,1]$ if $\{X_i,\pi_i\}$ is unsigned, and $I=\R$ if it is signed.  These Alberti representations are compatible with projection.
  \end{enumerate}
\end{thm}
The proof of Theorem \ref{thm:inverselimit_poinc} is contained in
\cite{cheeger_inverse_poinc}; note, however, that in \cite{cheeger_inverse_poinc} only the case of
what we call unsigned inverse systems is discussed: the modifications
for the case of signed inverse systems are straightforward. Alberti
representations are not explicitly mentioned in
\cite{cheeger_inverse_poinc}, but part (2) in Theorem
\ref{thm:inverselimit_poinc} follows from the discussion in 
 \cite[Sec.~6]{cheeger_inverse_poinc}.

\begin{remark}
In \cite{andrea_mutually_singular} it is shown that the collection of admissible inverse systems defined in \cite{cheeger_inverse_poinc} contains an uncountable family of the form $\{\{(X_i,d_i,\mu_{i,\al})\}_i\}_{\al\in \mathcal{A}}$, such that inverse limits
$\{(X_\infty,d_\infty,\mu_{\infty,\al})\}_{\al\in \mathcal{A}}$ --- which 
are all PI spaces \cite{cheeger_inverse_poinc} --- realize an uncountable family $\{\mu_{\infty,\al}\}_{\al\in \mathcal{A}}$ of mutually singular measures on the same inverse limit metric space $(X_\infty,d_\infty)$.  
\end{remark}

\def\stdtangpoint{(X_\infty,\mu_\infty,\psi,p)}
\def\stdtang{(\sigma Y_\infty,c\cdot\nu_\infty,\sigma\cdot\varphi,q)}
\def\stdinv{(Y_\infty,d_\infty,\nu_\infty)}
\def\stdfunpoint{(X_\infty,\mu_\infty,\psi,p;F,Z)}
\def\stdfun{(\sigma
  Y_\infty,c\cdot\nu_\infty,\sigma\cdot\varphi,q;G,W,w_0)}
\def\tbase{\partial_\psi\mid_p}

\begin{thm}
  \label{thm:invlim_bwup}
  Let $X_\infty$ be the inverse limit of an admissible inverse system
  $\left\{X_i,\pi_i\right\}$ and let $\psi=\varphi_\infty$; if the
  system is unsigned assume also that
  $p\not\in\psi^{-1}\left(\{0,1\}\right)$. Then each element of
  $\tang\stdtangpoint$ is of the form $\stdtang$ where:
  \begin{enumerate}
  \item The metric measure space $(Y_\infty,d_\infty,\nu_\infty)$ is
    the inverse limit of a signed admissible inverse system
    $\{Y_i,\pi_i\}$, and $\varphi$ is the function $\varphi_\infty$
    corresponding to $Y_\infty$.\vskip2mm
  \item The parameters $\sigma$ anc $c$ satisfy:
    \begin{equation}
      \label{eq:invlim_bwup_s1}
      \begin{aligned}
        \sigma&\in[1,m]\\
        c&=\frac{1}{\nu_\infty\left(\sball Y_\infty, q,1/\sigma.\right)}\,.
      \end{aligned}
    \end{equation}
  \item The basepoint $q$ satisfies $\varphi(q)=0$.
  \item Furthermore, up to renormalization, the blow-up of the distinguished Alberti representation of $(X_\infty,\mu_\infty)$ is the distinguished Alberti reprsentation on $(Y_\infty,\nu_\infty)$. 
  \end{enumerate}
\end{thm}

We may now apply Theorem \ref{thm:lip_measured_blow_up} to obtain the following:
\begin{thm}
  \label{thm:inv_measured_blow_up}
  Let $X_\infty$ be the inverse limit of an admissible inverse system
  $\left\{X_i,\pi_i\right\}$. Let $A\subset X_\infty$ be a 
  measurable subset, and $F:X_\infty\supset A\to Z$ be Lipschitz. 
  
  Then
  there is a full $\mu_\infty$-measure subset $S_F\subset A$ such that, for
  each $p\in S_F$ and
  each $\stdfun\in\tang\stdfunpoint$, one has that $G$ maps each unit-speed
  monotone geodesic line $\gamma:\real\to Y_\infty$ to a (possibly degenerate\footnote{This happens iff
    $\fnrm\tbase.=0$, i.e.~when $G\circ\gamma$ is constant.}) geodesic line in $W$ with
  constant speed $\sigma^{-1}\fnrm\tbase.$.
\end{thm}

\begin{definition}
\label{def_monotone_bigon}
An admissible inverse system $\{(X_i,\pi_i)\}$ satisfies the {\bf monotone bigon condition} if there is a constant $D$ such that for every $i$, if  $y_1,y_2\in X_\infty$  project under $\pi_i:X_\infty\ra X_i$ to the same point, there are monotone geodesic segments $\ga_1,\ga_2\subset X_\infty$ of length $<Dm^{-i}$ such that $y_i\in \ga_i$, and $\ga_1$, $\ga_2$ have the same endpoints.
\end{definition}

One may readily check that some standard examples of admissible systems, for instance Examples 1.2 and 1.4 from \cite{cheeger_kleiner_realization}, satisfy the monotone bigon condition.

\bigskip
\begin{proof}[Proof of Theorem \ref{thm_examples_main_theorem}]
We apply Theorem \ref{thm:inv_measured_blow_up}, and 
will show that under each of the assumptions 
(a)--(d), there is a full measure subset of points $p\in A$ such that if  $G:Y_\infty\ra W$ is as Theorem \ref{thm:inv_measured_blow_up}, then:

\begin{itemize}
\item If one of (a)--(c) holds then $G$ factors as 
\begin{equation}
\label{eqn_g_factors_through_geodesic}
G\stackrel{\varphi_\infty}{\lra}\R\stackrel{\bar G}{\lra}W
\end{equation}
where $\bar G$ is a geodesic with constant speed $\si_0=\sigma^{-1}\fnrm\tbase.$.
Here we allow $\si_0=0$, in which case $\bar G$ and $G$ are constant maps.
\item If (d) holds, then $G$ factors as $\bar G\circ \pi^\infty_{-\infty}$ where
$\pi^\infty_{-\infty}:Y_\infty\ra Y_{-\infty}$ is the projection map from the inverse limit to the direct limit, and 
\begin{equation}
\label{eqn_bar_g_displayed}
\bar G:Y_{-\infty}\ra W
\end{equation}
is a map whose restriction of any monotone geodesic in $Y_{-\infty}$
is a geodesic in $W$ with constant speed $\si_0=\sigma^{-1}\fnrm\tbase.$.
\end{itemize}

We first assume that (b) holds.  This also covers case (a) since (a) $\implies$ (b).

 
Let $\Ga$ be the set of monotone geodesics $\gamma: \R\ra Y_\infty$ such that $\varphi_\infty\circ\gamma=\id_\R$.  
We define an equivalence relation on $\Ga$ by saying that  
$\gamma_1\sim\gamma_2$ if there is a geodesic $\eta:\R\ra W$ with
constant speed $\sigma_0$ such that 
$G \circ \gamma_i=\eta\circ\varphi_\infty\circ \gamma_i$ for $i=\{1,2\}$.  

Note that if the images of $\gamma_1,\gamma_2\in\Ga$ intersect in an interval, then they are equivalent
by  assumption (b). 
By  concatenating rays to form monotone geodesics, it follows that if
the images of $\gamma_1$ and $\gamma_2$ intersect even in a single
point, they are equivalent.  Now define an equivalence relation on
$Y_\infty$ by saying that $y_1,y_2\in Y_\infty$ are equivalent if some
(or equivalently every) $\ga_1\in\Ga$ passing through $y_1$ is
equivalent to some (or every) geodesic $\ga_2\in\Ga$ passing through
$y_2$.  The cosets of this relation on $Y_\infty$ are closed, and any path $\gamma$ that is a concatenation of finitely many segments from monotone geodesics lies in a single coset. Therefore there is only one coset and $G$ factors as claimed.

Suppose (c) holds.  Since the conclusion is local, we may assume without loss of generality that there is a smooth map $\Psi:Z\ra \R^k$, where $k$ is the dimension of the horizontal space, and the derivative $D\Psi$ restricts to an isomorphism on every horizontal space.  
Hence by \cite{mitchell_tang}, 
for every $z\in Z$, the blow-up of $Z$ at $z$ is a Carnot group $W$, and the blow-up  $\hat\Psi:W\ra \R^k$ of $\Psi$ at $z$ yields the horizontal coordinate for $W$.  

We now proceed as before, except that we take $p$ to be a point of  differentiability of the composition $\Psi\circ F:A\ra \R^k$.  Passing to the blow-up $G:Y_\infty\ra W$, from differentiability we get that $\hat\Psi\circ G=\al\circ\varphi_\infty$  where $\al:\R\ra\R^k$ is
an affine map.  It follows that for every monotone geodesic $\ga\in \Ga$,
the composition $G\circ\ga$ projects under the horizontal coordinate 
$\hat\Psi:W\ra \R^k$ to the same constant speed geodesic $\al\circ\varphi_\infty\circ\ga$.  This implies that for every $\ga\in \Ga$, the composition $G\circ \ga$
is an integral curve of the left invariant vector field determined by $\al$.
In particular, if two such geodesics agree at a point, then they coincide.

Now consider the equivalence relation on $\Ga$ defined as before.
If  $\ga_1,\ga_2\in \Ga$
agree at some $t\in \R$, then by the above discussion $G\circ \ga_1=G\circ \ga_2$, i.e.
$\ga_1\sim\ga_2$.  The rest of the argument is the same.

Now assume (d) holds.   Again consider the map $G:Y_\infty\ra W$
obtained by applying metric differentiation at a point $p\in
X_\infty$.  As above, for every $\ga\in \Ga$, the composition $G\circ \ga$ is a geodesic of constant speed $\si_0$.  One checks readily that $Y_\infty$ inherits the monotone bigon condition. 
Pick $y_1,y_2\in Y_\infty$, and suppose that $\pi_j(y_1)=\pi_j(y_2)$. By the monotone bigon condition, there exist $\ga_i\in\Ga$ such that $\ga_i$ passes through $y_i$ and the maps $\ga_1,\ga_2:\R\ra Y_\infty$ agree outside a a compact subset of $\R$.  Then $G\circ\ga_1$ and $G\circ\ga_2$ are constant speed geodesics in a $CAT(0)$ space, and they agree outside a compact subset of $\R$.  It follows that $\ga_1=\ga_2$, so in particular $G(y_1)=G(y_2)$.    We have thus shown that $G$ factors through a map $G_j:Y_j\ra W$ which has the property that its restriction to any monotone geodesic in $Y_j$ is a geodesic in $W$ with speed $\si_0$.  This implies that $G$ factors as $\bar G\circ \pi^\infty_{-\infty}$, where 
$\pi^\infty_{-\infty}:Y_\infty\ra Y_{-\infty}$ is the projection map
to the direct limit  $Y_{-\infty}$ of the signed inverse system
$\{(Y_i,\pi_i)\}$, and 
$\bar G:Y_{-\infty}\ra W$ is a map whose
restriction to any monotone geodesic in $Y_{-\infty}$
is a geodesic in $W$ with constant speed $\sigma_0$.
Finally note that as there is a uniform bound on the cardinality of
the links of the vertices of $Y_j$ as $j\to -\infty$, then
$Y_{-\infty}$ is a finite union of geodesic rays issuing from a basepoint.
\end{proof}

\bigskip
\begin{remark}
  \label{rem:generality}
  Note that the previous argument for points (a)--(c) can be run more
  generally under the assumptions that $G$ maps monotone geodesics to constant speed geodesics   in $W$ which cannot branch, i.e. two such geodesics must coincide if they agree on an open interval. Note that it suffices to verify this non-branching property for the geodesics that arise from blow-up;  one may restrict these geodesics by exploiting differentiability of auxiliary Lipschitz functions, such as the function $\Psi$ appearing in case (c).
\end{remark}

\bigskip
\begin{proof}[Proof of Corollary \ref{cor:main_exa_thm}]

Let $F:A\ra Z$ be an $L$-bilipschitz embedding.

We first assume that we are in one of the cases (a)-(c).  Therefore we know that there
is a measurable subset $A_1\subset A$ with $\mu(A\setminus A_1)=0$
such that  for all  $x\in A_1$, for every blow-up $G:Y_\infty\ra W$ as in (the proof of)
Theorem \ref{thm_examples_main_theorem}, we have that $G=\eta\circ\varphi_\infty$
where $\eta:\R\ra W$ is a geodesic with constant speed lying in the interval
$[L^{-1},L]$.   Since $G$ is also $L$-bilipschitz, we conclude that for every
$y_1,y_2\in Y_\infty$, we have
\begin{equation}
\label{eqn_one_over_ell_squared}
d(\varphi_\infty(y_1),\varphi_\infty(y_2))\geq L^{-2}d(y_1,y_2)\,.
\end{equation}
This implies that for all $x\in A_1$ there is an $r(x)>0$ such that 
\begin{equation}
\label{eqn_one_over_two_ell_squared}
d(\varphi_\infty(x'),\varphi_\infty(x))> \frac12 L^{-2}d(x',x)
\end{equation}
 for all $x'\in B(x,r(x))$; otherwise there would be a sequence $x'_k\ra x$ violating
(\ref{eqn_one_over_two_ell_squared}), and by rescaling and passing to a limit, we get a blow-up contradicting (\ref{eqn_one_over_ell_squared}).  Now put
$$
S_j=\{x\in A_1\mid r(x)>j^{-1}\}\,.
$$
Then $\varphi_\infty\restr_{S_j}:S_j\ra \R$ is $2L^2$-bilipschitz on $\frac{r(x)}{2}$-balls, and $A_1=\cup_jS_j$.  This shows that $A_1$ is rectifiable.

Now suppose we are in case (d).  

Let 
$$
A_0=\{x\in A\mid \pi_j(x)\;\text{is not a vertex for any $j$}\}\subset A\,.
$$ 
Note that $A\setminus A_0$ is $\mu$-null.  Then there is a full measure subset $A_1\subset A_0$ such that for every $x\in A_1$,  every blow-up $G:Y_\infty\ra W$ 
of $F$ at $x$ has a factorization $G=\bar G\circ\pi^\infty_{-\infty}$ as in the proof of Theorem \ref{thm_examples_main_theorem}, and reasoning as above we get that 
\begin{equation}
\label{eqn_compression_infty_minus_infty}
d(\pi^\infty_{-\infty}(y_1),\pi^\infty_{-\infty}(y_2))\geq L^{-2}d(y_1,y_2)\,
\end{equation}
for all $y_1,y_2\in Y_\infty$.

Now for   $x\in A_1$, using (\ref{eqn_compression_infty_minus_infty}) and a contradiction argument, we get that for every $C$ there is a $j_x$ such that for all $j\geq j_x$,
the projection  $\pi_j:X_{j+1}\ra X_j$ maps the ball $B(\pi_{j+1}(x),Cm^{-j})\subset X_{j+1}$ bijectively onto $B(\pi_j(x),Cm^{-j})\subset X_j$.  Iterating this, it follows that for all $k\geq j \geq j_x$, the projection $\pi^k_j:X_k\ra X_j$ maps $B(\pi_k(x),Cm^{-k})$ bijectively onto $B(\pi_j(x),Cm^{-k})$.  Since $\pi_j(x)$ is not a vertex, for large enough $k$ we conclude that $B(\pi_k(x),Cm^{-k})$ has no branch points, i.e. it is isometric to an interval.  As $C$ is arbitrary, this together with (\ref{eqn_compression_infty_minus_infty}) implies that for every blow-up $G:Y_\infty\ra W$ of $F$ at $x$, the direct limit $Y_{-\infty}$ has no branch points.  Hence we are in the same situation as cases (a)-(c), and we may complete the proof as before.

\end{proof}

\bibliography{metric_diff}

\bibliographystyle{alpha}

\end{document}